\documentclass[12pt]{amsart}
\usepackage{amsmath,amssymb,amsfonts,amsthm,amscd,textcomp,times,booktabs,mathrsfs}
\usepackage{mathtools}
\usepackage[dvipdfmx]{graphicx}
\usepackage{braket}
\usepackage{color,float}
\usepackage{bm}
\usepackage[normalem]{ulem}
\usepackage{enumerate}
\usepackage[all,color,matrix,arrow,pdf]{xy}
\makeatletter
\newcommand*{\rom}[1]{\ensuremath{\mathrm{\expandafter\@slowromancap\romannumeral #1@}}}
\makeatother
\newtheorem{theorem}{Theorem}

\newtheorem{claim}[theorem]{Claim}
\newtheorem{proposition}[theorem]{Proposition}
\newtheorem{lemma}[theorem]{Lemma}

\theoremstyle{definition}
\newtheorem{definition}[theorem]{Definition}
\newtheorem{remark}[theorem]{Remark}
\newtheorem{example}[theorem]{Example}
\newtheorem{step}{Step}
\numberwithin{equation}{section}
\numberwithin{theorem}{section}
\numberwithin{table}{section}

\renewcommand{\theenumi}{\roman{enumi}}
\renewcommand{\labelenumi}{\rm{(\theenumi)}}

\newcommand{\N}{\ensuremath{\mathbb{N}}}
\newcommand{\Z}{\ensuremath{\mathbb{Z}}}

\newcommand{\R}{\ensuremath{\mathbb{R}}}
\newcommand{\C}{\ensuremath{\mathbb{C}}}
\newcommand{\I}{\ensuremath{\sqrt{-1}}}

\newcommand{\GL}{\ensuremath{\mathrm{GL}}}
\newcommand{\SL}{\ensuremath{\mathrm{SL}}}

\newcommand{\Spin}{\ensuremath{\mathrm{Spin(7)}}}
\newcommand{\SU}{\ensuremath{\mathrm{SU}}}
\newcommand{\U}{\ensuremath{\mathrm{U}}}

\newcommand{\del}{\ensuremath{\mathrm{\partial}}}
\newcommand{\delbar}{\ensuremath{\mathrm{\overline{\partial}}}}
\newcommand{\der}{\ensuremath{\mathrm{d}}}

\newcommand{\id}{\ensuremath{\mathrm{id}}}
\newcommand{\din}{\rotatebox{90}{\ensuremath{\in}}}

\newcommand{\dsubset}{\rotatebox{90}{\ensuremath{\subset}}}

\newcommand{\commute}{\rotatebox{180}{\ensuremath{\circlearrowright}}}

\newcommand{\reduce}[1]{\scalebox{1}{\ensuremath{#1}}}

\newcommand{\norm}[1]{\ensuremath{\left| #1 \right|}}
\newcommand{\Norm}[1]{\ensuremath{\left\| #1 \right\|}}
\newcommand{\restrict}[2]{\ensuremath{\left. #1 \right|_{#2}}}

\DeclareMathOperator{\Image}{Im}
\DeclareMathOperator{\Real}{Re}
\DeclareMathOperator{\Ker}{Ker}

\DeclareMathOperator{\sgn}{sgn}

\DeclareMathOperator*{\smalloplus}{\reduce{\bigoplus}}
\DeclareMathOperator*{\smallotimes}{\reduce{\bigotimes}}
\DeclareMathOperator*{\smallsum}{\reduce{\sum}}

\DeclareMathOperator*{\smallcup}{\reduce{\bigcup}}

\DeclareMathOperator*{\residue}{res}

\begin{document}
\title[Global smoothings of SNC complex surfaces with trivial canonical bundle]
{Differential geometric global smoothings\\
of simple normal crossing complex surfaces\\
with trivial canonical bundle}

\author{Mamoru Doi}
\email{doi.mamoru@gmail.com}

\author{Naoto Yotsutani}
\address{Kagawa University, Faculty of education, Mathematics, Saiwaicho $1$-$1$, Takamatsu, Kagawa, $760$-$8522$, Japan}
\email{yotsutani.naoto@kagawa-u.ac.jp}

\makeatletter
\@namedef{subjclassname@2020}{%
  \textup{2020} Mathematics Subject Classification}
\makeatother


\maketitle

\noindent{\bfseries Abstract.}
Let $X$ be a simple normal crossing (SNC) compact complex surface with trivial canonical bundle which includes triple intersections.
We prove that if $X$ is $d$-semistable, then there exists a family of smoothings in a differential geometric sense.
This can be interpreted as a differential geometric analogue of the smoothability results due to
Friedman, Kawamata-Namikawa, Felten-Filip-Ruddat, Chan-Leung-Ma, and others in algebraic geometry.
The proof is based on an explicit construction of local smoothings around the singular locus of $X$,
and the first author's existence result of holomorphic volume forms on global smoothings of $X$.
In particular, these volume forms are given as solutions of a nonlinear elliptic partial differential equation.

As an application, we provide several examples of $d$-semistable SNC  
complex surfaces with trivial canonical bundle including double curves,
which are smoothable to complex tori, primary Kodaira surfaces and $K3$ surfaces.
We also provide several examples of such complex surfaces including triple points, which are smoothable to $K3$ surfaces.

\section{Introduction}\label{sec:Intro}
This is a sequel to the first author's paper \cite{Doi09}, where he obtained a construction of compact complex surface with trivial canonical bundle
by gluing together compact complex surfaces with an anticanonical divisor.
As an application, he proved that if a simple normal crossing (SNC) complex surface $X$ is $d$-semistable and has at most double intersections,
then there exists a family of smoothings of $X$ in a differential geometric sense.
More precisely, there exist a smooth $6$-manifold $\mathcal{X}$ and a smooth surjective map $\varpi :\mathcal{X}\to\Delta\subset\C$
such that $\varpi^{-1}(0)=X$, $\varpi^{-1}(\zeta)$ for each $\zeta\in\Delta\setminus\{ 0\}$
is a smooth compact complex surface with trivial canonical bundle,
and the complex structure on $\varpi^{-1}(\zeta )$ depends \emph{continuously} on $\zeta$ outside the singular locus of $X$.
One purpose of this article 
is to extend the smoothability result in \cite{Doi09} to cover the cases where $X$ has triple intersections.
For another purpose, we provide many examples of SNC 
complex surfaces with trivial canonical bundle
including double and triple intersections, to which we can apply our smoothability result.

Throughout this article, 
$X=\smallcup_{i=1}^N X_i$ (or possibly $X=\smallcup_{i=0}^{N-1}X_i$)
denotes a compact connected complex surface with normal crossings of $N$ irreducible components
with $\dim_\C X_i=2$ for each $i$, unless otherwise specified.
Before stating the main result, we give some definitions.

\begin{definition}\label{def:SNC}
Let $X$ be a compact complex analytic surface with $N$ irreducible components $X_1,\dots ,X_N$.
Then we say that $X$ is a \emph{simple normal crossing (SNC) complex surface} if each irreducible component $X_i$ is smooth,
and $X$ is locally embedded in $\C^3$ as $\set{(\zeta^1,\zeta^2,\zeta^3)\in\C^3|\zeta^1\dots\zeta^\ell=0}$ for some $\ell\in\{ 1,2,3\}$.
\end{definition}

In particular, the above $X$ has no fourfold intersections: $Q_{ijk\ell}=X_i\cap X_j\cap X_k\cap X_\ell=\emptyset$ for all $i,j,k,\ell$.
Let $D_{ij}=X_i\cap X_j$ be a set of \textit{double curves} and $T_{ijk}=X_i\cap X_j\cap X_k$ be a set of \textit{triple points}.
We define index sets $I_i$ and $I_{ij}\subset I_i$ by
\begin{align*}
I_i&=\set{j\in\{ 1,\dots ,N\}|X_i\cap X_j\neq\emptyset}\quad\text{and}\\
I_{ij}&=\set{k\in\{ 1,\dots ,N\}|X_i\cap X_j\cap X_k\neq\emptyset}.
\end{align*}
Then $D_i=\sum_{j\in I_i}D_{ij}$ and $T_{ij}=\sum_{k\in I_{ij}}T_{ijk}$ are divisors on $X_i$ and $D_{ij}$ respectively,
and $T_i=\sum_{j\in I_i}T_{ij}$ is the set of triple points on $X_i$.
Note that we admit \emph{self-intersections}, that is, it may happen that $X_i\cap X_i\neq\emptyset$.
An example of an SNC complex surface with a self-intersection will be given in Example $\ref{ex:tori_Kodaira}$ for $N=1$.

\begin{definition}\label{definition:d-s.s}
Let $X=\smallcup_{i=1}^N X_i$ be an SNC complex surface.
Then $X$ is $d$-{\it semistable} if for each double curve $D_{ij}$ we have
\begin{equation}\label{def:d-s.s2}
N_{ij}\otimes N_{ji}\otimes [T_{ij}]\cong\mathcal{O}_{D_{ij}},
\end{equation}
where $N_{ij}$ denotes the holomorphic normal bundle $N_{D_{ij}/X_i}$ to $D_{ij}$ in $X_i$,
and $[T_{ij}]$ denotes the associated holomorphic line bundle of $T_{ij}$ (see, e.g., \cite{GH94}, p. $145$ and p. $134$ for the respective definitions).
\end{definition}
This is a differential geometric interpretation of Friedman's original complex analytic (and also algebro-geometric) definition that
a SNC complex surface $X$ is said to be \emph{$d$-semistable} if
\begin{equation}\label{def:d-s.s}
\left(\smallotimes_i \mathcal{I}_{X_i}/\mathcal{I}_{X_i}\mathcal{I}_{D}\right)^{\!\! *}\cong \mathcal{O}_D
\end{equation}
for the singular locus $D$ on $X$, where $\mathcal{I}_{X_i}$ and $\mathcal{I}_{D}$ are the ideal sheaves of $X_i$ and $D$ in $X$ respectively.

Now the main theorem of this article 
is described as follows.
\begin{theorem}\label{thm:smoothing}
Let $X=\smallcup_{i=1}^N X_i$ be an SNC 
complex surface.
Assume the following conditions:
\renewcommand{\theenumi}{\roman{enumi}}
\begin{enumerate}
\item $X$ is $d$-semistable;
\item each $D_i=\sum_{j\in I_i}D_{ij}$ is an anticanonical divisor on $X_i$; and
\item there exists a meromorphic volume form $\Omega_i$ on each $X_i$ with a single pole along $D_i$
such that the Poincar\'{e} residue $\residue_{D_{ij}}\Omega_i$ of $\Omega_{i}$ on $D_{ij}$
is minus the Poincar\'{e} residue $\residue_{D_{ij}}\Omega_j$ of $\Omega_{j}$ on $D_{ij}$ for all $i,j$.
(For the definition of Poincar\'{e} residues, see {\rm\cite{GH94}, pp. }$147$--$148$).
\end{enumerate}
Then there exist $\epsilon>0$ and a surjective map $\varpi :\mathcal{X}\to\Delta =\set{\zeta\in\C|\norm{\zeta}<\epsilon}$
such that the following statements hold.
\renewcommand{\theenumi}{\alph{enumi}}
\begin{enumerate}
\item $\mathcal{X}$ is a smooth real $6$-dimensional manifold and $\varpi$ is a smooth map.
\item $X_0=\varpi^{-1}(0)=X$.
\item For each $\zeta\in\Delta^{\! *}=\Delta\setminus\{ 0\}$, $X_\zeta=\varpi^{-1}(\zeta)$
is smooth and carries a complex structure which makes $X_\zeta$ a compact complex surface with trivial canonical bundle.
\item The complex structure on $X_\zeta$ depends continuously on $\zeta$ outside the singular locus $D=\smallcup_{i=1}^ND_i\subset X_0$.
More precisely, for any point $p\in\mathcal{X}\setminus D$
there exist a neighborhood $U$ of $p$ and a diffeomorphism $U\simeq V\times W$ with $W\subset\Delta$,
such that the induced complex structures on $V$ depend continuously on $\zeta\in W$.
\end{enumerate}
\end{theorem}
\begin{remark}\label{rem:smoothing}
Conditions (ii) and (iii) of Theorem $\ref{thm:smoothing}$ are equivalent to the condition
that the canonical bundle $K_X$ of the SNC complex surface is trivial (see Section $\ref{subsec:K_X_SNC}$).
Also, if
\begin{itemize}
\item the normal crossing complex surface $X=\smallcup_iX_i$ is not simple, or
\item we are given $\{ X_i,D_i,\Omega_i\}_i$ satisfying conditions (i)--(iii) and gluing isomorphisms between double curves,
but not all local embeddings of $X_i$ into $\C^3$,
\end{itemize}
we obtain an alternative SNC complex surface $X'$ as follows:
If we glue together the irreducible components $X_i$ along the double curves using the isomorphisms and local embeddings into $\C^3$,
we obtain a desired SNC complex surface $X'$ to which we can apply Theorem $\ref{thm:smoothing}$.
\end{remark}
Let us compare Theorem $\ref{thm:smoothing}$ with Friedman's smoothability result.
According to \cite{Ku77}, Theorem \rom{2},
the central fiber of a semistable degeneration of $K3$ surfaces is classified into either Type \rom{1}, \rom{2}, or \rom{3},
where the central fiber corresponding to a degeneration of Type $\nu$ contains a $\nu$-ple intersection but no $(\nu +1)$-ple intersection.
Conversely, Friedman proved in \cite{Fr83}, Theorem $5.10$ that if $X$ is a $d$-semistable $K3$ surface
(see Section $\ref{subsec:Friedman}$ for the definition),
then there exists a family of smoothings $\varpi :\mathcal{X}\to\Delta\subset\C$ of $X$ with $K_{\mathcal{X}}=\mathcal{O}_{\mathcal{X}}$,
where $\mathcal{X}$ is a $3$-dimensional complex manifold, $\Delta$ is a domain in $\C$, and $\varpi$ is a holomorphic map.
See Section $\ref{subsec:Friedman}$ for more details.
On the other hand, our smoothability result holds even when $H^1(X,\mathcal{O}_X)\neq 0$ or not all irreducible components of $X$ are K\"{a}hlerian.
In exchange for a broader scope of application, our smoothings $\varpi :\mathcal{X}\to\Delta$ are not holomorphic but only smooth,
that is, both $\varpi$ and $\mathcal{X}$ are smooth, although each fiber admits a complex structure which depends continuously on $\zeta\in\Delta$.
In Section $\ref{subsec:ex_SNC_double}$, we give some examples of SNC complex surfaces $X$
with trivial canonical bundle including double curves and satisfying $H^1(X,\mathcal{O}_X)\neq 0$,
which are smoothable to complex tori and primary Kodaira surfaces, in addition to some examples of $d$-semistable $K3$ surfaces of Type \rom{2}.
Also, in Sections $\ref{sec:TypeIIIK3}$ and $\ref{sec:MatchingProb}$,
we construct some explicit examples of $d$-semistable K3 surfaces of Type \rom{3} which are smoothable to $K3$ surfaces.
However, it remains an interesting problem whether there exists an example of a $d$-semistable SNC complex surface with trivial canonical bundle
including \emph{triple points} which is smoothable to a complex torus or a primary Kodaira surface.

From the modern viewpoint of logarithmic geometry which is a central tool for smoothings, Kawamata and Namikawa generalized Friedman's smoothability result 
in higher complex dimensions $n\geqslant 3$ \cite{KN94}. In the paper, they required that $H^{n-1}(X,\mathcal{O}_X)=0$, $H^{n-2}(X_i,\mathcal{O}_{X_i})=0$ for all $i$, 
and all irreducible components $X_i$ are K\"{a}hlerian for proving the existence of smoothings of a compact K\"ahler normal crossing variety $X$.
Note that the first requirement $H^{n-1}(X,\mathcal O_X)=0$ comes from the use of the $T^1$-lifting property.
Kawamata-Namikawa's smoothability result is particularly effective in constructing many examples of Calabi-Yau manifolds.
Indeed, they obtained new examples of Calabi-Yau threefolds
by smoothing $d$-semistable SNC Calabi-Yau threefolds $X=X_1\cup X_2$ with a double intersection $D=X_1\cap X_2$.
Also, N.-H. Lee \cite{Lee19} used \cite{KN94} to obtain further new examples of Calabi-Yau threefolds
by smoothing $d$-semistable SNC Calabi-Yau threefolds of Type III (i.e., with triple intersections).
Meanwhile, Hashimoto and Sano \cite{HS19} noticed that in \cite{KN94} it is not necessary for $X$ as a whole to be K\"{a}hlerian
although its irreducible components have to be (see \cite{HS19}, Remark $2.7$),
and constructed an infinite number of examples of non-K\"{a}hler threefolds with trivial canonical bundle
by smoothing $d$-semistable SNC non-K\"{a}hler threefolds with trivial canonical bundle including a double intersection.

A major breakthrough in algebraic extensions of Friedman's smoothability result has been achieved by Felten, Filip and Ruddat in \cite{FFR19}.
Conceptually, $\mathcal{T}^1_X={\mathrm{Ext}}_{\mathcal{O}_X}^1(\Omega_X^1, \mathcal{O}_X)$ measures the failure of the normal sequence and the notion of $d$-semistability defined in 
$\eqref{definition:d-s.s}$ is then given by $\mathcal{T}_X^1\cong \mathcal O_D$.
It was proved in \cite{FFR19}, Theorem $1.1$ that an SNC variety with trivial canonical bundle is smoothable if $\mathcal{T}_X^1$ is generated by global sections and the singular locus of $X$ is projective.
Moreover, their smoothability theorem holds even when not all of the irreducible components $X_i$ of an SNC variety $X$ are K\"ahlerian, or when not all of the cohomology groups of $X$ and $X_i$ vanish.
We also mention that Chan, Leung, and Ma proved that the existence of smoothings of $d$-semistable log smooth Calabi-Yau varieties \cite{CLM19}.
Keeping these modern logarithmic viewpoints in mind \cite{RS20}, 
one can interpret Theorem $\ref{thm:smoothing}$ as a differential geometric counterpart of algebraic extensions of Friedman's smoothability
result (e.g., Theorem $1.1$ in \cite{FFR19}).

Meanwhile, our result of differential geometric smoothings using the gluing technique brings some insight into both differential and algebraic geometry.
As mentioned above, the smoothing technique in complex analytic and algebraic geometry is particularly effective
for constructing special smooth \emph{complex} manifolds such as Calabi-Yau manifolds.
For this purpose, our construction of differential geometric smoothings is also useful
because we need a complex structure not on the total space $\mathcal{X}$
of the smoothings $\varpi :\mathcal{X}\to\Delta$ but only on a smooth fiber $X_\zeta =\varpi^{-1}(\zeta )$.
Meanwhile in differential geometry, the gluing technique is effectively used for constructing many compact manifolds with a special geometric structure
such as Calabi-Yau, $G_2$- and $\Spin$- structures and complex structures with trivial canonical bundle
by Joyce \cite{Joyce}, Kovalev \cite{Kov03}, Clancy \cite{Cl11}, and Doi and Yotsutani \cite{DY14, DY15, DY19, Doi09}.
Our result here not only enables us to reconstruct such a manifold as a smooth fiber of global differential geometric smoothings of an SNC manifold
with a special geometric structure including only \emph{double} intersections,
but also opens up the possibilities of treating global smoothings of such SNC manifolds
admitting \emph{triple} intersections to obtain new examples.

The proof of Theorem $\ref{thm:smoothing}$ is based on an explicit construction of global smoothings of $X$ by gluing together local smoothings
around double curves and triple points.
For this purpose, we note that the holomorphic coordinates on a neighborhood of $D_{ij}$ in $X_i$ are approximated
by those on a neighborhood of $D_{ij}$ in $N_{ij}$ with a Taylor expansion in terms of the fiber coordinate of $N_{ij}$ via an exponential map.
Then we realize differential geometric local smoothings $\varpi_{ij}:\mathcal{V}_{ij}\to\Delta\subset\C$
of $X_i\cup X_j$ around each double curve $D_{ij}$ as a complex hypersurface of $N_{ij}\oplus N_{ji}$.
In particular, we see from the simplicity of the normal crossing complex surface $X=\smallcup_i X_i$
that the space $\mathcal{V}_{ij}$ of local smoothings around each triple point in $D_{ij}$ includes a local model written as
\begin{equation*}
\Set{(u^1,u^2,u^3)\in\C^3|\norm{u^1}<1,\norm{u^2}<1,\norm{u^3}<1,u^1u^2u^3\in\Delta}.
\end{equation*}
By gluing together $(X_i\setminus D_i)\times\Delta$ and $\mathcal{V}_{ij}$ for all $i,j$,
we obtain differential geometric global smoothings $\varpi :\mathcal{V}\to\Delta$.
For each $\zeta\in\Delta$, we can consistently define an $\SL (2,\C )$-structure $\widetilde{\Omega}_\zeta$ on the fiber $\varpi^{-1}(\zeta )$
using condition (iii) of Theorem $\ref{thm:smoothing}$ such that we have $\der\widetilde{\Omega}_\zeta\to 0$ as $\zeta\to 0$ in an appropriate sense.
Finally, we prove that if $\norm{\zeta}$ is sufficiently small, then we can deform $\widetilde{\Omega}_\zeta$ to a $\der$-closed $\SL (2,\C )$-structure $\Omega_\zeta$
using the main result of \cite{Doi09}, so that $\varpi^{-1}(\zeta )$ is a compact complex surface with trivial canonical bundle.

This article 
is organized as follows.
In Section $\ref{sec:SL2Csurf}$, we briefly state the results in \cite{Doi09} which will be used in the proof of Theorem $\ref{thm:smoothing}$.
We introduce the notions of $\SL(2,\C)$- and $\SU(2)$-structures in Section $\ref{subsec:geom_SL2C}$,
state the existence theorem of complex structures with trivial canonical bundle in Section $\ref{subsec:exist_SL2C}$,
define the canonical bundles of SNC complex surfaces in Section $\ref{subsec:K_X_SNC}$,
and review the semistable degenerations of $K3$ surfaces in Section $\ref{subsec:Friedman}$.
Before constructing explicit local smoothings in Sections $\ref{subsec:smoothing_double}$ and $\ref{subsec:smoothing_triple}$,
in Section $\ref{subsec:coordinates}$ we will introduce local holomorphic coordinates suited to
the smoothing problem, and give an example which provides a local model around a triple point.
The proof of Theorem \ref{thm:smoothing} is given in Section $\ref{subsec:existence}$.
Also, in Section $\ref{subsec:ex_SNC_double}$
we give several examples of $d$-semistable SNC complex surfaces with trivial canonical bundle including only double curves,
which are smoothable to complex tori, primary Kodaira surfaces, and $K3$ surfaces.
In the last two sections, we construct examples of SNC complex surfaces with triple points which are smoothable to $K3$ surfaces.
In Section $\ref{subsec:bl-up_SNC}$, we see that the blow-up of an SNC complex surface with trivial canonical bundle at finite points
in the double curves excluding the triple points inherits good properties from the original one.
Then in Section $\ref{subsec:ex_4triple}$, we produce examples of $d$-semistable $K3$ surfaces with four points in Example $\ref{ex:4triple}$.
Section $\ref{sec:MatchingProb}$ is devoted to considering a more technical example.
After fixing our notation in Section $\ref{subsec:Notation}$, we consider in Section $\ref{subsec:MP}$
the \emph{mismatch problem} which we encounter when we try to glue all components together along their intersections.
In order to handle this kind of mismatch issue, we shall take the order of blow-ups carefully.
Consequently, we will show that one can still glue all components together after taking appropriate blow-ups in Section $\ref{subsec:dssSNC}$.
 
The first author is mainly responsible for Sections $\ref{sec:Intro}$, $\ref{subsec:geom_SL2C}$--$\ref{subsec:K_X_SNC}$, and $\ref{sec:Smoothing}$,
and the second author mainly for Sections $\ref{sec:Intro}$, $\ref{subsec:Friedman}$, $\ref{sec:TypeIIIK3}$, and $\ref{sec:MatchingProb}$.\\

\noindent{\bfseries Acknowledgements.} 
The authors would like to thank Professor Kento Fujita for allowing us to use his example which is the source of Example $\ref{ex:Fujita}$.
Naoto Yotsutani also thank to Professors Nam-Hoon Lee, Taro Sano, and Yuji Odaka for fruitful discussions through e-mails.
Finally, we are grateful to the referee for valuable comments which improved the presentation of our manuscript.
This work was partially supported by JSPS KAKENHI Grant Number $18$K$13406$
and Young Scientists Fund of Kagawa University Research Promotion Program $2021$ (KURPP).

\section{A brief review of complex surfaces with trivial canonical bundle}\label{sec:SL2Csurf}
In dealing with complex surfaces with trivial canonical bundles in differential geometry,
it is crucial to note that a complex structure of such a surface is characterized by a $\der$-closed $\SL (2,\C )$-structure,
which becomes a holomorphic volume form with respect to the resulting complex structure (see Proposition $\ref{prop:SL2C_cpxstr}$).
Then with the help of a Hermitian form which forms an $\SU (2)$-structure together with an $\SL (2,\C )$-structure,
we can reduce the problem whether a given $\SL (2,\C )$-structure $\psi$ with small $\der\psi$ can be deformed into a $\der$-closed $\SL (2,\C)$-structure,
to the solvability of a partial differential equation given by $\eqref{eq:PDE}$.
The first two subsections provide more details.
We introduce the notions of $\SL (2,\C)$- and $\SU (2)$-structures in Section $\ref{subsec:geom_SL2C}$,
and state in Section $\ref{subsec:exist_SL2C}$ the existence result of an $\der$-closed $\SL (2,\C )$ structure
as a solution of the above differential equation $\eqref{eq:PDE}$ below.
Meanwhile, Section $\ref{subsec:K_X_SNC}$ describes the canonical bundle of an SNC complex surface according to \cite{Fr83},
and Section $\ref{subsec:Friedman}$ reviews the classification of semistable degenerations of $K3$ surfaces and smoothability result of $d$-semistable $K3$ surfaces from the algebro-geometric viewpoint.

\subsection{$\SL (2,\C)$- and $\SU (2)$-structures}\label{subsec:geom_SL2C}
In this subsection, we briefly review the notions and results in \cite{Doi09} without proofs.
(See also \cite{Goto04} for reference.)
\begin{definition}
Let $V$ be an oriented vector space of dimension $4$.
Then $\psi_0\in\wedge^2 V^*\otimes\C$ is an \emph{$\SL (2,\C )$-structure} on $V$ if $\psi_0$ satisfies
\begin{equation*}
\psi_0\wedge\overline{\psi}_0>0\quad\text{and}\quad\psi_0\wedge\psi_0 =0.
\end{equation*}
\end{definition}

An $\SL (2,\C )$-structure $\psi_0$ on $V$ gives a decomposition of $V\otimes\C$:
\begin{gather*}
V\otimes\C =V^{1,0}\oplus V^{0,1},\quad\text{where}\\
V^{0,1}=\set{\zeta\in V\otimes\C|\iota_\zeta\psi_0=0},\quad V^{1,0}=\overline{V^{0,1}},
\end{gather*}
and $\iota_\zeta$ is the interior multiplication by $\zeta$.
Thus if $v\in V$, then $v$ is uniquely written as $v=v^{1,0}+\overline{v^{1,0}}$, and $v\mapsto v^{1,0}$ gives an isomorphism between real vector spaces.
Then the composition
\begin{equation*}
I_{\psi_0}:V\xrightarrow{\cong}V^{1,0}\xrightarrow{\sqrt{-1}}V^{1,0}\xrightarrow{\cong}V
\end{equation*}
satisfies $I_{\psi_0}^2=-\id_V$.
Thus $\psi_0$ defines a complex structure $I_{\psi_0}$ on $V$ such that $\psi_0$ is a complex differential form of type $(2,0)$ with respect to $I_{\psi_0}$.

Let $\mathcal{A}_{\SL (2,\C )}(V)$ be the set of $\SL (2,\C )$-structures on $V$.
Then $\mathcal{A}_{\SL (2,\C )}(V)$ is an orbit space under the action of the orientation-preserving general linear group $\GL_+(V)$.
Since each $\psi\in\mathcal{A}_{\SL (2,\C )}(V)$ has isotropy group $\SL (2,\C )$,
there is a one-to-one correspondence from the orbit $\mathcal{A}_{\SL (2,\C )}(V)$ to the homogeneous space $\GL_+(V)/\SL (2,\C )$.

\begin{definition}
Let $M$ be an oriented $4$-manifold.
Then $\psi\in C^\infty(\wedge^2 T^*\! M\otimes\C)$ is an \emph{$\SL (2,\C )$-structure} on $M$ if $\psi$ satisfies
\begin{equation*}
\psi\wedge\overline{\psi}>0\quad\text{and}\quad\psi\wedge\psi =0.
\end{equation*}
\end{definition}

We define $\mathcal{A}_{\SL (2,\C )}(M)$ to be the fiber bundle which has fiber $\mathcal{A}_{\SL (2,\C )}(T_x M)$ over $x\in M$.
Then an $\SL (2,\C )$-structure can be regarded as a smooth section of $\mathcal{A}_{\SL (2,\C )}(M)$.

Since an $\SL (2,\C )$-structure $\psi$ on $M$ induces an $\SL (2,\C )$-structure on each tangent space,
$\psi$ defines an almost complex structure $I_\psi$ on $M$ such that $\psi$ is a $(2,0)$-form with respect to $I_\psi$.
\begin{lemma}[Grauert, Goto \cite{Goto04}]
Let $\psi$ be an $\SL (2,\C )$-structure on an oriented $4$-manifold $M$.
If $\psi$ is $\der$-closed, then $I_\psi$ is an integrable complex structure on $M$ with trivial canonical bundle
and $\psi$ is a holomorphic volume form on M with respect to $I_\psi$.
\end{lemma}
The above lemma gives the following characterization of complex surfaces with trivial canonical bundle by $\der$-closed $\SL (2,\C )$-structures.
\begin{proposition}\label{prop:SL2C_cpxstr}
Let $M$ be an oriented $4$-manifold.
Then $M$ admits a complex structure with trivial canonical bundle if and only if $M$ admits a $\der$-closed $\SL (2,\C )$-structure.
\end{proposition}
Thus, if we say that $X$ be a complex surface with canonical trivial bundle, then we understand that
$X$ consists of an underlying oriented $4$-manifold $M$ and a $\der$-closed $\SL (2,\C )$-structure $\psi$ on $M$ such that
$\psi$ induces a complex structure $I_\psi$ on $M$ and becomes a holomorphic volume form on $X=(M,I_\psi )$.

Let $X$ be a compact complex surface with trivial canonical bundle.
If $X$ is simply connected or $H^1(X,\mathcal{O}_X)=0$, then $X$ is called a $K3$ surface.
According to the Enriques-Kodaira classification, it is known that a compact complex surface with trivial canonical bundle is either
a complex torus, a primary Kodaira surface, or a $K3$ surface (see \cite{BHPV}, Chapter \rom{6}).

\renewcommand{\labelenumi}{\rm(\roman{enumi})}
\begin{definition}\label{def:SU(2)_V}
Let $V$ be an oriented vector space of dimension $4$.
Then $(\psi_0,\kappa_0)\in (\wedge^2 V^*\otimes\C)\oplus\wedge^2 V^*$ is an \emph{$\SU (2)$-structure} on $V$
if $(\psi_0,\kappa_0)$ satisfies the following conditions:
\begin{enumerate}
\item $\psi_0$ is an $\SL (2,\C )$-structure on $V$,
\item $\psi_0\wedge\kappa_0=0$,
\item a bilinear form $g_{(\psi_0,\kappa_0)}$ on $V$ defined by
$g_{(\psi_0,\kappa_0)}(I_{\psi_0}\cdot, \cdot )=\kappa_0(\cdot, \cdot )$ is positive definite, and
\item $2\,\kappa_0^2=\psi_0\wedge\overline{\psi}_0$.
\end{enumerate}
\end{definition}

Conditions (ii) and (iii) imply that $\kappa_0$ is a (1,1)-form associated with the inner product $g_{(\psi_0,\kappa_0)}$ on $V$
which is Hermitian in the sense that $g_{(\psi_0,\kappa_0)}(I_{\psi_0}\cdot ,I_{\psi_0}\cdot )=g_{(\psi_0,\kappa_0)}(\cdot ,\cdot )$.
Let $\mathcal{A}_{\SU (2)}(V)$ be the set of $\SU (2)$-structures on the oriented vector space $V$.
Then $\mathcal{A}_{\SU (2)}(V)$ is an orbit space
under the action of $\GL_+(V)$, from which there is a one-to-one correspondence to $\GL_+(V)/\SU (2)$.

For $(\psi_0,\kappa_0)\in\mathcal{A}_{\SU (2)}(V)$, we have the orthogonal decomposition with respect to $g_{(\psi_0,\kappa_0)}$
\begin{equation*}
\wedge^2 V^*=\wedge^2_+\oplus\wedge^2_-,
\end{equation*}
where $\wedge^2_+$ and $\wedge^2_-$ are the set of self-dual and anti-self-dual 2-forms
with respect to the Hodge dual operator $*_{g_{(\psi_0,\kappa_0)}}:\wedge^2 V^*\to\wedge^2 V^*$, respectively.
Then $\wedge^2_+$ is spanned by $\{\Real\psi_0,\Image\psi_0,\kappa_0\}$, where $\Real\psi_0$ and $\Image\psi_0$ are the real and imaginary part of
$\psi_0$, and $\wedge^2_-$ coincides with the set of primitive real $(1,1)$-forms with respect to $\kappa_0$.

We also have the orthogonal decomposition
\begin{align*}
(\wedge^2 V^*\otimes\C )\oplus\wedge^2 V^*&\cong T_{(\psi_0,\kappa_0)}((\wedge^2 V^*\otimes\C )\oplus\wedge^2 V^*)\\
&=T_{(\psi_0,\kappa_0)}\mathcal{A}_{\SU (2)}(V)\oplus T_{(\psi_0,\kappa_0)}^\perp\mathcal{A}_{\SU (2)}(V),
\end{align*}
where $T_{(\psi_0,\kappa_0)}^\perp\mathcal{A}_{\SU (2)}(V)$ is the orthogonal complement to $T_{(\psi_0,\kappa_0)}\mathcal{A}_{\SU (2)}(V)$
with respect to $g_{(\psi_0,\kappa_0)}$.
The next lemma is crucial in solving the partial differential equation in the proof of Theorem \ref{thm:existence}.
\begin{lemma}\label{thm:ASD}
The tangent space of $\mathcal{A}_{\SU (2)}$ at $(\psi_0,\kappa_0)$ contains anti-self-dual subspaces:
\begin{equation*}
(\wedge^2_-\otimes\C )\oplus\wedge^2_-\subset T_{(\psi_0,\kappa_0)}\mathcal{A}_{\SU (2)}(V).
\end{equation*}
\end{lemma}
\begin{proof}
See \cite{Doi09}, Lemma $2.6$.
\end{proof}

\begin{definition}\label{thm:rho}
Let $V$ be an oriented vector space of dimension $4$.
We define a neighborhood of $\mathcal{A}_{\SU (2)}(V)$ in $(\wedge^2 V^*\otimes\C)\oplus\wedge^2 V^*$ by
\begin{equation*}
\mathcal{T}_{\SU (2)}(V)=\Set{(\psi_0+\alpha, \kappa_0+\beta )|\substack{\begin{aligned}
&(\psi_0, \kappa_0)\in\mathcal{A}_{\SU (2)}(V),\text{ and}\\
&(\alpha,\beta )\in T_{(\psi_0,\kappa_0)}^\perp\mathcal{A}_{\SU (2)}(V)\\
&\text{with }\norm{(\alpha,\beta )}_{g_{(\psi_0,\kappa_0)}}<\rho\end{aligned}}},
\end{equation*}
where $\rho$ is a positive constant and
$\norm{(\alpha,\beta)}_{g_{(\psi_0,\kappa_0)}}=\norm{\alpha}_{g_{(\psi_0,\kappa_0)}}+\norm{\beta}_{g_{(\psi_0,\kappa_0)}}$.
In other words, $\mathcal{T}_{\SU (2)}(V)$ is defined to be the $\rho$-neighborhood of $\mathcal{A}_{\SU (2)}(V)$ in $(\wedge^2 V^*\otimes\C)\oplus\wedge^2 V^*$.
\end{definition}

\begin{lemma}\label{lem:rho*}
There exists a positive constant $\rho_*$ such that if $\rho <\rho_*$ then any $(\psi',\kappa')\in\mathcal{T}_{\SU (2)}(V)$ can be uniquely written as
$(\psi_0+\alpha,\kappa_0+\beta)$, where $(\psi_0,\kappa_0)\in\mathcal{A}_{\SU (2)}(V)$ and
$(\alpha,\beta )\in T_{(\psi_0,\kappa_0)}^\perp\mathcal{A}_{\SU (2)}(V)$.
\end{lemma}
\begin{proof}
Let $T^\perp\!\mathcal{A}_{\SU (2)}(V)$ be the vector bundle whose fiber over $(\psi_0,\kappa_0)$ is $T_{(\psi_0,\kappa_0)}^\perp\mathcal{A}_{\SU (2)}(V)$.
Let us define a map $f:T^\perp\mathcal{A}_{\SU (2)}(V)\to(\wedge^2 V^*\otimes\C )\oplus\wedge^2 V^*$ by
\begin{gather*}
f:((\psi_0,\kappa_0),(\alpha ,\beta ))\mapsto (\psi_0+\alpha ,\kappa_0+\beta )\\
\text{for }(\psi_0,\kappa_0)\in\mathcal{A}_{\SU (2)}(V)\text{ and }(\alpha ,\beta )\in T_{(\psi_0,\kappa_0)}^\perp\mathcal{A}_{\SU (2)}(V).
\end{gather*}
Then by the $\epsilon$-neighborhood theorem in \cite{GP}, Chapter $2$, Section $3$,
there exists a positive function $\epsilon$ on $\mathcal{A}_{\SU (2)}(V)$ such that 
$f$ maps the $\epsilon$-neighborhood $\mathcal{A}^\epsilon_{\SU (2)}(V)$ of $\mathcal{A}_{\SU (2)}(V)$ in $T^\perp\!\mathcal{A}_{\SU (2)}(V)$
diffeomorphically onto its image in $(\wedge^2 V^*\otimes\C )\oplus\wedge^2 V^*$,
and $f^{-1}$ gives a unique orthogonal decomposition of $(\psi',\kappa')\in f(\mathcal{A}^\epsilon_{\SU (2)}(V))$.
Since $a\in\GL_+(V)$ induces an isometry from each neighborhood of $(\psi_0,\kappa_0)\in\mathcal{A}_{\SU (2)}(V)$ onto a neighborhood of
$(a^*\psi_0,a^*\kappa_0)$,
we can take $\epsilon$ to be a constant on $\mathcal{A}_{\SU (2)}(V)$.
Hence, the assertion holds if we take $\mathcal{T}_{\SU (2)}(V)=f(\mathcal{A}^\epsilon_{\SU (2)}(V))$ and $\rho^*=\epsilon$.
\end{proof}
Lemma \ref{lem:rho*} implies that for $\rho <\rho_*$ the projection $\Theta :\mathcal{T}_{\SU (2)}(V)\to\mathcal{A}_{\SU (2)}(V)$ is well-defined.
We also denote by $\Theta_1$ the composition of
$\Theta$ and the projection $\mathcal{A}_{\SU (2)}(V)\to\mathcal{A}_{\SL (2,\C )}(V)$.

\begin{definition}\label{def:SU(2)_M}
Let $M$ be an oriented $4$-manifold.
Then $(\psi ,\kappa)\in C^\infty(\wedge^2 T^*\! M\otimes\C)\oplus C^\infty(\wedge^2 T^*\! M)$ is an \emph{$\SU (2)$-structure} on $M$
if the restriction $(\restrict{\psi}{x},\restrict{\kappa}{x})$ is an $\SU (2)$-structure on $T_x M$ for all $x\in M$.
\end{definition}

Define $\mathcal{A}_{\SU (2)}(M)$ to be the fiber bundle whose fiber over $x\in M$ is $\mathcal{A}_{\SU (2)}(T_x M)$.
Then an $\SU (2)$-structure can be regarded as a smooth section of $\mathcal{A}_{\SU (2)}(M)$.

If $\psi$ and $\kappa$ are both $\der$-closed, then $X=(M,I_\psi ,\kappa)$ is a K\"{a}hler surface with trivial canonical bundle.
Moreover, the Ricci curvature of the K\"{a}hler metric $g$ vanishes due to condition (iv) of Definition \ref{def:SU(2)_V}.

\begin{definition}\label{def:Theta}
Let $M$ be an oriented $4$-manifold.
Choose $\rho <\rho_*$ so that the projection $\Theta$ is well-defined.
We define $\mathcal{T}_{\SU (2)}(M)$ to be the fiber bundle whose fiber over $x\in M$ is $\mathcal{T}_{\SU (2)}(T_x M)$,
and denote by $\Theta$ the projection from $\mathcal{T}_{\SU (2)}(M)$ to $\mathcal{A}_{\SU (2)}(M)$.
\end{definition}

\subsection{Existence theorem of $\der$-closed $\SL (2,\C )$-structures}\label{subsec:exist_SL2C}
We are now in a position to state the following existence theorem of a complex structure with trivial canonical bundle.
\renewcommand{\labelenumi}{\rm(\roman{enumi})}
\begin{theorem}[\cite{Doi09}, Theorems $4.1$ and $4.2$]\label{thm:existence}
Let $M$ be a smooth $4$-manifold and $\lambda ,\mu$, and $\nu$ positive constants.
Then there exists a positive constant $\epsilon_*$ such that whenever $0<\epsilon <\epsilon_*$, the following is true.

Let $(\psi ,\kappa )$ be an $\SU (2)$-structure on $M$ such that the associated metric $g$ is a complete Riemannian metric on $M$.
Suppose that the injectivity radius $\delta(g)$ and Riemann curvature $R(g)$ satisfy
$\delta(g)\geqslant\mu\epsilon$ and $\Norm{R(g)}\leqslant\nu\epsilon^{-2}$,
and that $\phi$ is a smooth complex $2$-form on $M$ with $\der\psi +\der\phi =0$, and
\begin{equation*}
\Norm{\phi}_{L^2}\leqslant\lambda\,\epsilon^3,\quad\Norm{\der\phi}_{L^8}\leqslant\lambda\,\epsilon ,\quad\text{and}\quad
\Norm{\der\kappa}_{L^8}\leqslant\lambda\,\epsilon^{-1/2}.
\end{equation*}
Let $\rho$ be as in Definition $\ref{thm:rho}$.
Then there exists $\eta\in C^\infty(\wedge^2_- T^*\! M\otimes\C )$ with $\Norm{\eta}_{C^0}<\rho$ such that $\der\Theta_1(\psi +\eta,\kappa) =0$.
Hence, the manifold $M$ admits a complex structure with trivial canonical bundle.
\end{theorem}

Note that the Hermitian form $\kappa$ in Theorem $\ref{thm:existence}$, which forms an $\SU (2)$-structure together with $\psi$,
only plays an auxiliary role for obtaining a $\der$-closed $\SL (2,\C )$-structure.
Since we only require a mild estimate for $\kappa$, it is not difficult to find such a $\kappa$.

The proof of Theorem $\ref{thm:existence}$ is summarized as follows.
The equation $\der\Theta_1(\psi +\eta,\kappa) =0$ is rewritten as
\begin{equation}\label{eq:PDE}
\der\eta =\der\phi +\der F(\eta )
\end{equation}
when $\eta\in C^\infty(\wedge^2_- T^*\! M\otimes\C )$, where the term $F(\eta )$ is defined by
\begin{equation*}
\Theta_1 (\psi +\eta ,\kappa )=\psi +\eta +F(\eta ).
\end{equation*}
We note that $F(\eta )$ is quadratic in $\eta\in C^\infty(\wedge^2_- T^*\! M\otimes\C )$ due to Lemma \ref{thm:ASD}.
To solve $\eqref{eq:PDE}$ we consider the recurrence equations
\begin{equation*}
\der\eta_j =\der\phi +\der F(\eta_{j-1})
\end{equation*}
with $j>0$ and $\eta_0=0$.
According to the Hodge theory, there exists a unique $\eta_j\in C^\infty(\wedge^2_- T^*\! M\otimes\C )$ for each $j>0$
such that $\eta_j-\phi-F(\eta_{j-1})$ is $L^2$-orthogonal to $\mathcal{H}^2_-$.
Then one can show that the sequence $\{\eta_j\}$ converges to a unique $\eta$ in the Sobolev space $L^8_1(\wedge^2_- T^*\! M\otimes\C )$.
The hypothesis on the injectivity radius and Riemann curvature in Theorem \ref{thm:existence} is a technical assumption to evaluate
$\Norm{\nabla\chi}_{L^8}$ for $\chi\in C^\infty (\wedge^2_-T^*\! M)$ in terms of $\Norm{\der\chi}_{L^8}$ and $\Norm{\chi}_{L^2}$,
and then $\Norm{\chi}_{C^0}$ in terms of $\Norm{\nabla\chi}_{L^8}$ and $\Norm{\chi}_{L^2}$.
Regularity of $\eta$ follows from the ellipticity of $\eqref{eq:PDE}$ when it is considered as an equation on $L^8(V)$ with $V=\smalloplus_{i=0}^4 T^*\! M$.
Then using the Sobolev embedding $L^8_1\hookrightarrow C^{0,1/2}$ and the standard bootstrapping method, we prove that $\eta$ is smooth.
For further details, see \cite{Doi09}, Section $4$ (see also \cite{Joyce}, Chapter $13$).

\subsection{Canonical bundle of an SNC complex surface}\label{subsec:K_X_SNC}
According to \cite{Fr83}, Remark $2.11$, the canonical bundle $K_X$ of an SNC complex surface $X$ is described as follows.
\begin{definition}\label{def:K_X_SNC}
Let $X=\smallcup_{i=1}^NX_i$ be an SNC complex surface with irreducible components $X_i$,
given by gluing isomorphisms $f_{ij}:D_{ij}\to D_{ji}$ for all $i,j$ with $i\leqslant j$ and $j\in I_i$,
where we distinguish $D_{ij}$ and $D_{ji}$ by regarding $D_{ij}\subset X_i$ and $D_{ji}\subset X_j$.
Also, we understand that if $D_{ij}$ has more than one irreducible component, then $f_{ij}$ is a union of the corresponding isomorphisms,
and if $D_{ii}\neq\emptyset$, then we divide the irreducible components of $D_{ii}$ into two as $D_{ii}=D'_{ii}\cup D''_{ii}$
and consider $f_{ii}$ as an isomorphism from $D'_{ii}$ to $D''_{ii}$.
Define line bundles $L_i$ on $X_i$ and $L_{ij}$ on $D_{ij}$ by
\begin{align*}
L_i=\restrict{K_X}{X_i}=K_{X_i}\otimes [D_i]\quad\text{and}\quad L_{ij}=\restrict{L_i}{D_{ij}}.
\end{align*}
By the adjunction formula, we calculate $L_{ij}$ as
\begin{align*}
L_{ij}&=\restrict{L_i}{D_{ij}}=\restrict{K_{X_i}}{D_{ij}}\otimes\restrict{[D_i]}{D_{ij}}\\
&=\restrict{\left(K_{X_i}\otimes [D_{ij}]\right)}{D_{ij}}
\otimes\restrict{\left[ \smallsum_{k\in I_{ij}}D_{ik}\right]}{D_{ij}}\cong K_{D_{ij}}\otimes [T_{ij}].
\end{align*}
Also, the restriction of $L_i$ to $D_{ij}$ is given by the Poincar\'{e} residue map.
Then the canonical bundle of $X$ is given by the collection of the line bundles $L_i$ on $X_i$,
together with the gluing isomorphisms $-f_{ij}^*:L_{ji}\to L_{ij}$,
that is, a set $\{ s_i\}_i$ of local sections $s_i\in H^0(X_i,L_i)$ together define a global section $s\in H^0(X,K_X)$ if and only if
\begin{equation*}
\residue_{D_{ij}}s_i=-f_{ij}^*\left(\residue_{D_{ji}}s_j\right)\quad\text{for all }i,j\text{ with }i\leqslant j,j\in I_i.
\end{equation*}
\end{definition}

The minus sign in the gluing isomorphisms $-f_{ij}^*:L_{ji}\to L_{ij}$ naturally arises when we consider a local model as follows.
Consider a local embedding $\{\zeta^1\zeta^2=0\}$ (resp. $\{\zeta^1\zeta^2\zeta^3=0\}$) of $X_1$ and $X_2$ around $p\in D_{12}\setminus T_{12}$
(resp. $X_1,X_2$, and $X_3$ around $p\in T_{123}$) into $\C^3$ with local representations $X_i=\{\zeta^i=0\}$.
Let $\Omega_0$ be a meromorphic volume form on $\C^3$ given by
\begin{equation*}
\Omega_0 =\frac{\der\zeta^1}{\zeta^1}\wedge\frac{\der\zeta^2}{\zeta^2}\wedge\eta^3,\quad\text{where }\eta^3=
\begin{dcases}
\der\zeta^3&\text{around }p\in D_{12}\setminus T_{12},\\
\der\zeta^3/\zeta^3&\text{around }p\in T_{12}.
\end{dcases}
\end{equation*}
Then $\Omega_0$ induces local meromorphic volume forms $\Omega_{i,0}$ on $X_i$ for $i=1,2$ given by
\begin{equation*}
\Omega_{i,0}=\residue_{X_i}\Omega_0=(-1)^{i'}\frac{\der\zeta^{i'}}{\zeta^{i'}}\wedge\eta^3,
\quad\text{where we set }i'=3-i,
\end{equation*}
which locally generate $H^0(X_i,L_i)$.
Consequently, these local generators $\Omega_{i,0}$ of $H^0(X_i,L_i)$ satisfy
\begin{equation*}
\residue_{D_{12}}\Omega_{1,0}=-\eta^3=-\residue_{D_{21}}\Omega_{2,0},
\end{equation*}
which explains the minus sign in the gluing isomorphisms $-f_{ij}^*:L_{ji}\to L_{ij}$ of $K_X$.

We immediately obtain the following result from Definition $\ref{def:K_X_SNC}$ because a section of $H^0(X_i,L_i)$ is given by a meromorphic volume form on $X_i$ with a single pole along $D_i$.
\begin{lemma}\label{lem:cond_K_X=O}
Let $X=\smallcup_{i=1}^N X_i$ be an SNC complex surface given by gluing isomorphisms $f_{ij}:D_{ij}\to D_{ji}$ for all $i,j$ with $i\leqslant j$ and $j\in I_i$.
Then the canonical bundle of $K_X$ is trivial if and only if the following conditions hold:
\renewcommand{\theenumi}{\roman{enumi}}
\begin{itemize}
\item each $D_i=\sum_{j\in I_i}D_{ij}$ is an anticanonical divisor on $X_i$; and
\item there exists a meromorphic volume form $\Omega_i$ on each $X_i$ with a single pole along $D_i$
such that
\begin{equation*}
\residue_{D_{ij}}\Omega_i =-f_{ij}^*\left(\residue_{D_{ji}}\Omega_j\right)\quad\text{for all }i,j\text{ with }i\leqslant j,j\in I_i.
\end{equation*}
\end{itemize}
\end{lemma}
This lemma implies that if we have $D_{ij}=D_{ji}$ for all $i,j$,
then conditions $\mathrm{(ii)}$ and $\mathrm{(iii)}$ of Theorem $\ref{thm:smoothing}$ are necessary and sufficient for the canonical bundle $K_X$ of $X$ to be trivial.
Also, Definition $\ref{def:K_X_SNC}$ leads to the following result, which is useful for computing $H^0(X,K_X)$.
\begin{proposition}\label{prop:gl.sect_K_X}
Let $X=\smallcup_{i=1}^NX_i$ be an SNC complex surface
given by gluing isomorphisms $f_{ij}:D_{ij}\to D_{ji}$ for all $i,j$ with $i\leqslant j$ and $j\in I_i$.
Then we have an exact sequence
\begin{equation*}
\vcenter{\xymatrix{
0\ar[r]&H^0(X,K_X)\ar[r]&\displaystyle\smalloplus_{i=1}^N H^0(X_i,L_i)
\ar[rr]^-{\displaystyle\rho =\smashoperator{\smalloplus_{i\leqslant j,j\in I_i}}\rho_{ij}}
&&\displaystyle\smalloplus_{i\leqslant j,j\in I_i}H^0(D_{ij},L_{ij}),
}}
\end{equation*}
where $L_i,L_{ij}$, and $\rho_{ij}$ are given by
\begin{gather*}
L_i=\restrict{K_X}{X_i}=K_{X_i}\otimes [D_i],
\quad L_{ij}=\restrict{L_i}{D_{ij}}=K_{D_{ij}}\otimes [T_{ij}],\\
\rho_{ij}:H^0(X_i,L_i)\oplus H^0(X_j,L_j)\to H^0(D_{ij},L_{ij}),\quad
(\Omega_i,\Omega_j)\mapsto\residue_{D_{ij}}\Omega_i+f_{ij}^*\left(\residue_{D_{ji}}\Omega_j\right) .
\end{gather*}
Hence, $H^0(X,K_X)$ is given by the kernel of the linear map $\rho$.
In particular, if $D_i=\sum_{j\in I_i}D_{ij}$ is an anticanonical divisor for all $i$,
then we have $L_i\cong\C$ and $L_{ij}\cong\C$ for all $i,j$, so that $\rho$ defines a linear map from $\C^N$ to $\C^M$, 
where $M$ is the number of double curves $D_{ij}$ with $i\leqslant j$.
\end{proposition}
We will use Proposition $\ref{prop:gl.sect_K_X}$ in Example $\ref{ex:Fujita}$.

\subsection{Semistable degenerations of $K3$ surfaces}\label{subsec:Friedman}
In this subsection, we give a summary of the classification of degenerations of $K3$ surfaces.
Let $\varpi :\mathcal{X}\to\Delta$ be a proper surjective holomorphic map from a compact complex $3$-dimensional manifold $\mathcal{X}$
to a domain $\Delta$ in $\C$ such that
\renewcommand{\labelenumi}{(\arabic{enumi})}
\begin{enumerate}
\item $\mathcal{X}$ is smooth outside the central fiber $X=\varpi^{-1}(0)$, and
\item for each $\zeta\in\Delta^{\! *}=\Delta\setminus\{0 \}$, the general fiber $X_\zeta =\varpi^{-1}(\zeta)$ is a smooth compact complex surface.
\end{enumerate}
We call $\varpi$ a {\emph{degeneration}} of compact complex surfaces.
Furthermore, a degeneration $\varpi$ is said to be \emph{semistable} if
\begin{enumerate}
\setcounter{enumi}{2}
\item the total space $\mathcal{X}$ is smooth, and
\item the central fiber $X$ is a normal crossing complex surface whose irreducible components are all K\"{a}hlerian.
\end{enumerate}
Let $\varpi :\mathcal{X}\to\Delta$ be a semistable degeneration of $K3$ surfaces, that is, the general fiber $X_\zeta$ is a $K3$ surface.
A degeneration $\varpi':\mathcal{X}'\to\Delta$ with the smooth total space $\mathcal{X}'$
is said to be a \emph{modification} of $\varpi :\mathcal{X}\to\Delta$
if there exists a birational map $\rho :\mathcal{X}\dasharrow\mathcal{X}'$ which is compatible with the projections $\varpi$ and $\varpi'$.
In particular, $\rho$ is an isomorphism over $\Delta^{\! *}$.
Mumford's semistable reduction theorem states that after a base change and a birational modification,
the central fiber of a degeneration is a reduced divisor with normal crossings.
Furthermore, if $\varpi:\mathcal{X}\to \Delta$ is a semistable degeneration of $K3$ surfaces,
then there is a modification $\varpi': \mathcal{X}'\to \Delta$ such that the total space $\mathcal{X}'$ has trivial canonical bundle,
according to the results of Kulikov \cite{Ku77} and Persson-Pinkham (see \cite{Huybrechts}, Chapter $6$, Theorem $5.1$).

The new family $\varpi':\mathcal{X}'\to \Delta$ is said to be a \emph{Kulikov degeneration} of the original degeneration $\varpi:\mathcal{X}\to \Delta$,
and the central fibers of Kulikov degenerations are classified into the following three cases, due to Kulikov \cite{Ku77} and Persson \cite{P77}.
\begin{theorem}[\cite{Ku77}, Theorem \rom{2}.
See also \cite{Huybrechts}, Chapter $6$, Theorem $5.2$]
\label{thm:d-s.sK3}
Let $\varpi: \mathcal{X}\to\Delta$ be a Kulikov degeneration, that is,
a semistable degeneration of $K3$ surfaces with $K_{\mathcal{X}}=\mathcal{O}_{\mathcal{X}}$ as above.
Then the central fiber $X=\varpi^{-1}(0)$ is one of the following three types:
\begin{enumerate}
\item[Type \rom{1}:] $X$ is a smooth $K3$ surface.
\item[Type \rom{2}:] $X=X_1\cup \dots \cup X_N$ is a chain of surfaces, where $X_1$ and $X_N$ are rational surfaces,
$X_2,\dots ,X_{N-1}$ are elliptic ruled surfaces, and $X_i\cap X_{i+1}$,~$i=1, \dots N-1$ are smooth elliptic curves.
\item[\hspace{5ex}Type \rom{3}:] $X=\smallcup_{i=1}^N X_i$, where each $X_i$ is a rational surface
and the double curves $D_{ij}=X_i\cap X_j \subseteq X_i$ are cycles of rational curves.
\end{enumerate}
\end{theorem}
If one omits the assumption that all irreducible components of $X$ are K\"{a}hlerian in Theorem $\ref{thm:d-s.sK3}$,
other types of surfaces may arise as irreducible components of the central fiber of semistable degenerations \cite{Ni88}.
Meanwhile, it is known that $d$-semistability defined in $\eqref{def:d-s.s}$ is a necessary condition
for an SNC complex manifold $X$ to be the central fiber of a semistable degeneration.
Hence, it is natural to consider the converse problem:
\emph{For what $d$-semistable SNC complex manifolds (or surfaces) $X$ does there exist a family of (global) smoothings $\varpi :\mathcal{X}\to\Delta$ of $X$,
i.e., a semistable degeneration $\varpi :\mathcal{X}\to\Delta$ such that $X$ is the central fiber of $\varpi$?}
Friedman investigated this problem for $K3$ surfaces \cite{Fr83}.
In order to state his result more precisely, we need the following.
\begin{definition}\label{def:d-ssK3}
Let $X=\smallcup_i X_i$ be an SNC compact complex surface.
$X$ is said to be a \emph{$d$-semistable $K3$ surface if}
\renewcommand{\labelenumi}{(\roman{enumi})}
\begin{enumerate}
\item $X$ is $d$-semistable and each $X_i$ is K\"{a}hlerian;
\item $X$ has trivial canonical bundle; and
\item $X$ is of either Type \rom{1}, \rom{2}, or \rom{3} given in Theorem $\ref{thm:d-s.sK3}$.
\end{enumerate}
\end{definition}
For a $d$-semistable $K3$ surface $X$, we have $H^1(X,\mathcal{O}_X)=0$ by \cite{Fr83}, Lemma $5.7$.
Moreover, Friedman proved that any $d$-semistable $K3$ surface will be the central fiber of a semistable degeneration.
\begin{theorem}[\cite{Fr83}, Theorem $5.10$]\label{thm:Friedman}
Any $d$-semistable $K3$ surface is smoothable to $K3$ surfaces.
More precisely, there exists a semistable degeneration $\varpi :\mathcal{X}\to\Delta$
such that the fiber $X_\zeta=\varpi^{-1}(\zeta )$ over each $\zeta\in\Delta^{\! *}$ is a $K3$ surface.
\end{theorem}

When we posted a preprint of this article 
on arXiv, we did not know whether there exists an SNC complex surface $X$
which satisfies conditions (i), (ii) of Definition $\ref{def:d-ssK3}$ and $H^1(X,\mathcal{O}_X)=0$, but not condition (iii).
However, it was known that such surfaces actually exist and the following example was kindly mentioned to us by the referee.

Starting from a $d$-semistable $K3$ surface of Type \rom{2}, we consider the total space of a family of smoothings $\varpi: \mathcal X \to \Delta$. 
Assume the rational surface $X_1$ of the central fiber $X=\smallcup_{i=1}^N X_i$ contains a $(-1)$-curve $E$. We further assume that $E$ is a $(-1,-1)$-curve on $\mathcal X$, i.e.,
$E$ is an algebraic curve on $\mathcal X$ such that $E\cong \C P^1$ and $N_{C/\mathcal{X}}\cong \mathcal O_{\C P^1}(-1)\oplus \mathcal O_{\C P^1}(-1)$.
Then the elliptic curve $X_1\cap X_2$ intersects to $E$ at the point $P$, which is an ordinary double point in $X_2\subset \mathcal{X}$.
Taking the Atiyah flop $\varphi:\mathcal{X}'\dasharrow \mathcal X$ (which is called the {\emph{elementary modifications}} in \cite{FM83}), we consider the proper transform $\widetilde{X}_2$ of $X_2$ under $\varphi$.
Then the restriction of $\varphi$ on $\widetilde{X}_2$ gives the blow-up
\[
\restrict{\varphi}{\widetilde{X}_2}: \widetilde{X}_2=\mathrm{Bl}_P(X_2) \dasharrow X_2,
\]
which yields that $\widetilde{X}_2$ is no longer a ruled surface.
Thus, we conclude that the central fiber $X'$ of $\mathcal{X}'$ does not belong to any of the types in Theorem $\ref{thm:d-s.sK3}$.

\section{Explicit construction of differential geometric global smoothings}\label{sec:Smoothing}
In this section, we explicitly construct differential geometric global smoothings
of a given SNC complex surface $X$ satisfying conditions (i)--(iii) of Theorem $\ref{thm:smoothing}$.
For this purpose, we introduce in Section $\ref{subsec:coordinates}$ local holomorphic coordinates on $X$ suited to the smoothing problem.
Then we construct local smoothings around double curves $D_{ij}$ without a triple point, and triple points $T_{ijk}$
in Sections $\ref{subsec:smoothing_double}$ and $\ref{subsec:smoothing_triple}$, respectively.
In Section $\ref{subsec:existence}$, we construct global smoothings $\varpi :\mathcal{X}\to\Delta$ of $X$ by gluing together the above local smoothings,
and then use Theorem $\ref{thm:existence}$ to prove Theorem $\ref{thm:smoothing}$ which states that
each fiber $X_\zeta =\varpi^{-1}(\zeta )$ admits a complex structure with trivial canonical bundle depending continuously on $\zeta\in\Delta$.
Section $\ref{subsec:ex_SNC_double}$ provides several explicit examples of $d$-semistable SNC complex surfaces with trivial canonical bundle
including at most double curves, which we see are smoothable to complex tori, primary Kodaira surfaces, and $K3$ surfaces by Theorem $\ref{thm:smoothing}$.
In particular, we construct in Example $\ref{ex:K3_typeII}$ $d$-semistable $K3$ surfaces of Type \rom{2}
with any number $N\geqslant 2$ of irreducible components.

\subsection{Local coordinates on an SNC complex surface}\label{subsec:coordinates}
Let $X=\smallcup_{i=1}^N X_i$ (or possibly $X=\smallcup_{i=0}^{N-1} X_i$) be an SNC complex surface
satisfying conditions (i)--(iii) of Theorem \ref{thm:smoothing},
and $X_i=\smallcup_{\alpha\in\Lambda_i}U_{i,\alpha}$ be an open covering of $X_i$ such that $\Lambda_i$ a finite subset of $\N\cup\{ 0\}$.
We can find a local holomorphic coordinate system $\{ U_{i,\alpha},(z_{i,\alpha}^1,z_{i,\alpha}^2)\}$ on $X_i$, satisfying the following conditions:
\renewcommand{\labelenumi}{(\Alph{enumi})}
\begin{enumerate}
\item $U_{i,\alpha}$ is an open neighborhood of $(0,0)$ in $\C^2$;
\item if $U_{i,\alpha}\cap D_i\neq\emptyset$ and $U_{i,\alpha}\cap T_i =\emptyset$,
then $U_{i,\alpha}=\{(z_{i,\alpha}^1,z_{i,\alpha}^2)\in\C^2|z_{i,\alpha}^1\in U^1_{i,\alpha},\norm{z_{i,\alpha}^2}<1\}$
for some domain $U^1_{i,\alpha}$ in $\C$, and $U_{i,\alpha}\cap D_i=\{ z_{i,\alpha}^2=0\}$; and
\item if $U_{i,\alpha}\cap T_i\neq\emptyset$,
then $U_{i,\alpha}=\set{(z_{i,\alpha}^1,z_{i,\alpha}^2)\in\C^2|\norm{z_{i,\alpha}^1}<1,\norm{z_{i,\alpha}^2}<1}$,
and $U_{i,\alpha}\cap D_i=\{ z_{i,\alpha}^1 z_{i,\alpha}^2=0\}$, so that $U_{i,\alpha}\cap T_i=\{ z_{i,\alpha}^1=z_{i,\alpha}^2=0\}$.
\end{enumerate}
In particular, each $U_{i,\alpha}$ contains at most one triple point.
Now we set
\begin{align*}
\Lambda_i^{(0)}&=\set{\alpha\in\Lambda_i|U_{i,\alpha}\cap D_i=\emptyset},\\
\Lambda_i^{(1)}&=\set{\alpha\in\Lambda_i|U_{i,\alpha}\cap D_i\neq\emptyset\textrm{ and }U_{i,\alpha}\cap T_i=\emptyset},\\
\Lambda_i^{(2)}&=\set{\alpha\in\Lambda_i|U_{i,\alpha}\cap T_i\neq\emptyset},\\
\Lambda_{ij}&=\set{\alpha\in\Lambda_i|U_{i,\alpha}\cap D_{ij}\neq\emptyset},
\quad\Lambda_{ij}^{(1)}=\Lambda_{ij}\cap\Lambda_i^{(1)},\quad\Lambda_{ij}^{(2)}=\Lambda_{ij}\cap\Lambda_i^{(2)},\quad\text{and}\\
\Lambda_{ijk}&=\Lambda_{ij}\cap\Lambda_{ik}\subset\Lambda_i^{(2)}.
\end{align*}
Then we choose $\Lambda_i$ so that
\begin{enumerate}
\setcounter{enumi}{3}
\item if $U_{i,\alpha}\cap U_{j,\beta}=\emptyset$ for $i\neq j$, $\alpha\in\Lambda_i$, and $\beta\in\Lambda_j$, then $\alpha\neq\beta$; and
\item $\Lambda_{ij}=\Lambda_{ji}$ and if $\alpha\in\Lambda_{ij}$, then $U_{i,\alpha}\cap D_{ij}=U_{j,\alpha}\cap D_{ij}$ and
$\restrict{z_{i,\alpha}^1}{D_{ij}}=\restrict{z_{j,\alpha}^1}{D_{ij}}$.
\end{enumerate}
In particular, we see from condition (E) that $\Lambda_{i'j'k'}=\Lambda_{ijk}$ for any permutation $(i',j',k')$ of $(i,j,k)$.
By condition (ii) of Theorem \ref{thm:smoothing}, we can choose the above coordinate system so that
\begin{enumerate}
\setcounter{enumi}{5}
\item the meromorphic volume form $\Omega_i$ in (iii) of Theorem \ref{thm:smoothing} can be locally represented on $U_{i,\alpha}$ as
\begin{equation*}
\Omega_i=\begin{dcases}
\der z_{i,\alpha}^1\wedge\der z_{i,\alpha}^2&\text{if }\alpha\in\Lambda_i^{(0)},\\
\der z_{i,\alpha}^1\wedge\frac{\der z_{i,\alpha}^2}{z_{i,\alpha}^2}&\text{if }\alpha\in\Lambda_i^{(1)},\\
\sigma_{i,\alpha}\frac{\der z_{i,\alpha}^1}{z_{i,\alpha}^1}\wedge\frac{\der z_{i,\alpha}^2}{z_{i,\alpha}^2}&\text{if }\alpha\in\Lambda_i^{(2)},
\end{dcases}
\end{equation*}
where $\sigma_{i,\alpha}$ for $\alpha\in\Lambda_i^{(2)}$ takes either $1$ or $-1$.
\end{enumerate}

In terms of the above coordinate system, we define a new one
$\{ U_{i,\alpha}, (z_{ij,\alpha}, w_{ij,\alpha})\}_{\alpha\in\Lambda_{ij}}$ around $D_{ij}$ as follows:
\begin{enumerate}
\setcounter{enumi}{6}
\item if $\alpha\in\Lambda_{ij}^{(1)}$, then we set $z_{ij,\alpha}=z_{i,\alpha}^1,\; w_{ij,\alpha}=z_{i,\alpha}^2$; and
\item if $\alpha\in\Lambda_{ij}^{(2)}$, then between $z_{i,\alpha}^1$ and $z_{i,\alpha}^2$,
we choose as $w_{ij,\alpha}$ the coordinate which is a defining function of $D_{ij}$ on $U_{i,\alpha}$, and $z_{ij,\alpha}$ as the remainder,
so that $U_{i,\alpha}\cap D_{ij}=\{ w_{ij,\alpha}=0\}$ and $U_{i,\alpha}\cap D_{ik}=\{ z_{ij,\alpha}=0\}$ for some $k\in I_{ij}$.
\end{enumerate}
In particular, we have
\begin{equation}\label{eq:zw}
z_{ik,\alpha}=w_{ij,\alpha},\quad w_{ik,\alpha}=z_{ij,\alpha}\quad\text{for }\alpha\in\Lambda_{ijk}.
\end{equation}
Condition (E) is rephrased as
\begin{enumerate}
\item[($\text{E}'$)] Letting $U_{ij,\alpha}=U_{i,\alpha}\cap D_{ij}$ and $x_{ij,\alpha}=\restrict{z_{ij,\alpha}}{U_{ij,\alpha}}$ for $\alpha\in\Lambda_{ij}$,
we have $\Lambda_{ij}=\Lambda_{ji}$, $U_{ij,\alpha}=U_{ji,\alpha}$, and $x_{ij,\alpha}=x_{ji,\alpha}$ for all $i\neq j$ and $\alpha\in\Lambda_{ij}$.
\end{enumerate}
We can further choose the coordinate system so that the following condition holds.
\begin{enumerate}
\setcounter{enumi}{8}
\item The meromorphic volume form $\Omega_i$ in (F) is locally represented on $U_{i,\alpha}$ for $\alpha\in\Lambda_{ij}$ as
\begin{equation}\label{eq:Omega_i_Lambda_ij}
\Omega_i=\begin{dcases}
-\epsilon_{ij}\der z_{ij,\alpha}\wedge\frac{\der w_{ij,\alpha}}{w_{ij,\alpha}}&\text{for }\alpha\in\Lambda_{ij}^{(1)},\\
-\sigma_{ijk}\frac{\der z_{ij,\alpha}}{z_{ij,\alpha}}\wedge\frac{\der w_{ij,\alpha}}{w_{ij,\alpha}}
&\text{for }\alpha\in\Lambda_{ijk}\subset\Lambda_{ij}^{(2)},
\end{dcases}
\end{equation}
where $\epsilon_{ij}=(j-i)/\norm{j-i}$, and $\sigma_{ijk}\in\{ 1,-1\}$ does not depend on $\alpha\in\Lambda_{ijk}$ and satisfies
\begin{equation}\label{eq:relation_sigma}
\sigma_{i'j'k'}=\sgn\begin{pmatrix}i&j&k\\ i'&j'&k'\end{pmatrix}\cdot\sigma_{ijk}.
\end{equation}
Also, the Poincar\'{e} residue $\residue_{D_{ij}}\Omega_i$ gives a meromorphic volume form $\psi_{D_{ij}}$ on $D_{ij}$,
which is locally represented on $U_{ij,\alpha}$ for $\alpha\in\Lambda_{ij}$ as
\begin{equation}\label{eq:psi_D_ij}
\psi_{D_{ij}}=\begin{dcases}
\epsilon_{ij}\der x_{ij,\alpha}&\text{for }\alpha\in\Lambda_{ij}^{(1)},\\
\sigma_{ijk}\frac{\der x_{ij,\alpha}}{x_{ij,\alpha}}&\text{for }\alpha\in\Lambda_{ijk}\subset\Lambda_{ij}^{(2)}.
\end{dcases}
\end{equation}
\end{enumerate}
Indeed, condition (iii) of Theorem $\ref{thm:smoothing}$ and equation $\eqref{eq:zw}$, respectively, give that
\begin{equation*}\label{eq:sigma_ikj}
\sigma_{ikj}=-\sigma_{ijk}\quad\text{and}\quad\sigma_{ikj}=-\sigma_{ijk},
\end{equation*}
from which $\eqref{eq:relation_sigma}$ follows.
Let $U'_{i,\alpha}=U_{i,\alpha}\setminus D_i$ and $U'_{ij,\alpha}=U_{ij,\alpha}\setminus T_{ij}$.
Then $\eqref{eq:Omega_i_Lambda_ij}$ (resp. $\eqref{eq:psi_D_ij}$) is a local representation of
the \emph{holomorphic} volume form $\Omega_i$ on $U'_{i,\alpha}$ in $X_i\setminus D_i$
(resp. $\psi_{D_{ij}}$ on $U'_{ij,\alpha}$ in $D_{ij}\setminus T_{ij}$).
Also, we can define a smooth Hermitian form $\omega_{D_{ij}}$ on $D_{ij}\setminus T_{ij}$ by
\begin{equation}\label{eq:omega_D_ij}
\omega_{D_{ij}}=\frac{\I}{2}\psi_{D_{ij}}\wedge\overline{\psi}_{D_{ij}},
\end{equation}
which is locally represented on $U'_{ij,\alpha}$ for $\alpha\in\Lambda_{ij}$ as
\begin{equation*}
\omega_{D_{ij}}=\begin{dcases}
\frac{\I}{2}\der x_{ij,\alpha}\wedge\der\overline{x}_{ij,\alpha}&\text{for }\alpha\in\Lambda_{ij}^{(1)},\\
\frac{\I}{2}\frac{\der x_{ij,\alpha}\wedge\der\overline{x}_{ij,\alpha}}{\norm{x_{ij,\alpha}}^2}&\text{for }\alpha\in\Lambda_{ij}^{(2)}.
\end{dcases}
\end{equation*}

\begin{example}\label{ex:tetra_in_CP3}
Here we suppose indices $i,j,k$, and $\ell$ will take $0,1,2$, or $3$.
Let us consider $X=\smallcup_{i=0}^3X_i$ in $\C P^3$, where each $X_i$ is a complex hyperplane
\begin{equation*}
X_i=\set{[\bm{z}]=[z^0,z^1,z^2,z^3]\in\C P^3|z^i=0}\cong\C P^2.
\end{equation*}
We will see that $X$ is an SNC complex surface satisfying conditions (ii) and (iii) of Theorem $\ref{thm:smoothing}$, but not condition (i).
We will also calculate all $\sigma_{ijk}\in\{ 1,-1\}$ which appear in the local representation $\eqref{eq:Omega_i_Lambda_ij}$
of the meromorphic volume form $\Omega_i$ on $X_i$.\\
\indent (1)~{\it $X$ is an SNC complex surface.}\\
Letting $U_i=\set{[\bm{z}]\in\C P^3|z^i\neq 0}$, we have an open covering $\{ U_i\}$ of $\C P^3$.
Let us define $\zeta_i^j=z^j/z^i$ on $U_i$.
Then $\bm{\zeta}_i=(\zeta_i^0,\dots ,\zeta_i^3)$ with $\zeta_i^i=1$ defines local coordinates on $U_i$.
Also, defining meromorphic $1$-forms $\eta_i^j=\der\zeta_i^j/\zeta_i^j$ on $U_i$, we have
\begin{equation}\label{eq:eta}
\eta_j^k=\eta_i^k-\eta_i^j\quad\text{on }U_i\cap U_j
\end{equation}
by the coordinate transformation $\bm{\zeta}_j=(\zeta_i^j)^{-1}\bm{\zeta}_i$ on $U_i\cap U_j$.
If $i,j,k$, and $\ell$ are mutually distinct, then in terms of the local coordinates $\bm{\zeta}_\ell\in\C^3\times\{ 1\}$ on $U_\ell$,
$X\cap U_\ell=X_i\cup X_j\cup X_k$ is represented as $\set{\bm{\zeta}_\ell\in\C^3\times\{ 1\}|\zeta_\ell^i\zeta_\ell^j\zeta_\ell^k =0}$.
Hence, $X$ is an SNC complex surface.\\
\indent (2)~{\it $X$ satisfies condition {\rm (ii)} of Theorem $\ref{thm:smoothing}$.}\\
Letting $U_{i,j}=\set{[\bm{z}]\in X_i|z^j\neq 0}=X_i\cap U_j$, we have an open covering $\{ U_{i,j}\}_{j\neq i}$ of $X_i$
and local coordinates $\bm{\zeta}_{i,j}=(\zeta_j^0,\dots ,\widehat{\zeta_j^i},\dots ,\zeta_j^3)$ with $\zeta_j^j=1$ on $U_{i,j}$ for $j\neq i$ in $X_i$,
where the hatted component is meant to be omitted.
Note that in this example, we are using the notation $U_{i,j}$ above and $U_{ij,k}$ below in place of $U_{i,\alpha}$ and $U_{ij,\alpha}$, respectively.
Then we see that $\zeta_k^j$ for $j,k\neq i$ is a local defining function of $D_{ij}$ on $U_{i,k}$.
Thus, $D_i$ is an anticanonical divisor on $X_i$ because we have
\begin{equation*}
K_{X_i}=K_{\C P^2}\cong\mathcal{O}_{\C P^2}(-3)\quad\text{and}\quad[D_i]=[D_{ij}]\otimes [D_{ik}]\otimes [D_{i\ell}]\cong\mathcal{O}_{\C P^2}(3),
\end{equation*}
where $j,k,\ell\neq i$ are all distinct.\\
\indent (3)~{\it Meromorphic volume forms $\Omega_i$ on $X_i$, which give all $\sigma_{ijk}$,
satisfy condition {\rm (iii)} of Theorem $\ref{thm:smoothing}$.}\\
Using $\eqref{eq:eta}$, we obtain well-defined meromorphic volume forms $\Omega_i$ on $X_i$ locally represented as
\begin{equation*}
\Omega_i=-\epsilon_{ijk\ell}\,\eta_j^k\wedge\eta_j^\ell\quad\text{on }U_{i,j}\text{ for }k,\ell\in\{ 0,1,2,3\}\setminus\{ i,j\}\text{ with }k\neq\ell ,
\end{equation*}
where $\epsilon_{ijk\ell}=\epsilon_{ij}\epsilon_{ik}\epsilon_{i\ell}\epsilon_{jk}\epsilon_{j\ell}\epsilon_{k\ell}$ is the Levi-Civita symbol.
Indeed, $\Omega_i$ is induced from a meromorphic volume form $\Omega_{\C P^3}$ on $\C P^3$ defined by
\begin{equation*}
\Omega_{\C P^3}=(-1)^j\eta_j^0\wedge\dots\wedge\widehat{\eta_j^j}\wedge\dots\wedge\eta_j^3\quad\text{on }U_j,
\end{equation*}
by the Poincar\'{e} residue map as $\Omega_i=\residue_{X_i}\Omega_{\C P^3}$.
Thus, we have $\sigma_{ijk}=\epsilon_{ijk\ell}$, where $\ell\in\{ 0,1,2,3\}\setminus\{ i,j,k\}$.
Letting $U_{ij,k}=\set{[\bm{z}]\in D_{ij}|z^k\neq 0}=X_i\cap X_j\cap U_k$, we have an open covering $\{ U_{ij,k}\}_{k\neq i,j}$ of $D_{ij}$
and local coordinates $\bm{\zeta}_{ij,k}=(\zeta_k^0,\dots ,\widehat{\zeta_k^i},\dots ,\widehat{\zeta_k^j},\dots ,\zeta_k^3)$ with $\zeta_k^k=1$
on $U_{ij,k}$ for $k\neq i,j$ in $D_{ij}$.
Then the Poincar\'{e} residues $\residue_{D_{ij}}\Omega_i$ and $\residue_{D_{ij}}\Omega_j$ are locally represented on $U_{ij,k}$ as
\begin{align*}
\residue_{D_{ij}}\Omega_i&=-\residue_{\zeta_k^j=0}\epsilon_{ikj\ell}\,\eta_k^j\wedge\eta_k^\ell
=-\epsilon_{ikj\ell}\,\eta_k^\ell =\epsilon_{ijk\ell}\,\eta_k^\ell,\\
\residue_{D_{ij}}\Omega_j&=-\residue_{\zeta_k^i=0}\epsilon_{jki\ell}\,\eta_k^i\wedge\eta_k^\ell
=-\epsilon_{jki\ell}\,\eta_k^\ell =-\epsilon_{ijk\ell}\,\eta_k^\ell ,
\end{align*}
so that we have $\residue_{D_{ij}}\Omega_i=-\residue_{D_{ij}}\Omega_j$.
Hence, condition (iii) of Theorem $\ref{thm:smoothing}$ holds.\\
\indent (4)~{\it $X$ does not satisfy condition {\rm (i)} of Theorem $\ref{thm:smoothing}$.}\\
Note that $\zeta_\ell^k$ is a local defining function of $T_{ijk}=X_i\cap X_j\cap X_k$ on $U_{ij,\ell}$ in $D_{ij}$.
Thus, for $N_{ij}=N_{D_{ij}/X_i}$ and $T_{ij}=\sum_{k\in I_{ij}}T_{ijk}$, it follows from the adjunction formula that
\begin{equation*}
N_{ij}\cong\restrict{[D_{ij}]}{D_{ij}}=\restrict{\mathcal{O}_{\C P^2}(1)}{\C P^1}=\mathcal{O}_{\C P^1}(1)\cong N_{ji},\quad [T_{ij}]
=[T_{ijk}]\otimes [T_{ij\ell}]\cong\mathcal{O}_{\C P^1}(2),
\end{equation*}
where $k,\ell\in\{ 0,1,2,3\}\setminus\{ i,j\}$ with $k\neq\ell$, so that $X$ is \emph{not} $d$-semistable and does not satisfy condition (i).

Although the above $X$ is not $d$-semistable, this example will provide
a local model of the neighborhood of a triple point in Section $\ref{subsec:smoothing_triple}$.
Also, we will see in Section $\ref{sec:TypeIIIK3}$ that if we blow up $X$ at appropriate points,
then the resulting SNC complex surface becomes a $d$-semistable $K$3 surface of Type \rom{3}
satisfying conditions (i)--(iii) of Theorem $\ref{thm:smoothing}$, so that we can apply Theorem $\ref{thm:smoothing}$.
\end{example}

Let $(x_{ij,\alpha},y_{ij,\alpha})$ be local holomorphic coordinates of $\pi_{ij}^{-1}(U_{ij,\alpha})\subset N_{ij}$,
where $\pi_{ij}$ is the projection from $N_{ij}$ to $D_{ij}$ and $y_{ij,\alpha}$ are fiber coordinates.
Also, let $D'_{ij}=D_{ij}\setminus T_{ij}$ and $N'_{ij}=\restrict{N_{ij}}{D'_{ij}}=N_{ij}\setminus\pi_{ij}^{-1}(T_{ij})$.
Then from condition (i) of Theorem \ref{thm:smoothing}, we may further assume that
\begin{enumerate}
\setcounter{enumi}{9}
\item the map $h_{ij,\zeta}:N'_{ij}\setminus D'_{ij}\to N'_{ji}\setminus D'_{ji}$ locally defined by
\begin{equation}\label{eq:h_ij}
h_{ij,\zeta}:(x_{ij,\alpha},y_{ij,\alpha})\mapsto (x_{ji,\alpha},y_{ji,\alpha})=\begin{dcases}
(x_{ij,\alpha},\zeta /y_{ij,\alpha})&\textrm{for }\alpha\in\Lambda_{ij}^{(1)},\\
(x_{ij,\alpha},\zeta /(x_{ij,\alpha}y_{ij,\alpha}))&\textrm{for }\alpha\in\Lambda_{ij}^{(2)},
\end{dcases}
\end{equation}
is a well-defined isomorphism for $\zeta\in\C^*$.
\end{enumerate}
Thus we see for all $i,j$ with $i\neq j$  that
\begin{equation}\label{eq:composition_h}
h_{ji,\xi}\circ h_{ij,\zeta}(x_{ij,\alpha},y_{ij,\alpha})=(x_{ij,\alpha},\zeta^{-1}\xi y_{ij,\alpha}).
\end{equation}

Now consider a Hermitian metric around $D_{ij}$ (which is temporary and different from what is considered later)
such that the associated $2$-form coincides with that of a flat metric
\begin{equation*}
\frac{\I}{2}\left(\der z_{ij,\alpha}\wedge\der\overline{z}_{ij,\alpha}+\der w_{ij,\alpha}\wedge\der\overline{w}_{ij,\alpha}\right)
\end{equation*}
on $U_{i,\alpha}$ for $\alpha\in\Lambda_{ij}^{(2)}$.
Then by the tubular neighborhood theorem, we have a diffeomorphism $\Phi_{ij}$
from a neighborhood $V_{ij}$ of the zero section of $N_{ij}$ to a neighborhood $W_{ij}$ of $D_{ij}$ in $X_i$
such that $W_{ij}\supset\smallcup_{\alpha\in\Lambda_{ij}}U_{i,\alpha}$ and $\Phi_{ij}$
is locally represented on $U_{i,\alpha}$ for $\alpha\in\Lambda_{ij}$ as
\begin{equation}\label{eq:zw_approx}
\begin{aligned}
z_{ij,\alpha}=x_{ij,\alpha}+O(\norm{y_{ij,\alpha}}^2),\quad&w_{ij,\alpha}=y_{ij,\alpha}+O(\norm{y_{ij,\alpha}}^2)
\quad\textrm{for }\alpha\in\Lambda_{ij}^{(1)},\textrm{ and}\\
z_{ij,\alpha}=x_{ij,\alpha},\quad&w_{ij,\alpha}=y_{ij,\alpha}\quad\textrm{for }\alpha\in\Lambda_{ij}^{(2)}.
\end{aligned}
\end{equation}
Thus, it follows from $\eqref{eq:zw}$ and $\eqref{eq:zw_approx}$ that
\begin{equation}\label{eq:relation_xy}
\begin{aligned}
x_{ij,\beta}&=x_{ji,\beta}=y_{ik,\beta}=y_{jk,\beta},\\
x_{ik,\beta}&=x_{ki,\beta}=y_{ij,\beta}=y_{kj,\beta},\\
x_{jk,\beta}&=x_{kj,\beta}=y_{ji,\beta}=y_{ki,\beta}\quad\textrm{for }\beta\in\Lambda_{ijk}.
\end{aligned}
\end{equation}

In the rest of this subsection, we assume $i<j$ unless otherwise mentioned.
Let $\Norm{\cdot}_{ij}$ be a bundle norm on $N'_{ij}$ such that 
\begin{equation}\label{eq:norm_ij_triple}
\Norm{(x_{ij,\beta},y_{ij,\beta})}_{ij}^2=\norm{x_{ij,\beta}}\cdot\norm{y_{ij,\beta}}^2\quad\text{for }\beta\in\Lambda_{ij}^{(2)},
\end{equation}
which makes sense because $N'_{ij}$ does not include the fibers over the triple points $T_{ij}$.
Then we define a cylindrical parameter $t_{ij}$ on $N'_{ij}\setminus D'_{ij}$ by
\begin{equation}\label{eq:def_t_ij}
t_{ij}(\bm{v})=-\log\Norm{\bm{v}}_{ij}^2\quad\text{for }\bm{v}\in N'_{ij}\setminus D'_{ij}.
\end{equation}
In particular, using $\eqref{eq:norm_ij_triple}$ we have
\begin{equation}\label{eq:t_ij_triple}
t_{ij}(x_{ij,\beta},y_{ij,\beta})=-\log\norm{x_{ij,\beta}}-\log\norm{y_{ij,\beta}}^2\quad\text{for }\beta\in\Lambda_{ij}^{(2)}.
\end{equation}
Note that $t_{ij}$ satisfies
\begin{equation}\label{eq:t_ij_zeta}
t_{ij}(\zeta\bm{v})=t_{ij}(\bm{v})-\log\norm{\zeta}^2\quad\text{for }\zeta\in\C^*.
\end{equation}
For later convenience, we further extend $t_{ij}$ so as to take $\infty$ on $D_{ij}$, so that $t_{ij}^{-1}(\infty )=D_{ij}$.
Via $\Phi_{ij}$, we can also define a function $\widetilde{t}_{ij}=t_{ij}\circ\Phi_{ij}^{-1}$
on $W_{ij}\setminus\smallcup_{k\in I_{ij}}(D_{ik}\setminus T_{ijk})\subset X_i$, which takes $\infty$ on $D_{ij}$.
We will denote $t_{ij}^{-1}(I)$ (resp. $\widetilde{t}_{ij}^{-1}(I)$) for an interval $I$ in $(0,\infty ]$
simply by $t_{ij}^{-1}I$ (resp. $\widetilde{t}_{ij}^{-1}I$).
Letting $D_{ij}^{(1)}=\smallcup_{\alpha\in\Lambda_{ij}^{(1)}}U_{ij,\alpha}\subset D'_{ij}$, we may assume that
\begin{equation*}
t_{ij}^{-1}(0,\infty ]\cap\pi_{ij}^{-1}(D_{ij}^{(1)})\subset V_{ij}\quad\text{and}\quad
\Phi_{ij}(t_{ij}^{-1}(0,\infty ]\cap\pi_{ij}^{-1}(D_{ij}^{(1)}))\subset\smallcup_{\alpha\in\Lambda_{ij}^{(1)}}U_{i,\alpha}.
\end{equation*}
Now setting
\begin{equation}\label{eq:component_V_ij}
\begin{aligned}
V_{ij}^{(1)}&=t_{ij}^{-1}(0,\infty ]\cap\pi_{ij}^{-1}(D_{ij}^{(1)}),\quad\text{and for }\beta\in\Lambda_{ij}^{(2)}\\
V_{ij,\beta}&=\Phi_{ij}^{-1}(U_{i,\beta})=\set{(x_{ij,\beta},y_{ij,\beta})\in\C^2|\norm{x_{ij,\beta}}<1,\norm{y_{ij,\beta}}<1},\\
V'_{ij,\beta}&=V_{ij,\beta}\cap N'_{ij}=\set{(x_{ij,\beta},y_{ij,\beta})\in\C^2|0<\norm{x_{ij,\beta}}<1,\norm{y_{ij,\beta}}<1},
\end{aligned}
\end{equation}
we redefine the neighborhood $V_{ij}$ of $D_{ij}$ in $N_{ij}$ by
\begin{equation}\label{eq:V_ij}
\begin{aligned}
V_{ij}&=V_{ij}^{(1)}\cup\smallcup_{\beta\in\Lambda_{ij}^{(2)}}V_{ij,\beta},\quad\text{so that}\\
V'_{ij}&=V_{ij}^{(1)}\cup\smallcup_{\beta\in\Lambda_{ij}^{(2)}}V'_{ij,\beta},\quad\text{where}\quad V'_{ij}=V_{ij}\cap N'_{ij}.
\end{aligned}
\end{equation}
Accordingly, redefining the neighborhood $W_{ij}=\Phi_{ij}(V_{ij})$ of $D_{ij}$ in $X_i$
and setting $W_{ij}^{(1)}=\Phi_{ij}(V_{ij}^{(1)})$, we have
\begin{equation}\label{eq:W_ij}
W_{ij}=W_{ij}^{(1)}\cup\smallcup_{\beta\in\Lambda_{ij}^{(2)}}U_{i,\beta}.
\end{equation}
Thus, $t_{ij}$ takes values in $(0,\infty ]$ on $V'_{ij}$,
and accordingly $\widetilde{t}_{ij}$ takes values in $(0,\infty ]$ on
$\Phi_{ij}(V'_{ij})=W_{ij}\setminus\smallcup_{k\in I_{ij}}D_{ik}$.
In particular, if $\Lambda_{ij}^{(2)}=\emptyset$, i.e., if $D_{ij}$ has no triple point,
then we have $D'_{ij}=D_{ij}$ and $V_{ij}=V'_{ij}=t_{ij}^{-1}(0,\infty ]$.
Meanwhile, we define a cylindrical parameter $t_{ji}$ on $N'_{ji}\setminus D'_{ji}$ by
\begin{equation}\label{eq:t_ji}
t_{ji}=-t_{ij}\circ h_{ji,1}.
\end{equation}
Then we see from $\eqref{eq:h_ij}$ that
$\eqref{eq:t_ij_triple}$ and $\eqref{eq:t_ij_zeta}$ still hold if we interchange $i$ and $j$.
As with $t_{ij}$, we will regard $t_{ji}$ as a function on $N'_{ji}\cup T_{ji}$, taking $\infty$ on $D_{ji}$.
Now let us define $T_\zeta\in (-\infty ,\infty ]$ for $\zeta\in\C$ by
\begin{equation}\label{eq:T_zeta}
T_\zeta =\begin{dcases}
-\log\norm{\zeta}&\text{if }\zeta\neq 0,\\
\infty&\text{if }\zeta =0.\end{dcases}
\end{equation}
Then by $\eqref{eq:composition_h}$, $\eqref{eq:t_ij_zeta}$ and $\eqref{eq:t_ji}$, we have for all $i,j$ with $i\neq j$ and $\zeta\in\C^*$ that
\begin{equation}\label{eq:pullback_t_ji}
h_{ij,\zeta}^*t_{ji}=-t_{ij}\circ h_{ji,1}\circ h_{ij,\zeta}=-t_{ij}-\log\norm{\zeta}^2=2T_\zeta -t_{ij}.
\end{equation}
In the same way as above, we can define a neighborhood $V_{ji}$ (resp. $V'_{ji}$) of $D_{ij}$ (resp. $D'_{ij}$) in $N_{ji}$ (resp. $N'_{ji}$),
a neighborhood $W_{ji}$ of $D_{ij}$ in $X_j$,
and a cylindrical parameter $\widetilde{t}_{ji}$ on $W_{ji}\setminus\smallcup_{k\in I_{ij}}(D_{jk}\setminus T_{ijk})$ which takes $\infty$ on $D_{ij}$.

For later use, we define $W_{ij}^T\subset W_{ij}$ and $X_i^T\subset X_i$ for $T\in (0,\infty ]$ by
\begin{equation}\label{eq:W-X^T}
\begin{aligned}
W_{ij}^T&=\begin{dcases}
\widetilde{t}_{ij}^{-1}(0,T)\setminus\smallcup_{k\in I_{ij}}\widetilde{t}_{ik}^{-1}[T,\infty ]&\text{for }T\in (0,\infty ),\\
W_{ij}&\text{for }T=\infty ,
\end{dcases}\quad\text{and}\\
X_i^T&=\begin{dcases}
X_i\setminus\smallcup_{j\in I_i}\widetilde{t}_{ij}^{-1}[T,\infty ]&\text{for }T\in (0,\infty ),\\
X_i&\text{for }T=\infty .
\end{dcases}
\end{aligned}
\end{equation}
Then we have $W_{ij}^T\subset X_i^T\subset X_i\setminus D_i$ for $T\in (0,\infty )$.

\subsection{Local smoothings of $X_i\cup X_j$ around $D_{ij}$ without a triple point}\label{subsec:smoothing_double}
Here we suppose $D_{12}\neq\emptyset$ is a double curve without a triple point, so that $T_{12}=\emptyset$.
Also, indices $i,j$ will take $1$ or $2$, and the pair $(i,j)$ will take $(1,2)$ or $(2,1)$.
For general $D_{n_1 n_2}=X_{n_1}\cap X_{n_2}$ with $n_1<n_2$, we replace subscripts $1,2$ and $i,j$ with $n_1,n_2$ and $n_i,n_j$, respectively.
In Section $\ref{subsec:coordinates}$, we have chosen the coordinate system $\{U_{i,\alpha},(z_{ij,\alpha},w_{ij,\alpha})\}$
with $\alpha\in\Lambda_{ij}=\Lambda_{ij}^{(1)}$ around $D_{ij}\subset X_i$,
where $w_{ij,\alpha}$ is a defining function of $D_{ij}$ on $U_{i,\alpha}$ and the holomorphic volume form $\Omega_i$ is locally represented as
$\epsilon_{ij}\der z_{ij,\alpha}\wedge\der w_{ij,\alpha}/w_{ij,\alpha}$ on $U'_{i,\alpha}$ for $\alpha\in\Lambda_{ij}$.
Also, we have a holomorphic volume form $\psi_{D_{ij}}=\residue_{D_{ij}}\Omega_i$ given by $\eqref{eq:psi_D_ij}$
and a Hermitian form $\omega_{D_{ij}}$ given by $\eqref{eq:omega_D_ij}$ on $D_{ij}$.

We define a smooth complex $2$-form $\Omega_{ij}^\infty$ and a real $2$-form $\omega_{ij}^\infty$ on $N_{ij}\setminus D_{ij}$ by
\begin{equation}\label{eq:Oomega_ij_double}
\begin{aligned}
\Omega_{ij}^\infty&=\pi_{ij}^*\psi_{D_{ij}}\wedge\del t_{ij},\\
\omega_{ij}^\infty&=\pi_{ij}^*\omega_{D_{ij}}+\frac{\I}{2}\del t_{ij}\wedge\delbar t_{ij}.
\end{aligned}
\end{equation}
Then $\Omega_{ij}^\infty$ and $\omega_{ij}^\infty$ are locally represented on $\pi_{ij}^{-1}(U_{ij,\alpha})$ for $\alpha\in\Lambda_{ij}$ as
\begin{align*}
\Omega_{ij}^\infty&=-\epsilon_{ij}\der x_{ij,\alpha}\wedge\frac{\der y_{ij,\alpha}}{y_{ij,\alpha}},\\
\omega_{ij}^\infty&=\frac{\I}{2}\left(\der x_{ij,\alpha}\wedge\der\overline{x}_{ij,\alpha}
+\frac{\der y_{ij,\alpha}\wedge\der\overline{y}_{ij,\alpha}}{\norm{y_{ij,\alpha}}^2}\right) .
\end{align*}
\begin{lemma}\label{lem:SU(2)_double}
The $2$-form $\Omega_{ij}^\infty$ defined in $\eqref{eq:Oomega_ij_double}$ is holomorphic on $N_{ij}\setminus D_{ij}$.
Also, $(\Omega_{ij}^\infty ,\omega_{ij}^\infty )$ defines an $\SU (2)$-structure on $N_{ij}\setminus D_{ij}$ such that the associated metric $g_{ij}^\infty$ is cylindrical.
\end{lemma}
\begin{proof}
Since the bundle norm $\Norm{\cdot}_{ij}$ on $N_{ij}$ is locally written as
\begin{equation*}
\Norm{(x_{ij,\alpha},y_{ij,\alpha})}_{ij}=e^{\phi_{ij,\alpha}(x_{ij,\alpha})}\norm{y_{ij,\alpha}}^2
\quad\text{for }(x_{ij,\alpha},y_{ij,\alpha})\in\pi_{ij}^{-1}(V_{ij,\alpha})
\end{equation*}
for some smooth function $\phi_{ij,\alpha}$ on $V_{ij,\alpha}$, it follows from $\eqref{eq:def_t_ij}$ that
\begin{equation*}
\Omega_{ij}^\infty =-\pi_{ij}^*\psi_{D_{ij}}\wedge\frac{\der y_{ij,\alpha}}{y_{ij,\alpha}}\quad\text{on }\pi_{ij}^{-1}(V_{ij,\alpha}),
\end{equation*}
which shows that $\Omega_{ij}^\infty$ is holomorphic on $N_{ij}\setminus D_{ij}$.
Also, we see easily from $\eqref{eq:Oomega_ij_double}$ that $2\omega_{ij}^2=\Omega_{ij}\wedge\overline{\Omega}_{ij}$.
Thus, $(\Omega_{ij}^\infty ,\omega_{ij}^\infty )$ defines an $\SU (2)$-structure on $N_{ij}\setminus D_{ij}$
due to Definitions $\ref{def:SU(2)_V}$ and $\ref{def:SU(2)_M}$.
For the proof that the associated metric is cylindrical, see \cite{Doi09}, Lemma $3.2$.
\end{proof}
Note in particular that $\Omega_{ij}^\infty$ is a $\der$-closed $\SL (2,\C )$-structure on $N_{ij}\setminus D_{ij}$.
If we regard $\Omega_{ij}^\infty$ and $\omega_{ij}^\infty$ as defined on $W_{ij}\setminus D_{ij}$ via $\Phi_{ij}$,
then it follows from $\eqref{eq:zw_approx}$ that
\begin{equation}\label{eq:asymptotic_SU(2)_double}
\norm{\Omega_i-\Omega_{ij}^\infty}=O(e^{-t_{ij}/2}),\quad\norm{\omega_{ij}-\omega_{ij}^\infty}=O(e^{-t_{ij}/2}),
\end{equation}
where $\omega_{ij}$ is the $(1,1)$-part of $\omega_{ij}^\infty$, normalized so that $\Omega_i\wedge\overline{\Omega}_i=2\omega_{ij}\wedge\omega_{ij}$,
and the norm is measured by the cylindrical metric $g_{ij}^\infty$ associated with $(\Omega_{ij}^\infty ,\omega_{ij}^\infty )$.
Letting $c>1$ and shrinking $U_{i,\alpha}$ for each $\alpha\in\Lambda_{ij}^{(1)}$ by the coordinate transformation $w_{ij,\alpha}\mapsto cw_{ij,\alpha}$
in conditions (B) and (G) if necessary, we may assume that $(\Omega_i,\omega_{ij})$ is an $\SU (2)$-structure on $W_{ij}\setminus D_{ij}$.
Thus we have the associated Hermitian metric $g_{ij}$ on $W_{ij}\setminus D_{ij}$.
We also see from $\eqref{eq:pullback_t_ji}$ and $\eqref{eq:Oomega_ij_double}$ that
\begin{equation}\label{eq:pullback_Oomega_ij_double}
h_{ij,\zeta}^*\Omega_{ji}^\infty =\Omega_{ij}^\infty,\quad h_{ij,\zeta}^*\omega_{ji}^\infty =\omega_{ij}^\infty.
\end{equation}

Now we construct local smoothings of $X_1\cup X_2$ around $D_{12}$.
Fix $\epsilon\leqslant e^{-1}$ and let $\Delta =\Delta_\epsilon =\set{\zeta\in\C|\norm{\zeta}<\epsilon}$ be a domain in $\C$.
Then a family of local smoothing $\varpi_{12}:\mathcal{V}_{12}\to\Delta$ of $N_{12}\cup N_{21}$ is given by
\begin{align*}
&\mathcal{V}_{12}=\set{(x_{12,\alpha},y_{12,\alpha},y_{21,\alpha})\in V_{12}\oplus V_{21}|y_{12,\alpha}y_{21,\alpha}\in\Delta},\\
&\varpi_{12}(x_{12,\alpha},y_{12,\alpha},y_{21,\alpha})=y_{12,\alpha}y_{21,\alpha},
\end{align*}
where $V_{ij}=t_{ij}^{-1}(0,\infty ]\subset N_{ij}$.
We have $\varpi_{12}^{-1}(0)=V_{12}\cup V_{21}$ and $V_{12}\cap V_{21}=D_{12}$.
Define $p_{ij}:\mathcal{V}_{12}\to N_{ij}$ by
\begin{equation*}
p_{ij}(x_{12,\alpha},y_{12,\alpha},y_{21,\alpha})=(x_{12,\alpha},y_{ij,\alpha}),
\end{equation*}
and also define $p_{ij,\zeta}=\restrict{p_{ij}}{\varpi_{12}^{-1}(\zeta )}$.
Suppose $\zeta\in\Delta^{\! *}=\Delta\setminus\{ 0\}$.
Then we see that $p_{ij,\zeta}$ is an injective diffeomorphism, i.e., a diffeomorphism onto its image.
We have
\begin{equation*}
p_{21}=h_{12,\zeta}\circ p_{12}\quad\text{on }\varpi_{12}^{-1}(\zeta ),
\end{equation*}
from which it follows that
\begin{equation}\label{eq:relation_hp}
h_{ij,\zeta}=p_{ji,\zeta}\circ p_{ij,\zeta}^{-1}\quad\text{on }p_{ij}(\varpi_{12}^{-1}(\zeta )).
\end{equation}
Also, we see from $\eqref{eq:pullback_t_ji}$ that
\begin{equation}\label{eq:t_12+t_21}
t_{12}\circ p_{12,\zeta}+t_{21}\circ p_{21,\zeta}=-\log\norm{\zeta}^2=2T_\zeta .
\end{equation}
Next suppose $\zeta =0$.
Then $p_{ij,\zeta}=p_{ij,0}$ is an identity map on $V_{ij}$ and coincides on $V_{ji}$ with the projection $\pi_{ji}$ to $D_{12}$.
Thus, it follows from $\eqref{eq:t_12+t_21}$ that the image of $p_{ij,\zeta}$ for $\zeta\in\Delta$ is given by
\begin{equation}\label{eq:image_p_zeta}
p_{ij,\zeta}(\varpi_{12}^{-1}(\zeta ))=\begin{dcases}
V_{ij}=t_{ij}^{-1}(0,\infty ]&\text{for }\zeta =0,\\
t_{ij}^{-1}(0,2T_\zeta)&\text{for } \zeta\in\Delta^{\! *}.\end{dcases}
\end{equation}
Hence, we see from $\eqref{eq:image_p_zeta}$ that $\varpi_{12}^{-1}(\zeta )$ for $\zeta\in\Delta^{\! *}$
is obtained by gluing together $t_{12}^{-1}(0,2T_\zeta )$ and $t_{21}^{-1}(0,2T_\zeta )$ using the diffeomorphism $h_{12,\zeta}$.
For $\zeta\in\Delta^{\! *}$, let us define a holomorphic volume form $\Omega_{12,\zeta}$ and a smooth Hermitian form $\omega_{12,\zeta}$
on $\varpi_{12}^{-1}(\zeta )$ by
\begin{equation*}
\Omega_{12,\zeta}=\left(p_{ij,\zeta}^{-1}\right)^*\Omega_{ij}^\infty\quad\text{and}\quad
\omega_{12,\zeta}=\left(p_{ij,\zeta}^{-1}\right)^*\omega_{ij}^\infty .
\end{equation*}
Then we see from $\eqref{eq:pullback_Oomega_ij_double}$ and $\eqref{eq:relation_hp}$
that $(\Omega_{12,\zeta},\omega_{12,\zeta})$ is a well-defined $\SU (2)$-structure on $\varpi_{12}^{-1}(\zeta )$.
Meanwhile, if $\zeta =0$, then $(\Omega_{ij}^\infty ,\omega_{ij}^\infty )$ gives a singular $\SU (2)$-structure
on the irreducible component $V_{ij}$ of $\varpi_{12}^{-1}(0)$.

We have injective diffeomorphisms $\widetilde{\Phi}_{ij}:\mathcal{V}_{12}\setminus (N_{ji}\setminus D_{12})\to X_i\times\Delta$
for $(i,j)=(1,2),(2,1)$ defined by
\begin{align*}
\widetilde{\Phi}_{ij}&=(\Phi_{ij}\circ p_{ij},\varpi_{12}),\quad\text{that is,}\\
\widetilde{\Phi}_{ij}(x_{12,\alpha},y_{12,\alpha},y_{21,\alpha})
&=(\Phi_{ij}(x_{12,\alpha},y_{ij,\alpha}),\varpi_{12}(x_{12,\alpha},y_{12,\alpha},y_{21,\alpha})),
\end{align*}
which are compatible with the projections to $\Delta$.
Noting that 
\begin{equation*}
\mathcal{V}_{12}\setminus (N_{ji}\setminus D_{12})=V_{ij}\cup\smallcup_{\zeta\in\Delta^{\! *}}\varpi_{12}^{-1}(\zeta ),
\end{equation*}
we see from $\eqref{eq:image_p_zeta}$ that the image of $\widetilde{\Phi}_{ij}$ is given by
\begin{equation}\label{eq:image_Phitilde_zeta}
\Image\widetilde{\Phi}_{ij}=\smallcup_{\zeta\in\Delta}(\Phi_{ij}\circ p_{ij}(\varpi_{12}^{-1}(\zeta )),\zeta )
=(W_{ij}\times\{ 0\})\cup\smallcup_{\zeta\in\Delta^*}\widetilde{t}_{ij}^{-1}(0,2T_\zeta )\times\{\zeta\} .
\end{equation}

Now let $W_{ij}^T$ and $X_i^T$ be as defined in $\eqref{eq:W-X^T}$.
Then defining $\mathcal{W}_{ij}\subset W_{ij}\times\Delta$ and $\mathcal{X}_i\subset X_i\times\Delta$ with $\mathcal{W}_{ij}\subset\mathcal{X}_i$ by
\begin{equation}\label{eq:XX-WW}
\mathcal{W}_{ij}=\smallcup_{\zeta\in\Delta}W_{ij}^{T_\zeta +1}\times\{\zeta\}\quad\text{and}\quad
\mathcal{X}_i=\smallcup_{\zeta\in\Delta}X_i^{T_\zeta +1}\times\{\zeta\} ,
\end{equation}
we see from $\eqref{eq:image_Phitilde_zeta}$ that $\widetilde{\Phi}_{ij}^{-1}:\mathcal{W}_{ij}\to\mathcal{V}_{12}\setminus (N_{ji}\setminus D_{12})$
is also an injective diffeomorphism which is compatible with the projections to $\Delta$.
Also, we see from $\eqref{eq:t_12+t_21}$ that
$\widetilde{\Phi}_{12}^{-1}\cup\widetilde{\Phi}_{21}^{-1}:\mathcal{W}_{12}\cup\mathcal{W}_{21}\to\mathcal{V}_{12}$ is surjective.
Thus, we can glue together $\mathcal{X}_1,\mathcal{X}_2$, and $\mathcal{V}_{12}$ using the injective diffeomorphism
$\widetilde{\Phi}_{ij}^{-1}:\mathcal{W}_{ij}\to\mathcal{V}_{12}\setminus (N_{ji}\setminus D_{12})$
to obtain a family of local smoothings of $X_1\cup X_2$ around $D_{12}$.
This gluing procedure is diagrammed as follows:
\begin{equation}\label{eq:gluing_double}
\vcenter{\xymatrix@R=0pt@C=0pt{
\mathcal{X}_1&&&\mathcal{V}_{12}\ar@{=}[r]&\mathcal{V}_{12}&&&\mathcal{X}_2\\
\dsubset&&&\dsubset&\dsubset&&&\dsubset\\
\mathcal{W}_{12}\ar@{^{(}->}[rrr]^-{\widetilde{\Phi}_{12}^{-1}}&&&\mathcal{V}_{12}\setminus (N_{21}\setminus D_{12})
&\mathcal{V}_{12}\setminus (N_{12}\setminus D_{12})&&&\mathcal{W}_{21}\ar@{_{(}->}[lll]_-{\widetilde{\Phi}_{21}^{-1}}\\
\dsubset&&&\dsubset&\dsubset&&&\dsubset\\
W_{12}^{T_\zeta +1}\times\{\zeta\}\ar@{^{(}->}[rrr]^-{\widetilde{\Phi}_{12}^{-1}}
&&&\varpi_{12}^{-1}(\zeta )\setminus (V_{21}\setminus D_{12})&\varpi_{12}^{-1}(\zeta )\setminus (V_{12}\setminus D_{12})
&&&W_{21}^{T_\zeta +1}\times\{\zeta\}\ar@{_{(}->}[lll]_-{\widetilde{\Phi}_{21}^{-1}}
}}
\end{equation}
where the last line yields the fiber $\varpi_{12}^{-1}(\zeta )$ of the local smoothings $\varpi_{12}:\mathcal{V}_{12}\to\Delta$ over $\zeta\in\Delta$.

\subsection{Local smoothings of $X_i\cup X_j\cup X_k$ around $D_{ij}\cup D_{jk}\cup D_{ki}$}\label{subsec:smoothing_triple}
Here we suppose $T_{123}\neq\emptyset$ and consider local smoothings of $X_1\cup X_2\cup X_3$ around $D_{12}\cup D_{23}\cup D_{31}$.
Indices $i,j,k$ will take $1$, $2$, or $3$, while $\ell$ will take all possible values besides $1,2,3$.
For local smoothings of $X_{n_1}\cup X_{n_2}\cup X_{n_3}$ around $D_{n_1n_2}\cup D_{n_2n_3}\cup D_{n_3n_1}$ for
general $n_1,n_2,n_3$ with $n_1<n_2<n_3$, we will be done if we replace subscripts $1,2,3$ and $i,j,k$ with $n_1,n_2,n_3$ and $n_i,n_j,n_k$, respectively.
For later convenience, we will use $\epsilon_{ij}=(j-i)/\norm{j-i}$ as before,
and the Levi-Civita symbol $\epsilon_{ijk}=\epsilon_{ij}\epsilon_{ik}\epsilon_{jk}$,
so that we have $\epsilon_{n_in_j}=\epsilon_{ij}$ and $\epsilon_{n_in_jn_k}=\epsilon_{ijk}$.
Define $\nu_{ij}$ by $\nu_{ij}=\sum_{k=1}^3k\norm{\epsilon_{ijk}}$,
that is, $\nu_{ij}\in\{ 1,2,3\}$ is the unique number such that $\epsilon_{ij\nu_{ij}}\neq 0$.
Also, let $D'_{ij}=D_{ij}\setminus T_{ij}$ and $N'_{ij}=\restrict{N_{ij}}{D'_{ij}}$ as in Section $\ref{subsec:coordinates}$.

Recall that we have a holomorphic volume form $\Omega_i$ on $X_i\setminus D_i$ with a local representation around $D_{ij}$ in $\eqref{eq:Omega_i_Lambda_ij}$,
and a holomorphic volume form $\psi_{D_{ij}}=\residue_{D_{ij}}\Omega_i$ in $\eqref{eq:psi_D_ij}$
and a smooth Hermitian form $\omega_{D_{ij}}$ in $\eqref{eq:omega_D_ij}$ on $D'_{ij}$.
We define a smooth complex volume form $\Omega_{ij}^\infty$
and a smooth Hermitian form $\omega_{ij}^\infty$ on $N'_{ij}\setminus D'_{ij}$ by
\begin{equation}\label{eq:Oomega_ij_triple}
\begin{aligned}
\Omega_{ij}^\infty&=\pi_{ij}^*\psi_{D_{ij}}\wedge\del t_{ij},\\
\omega_{ij}^\infty&=\frac{\sqrt{3}}{2}\pi_{ij}^*\omega_{D_{ij}}+\frac{\I}{\sqrt{3}}\del t_{ij}\wedge\delbar t_{ij}\\
&=\frac{\I}{\sqrt{3}}\left\{\frac{3}{4}\pi_{ij}^*\left(\psi_{D_{ij}}\wedge\overline{\psi}_{D_{ij}}\right) +\del t_{ij}\wedge\delbar t_{ij}\right\} .
\end{aligned}
\end{equation}

Then as in Lemma $\ref{lem:SU(2)_double}$, we see that
$(\Omega_{ij}^\infty ,\omega_{ij}^\infty)$ defines an $\SU (2)$-structure on $N'_{ij}\setminus D'_{ij}$
such that $\Omega_{ij}^\infty$ is holomorphic and the associated metric is cylindrical.
In particular, $\Omega_{ij}^\infty$ and $\omega_{ij}^\infty$ are locally represented
on $\pi_{ij}^{-1}(U'_{ij,\beta})\setminus D'_{ij}$ for $\beta\in\Lambda_{ijk}\subset\Lambda_{ij}^{(2)}$ as
\begin{equation}\label{eq:Oomega_ij_triple_local}
\begin{aligned}
\Omega_{ij}^\infty =&-\sigma_{ijk}\frac{\der x_{ij,\beta}}{x_{ij,\beta}}\wedge\frac{\der y_{ij,\beta}}{y_{ij,\beta}},\\
\omega_{ij}^\infty =&\frac{\I}{2\sqrt{3}}\Biggl\{\frac{\der x_{ij,\beta}\wedge\der\overline{x}_{ij,\beta}}{\norm{x_{ij,\beta}}^2}
+\frac{\der y_{ij,\beta}\wedge\der\overline{y}_{ij,\beta}}{\norm{y_{ij,\beta}}^2}\\
&\phantom{\frac{\I}{2\sqrt{3}}\Biggr\{}+\left(\frac{\der x_{ij,\beta}}{x_{ij,\beta}}+\frac{\der y_{ij,\beta}}{y_{ij,\beta}}\right)
\wedge\left(\frac{\der\overline{x}_{ij,\beta}}{\overline{x}_{ij,\beta}}+\frac{\der\overline{y}_{ij,\beta}}{\overline{y}_{ij,\beta}}\right)\Biggr\}
\end{aligned}
\end{equation}
where we used in $\eqref{eq:Oomega_ij_triple}$
\begin{equation*}
\pi_{ij}^*\psi_{D_{ij}}=\sigma_{ijk}\frac{\der x_{ij,\beta}}{x_{ij,\beta}}\quad\text{and}\quad
\del t_{ij}=-\frac{1}{2}\frac{\der x_{ij,\beta}}{x_{ij,\beta}}-\frac{\der y_{ij,\beta}}{y_{ij,\beta}}
\quad\text{for }\beta\in\Lambda_{ijk}\subset\Lambda_{ij}^{(2)},
\end{equation*}
which follows from $\eqref{eq:psi_D_ij}$ and $\eqref{eq:t_ij_triple}$ respectively.
As in Section $\ref{subsec:smoothing_double}$, we can also regard $(\Omega_{ij}^\infty ,\omega_{ij}^\infty)$
as an $\SU (2)$-structure on $W_{ij}\setminus D_i\subset X_i$ via $\Phi_{ij}$.

We see from $\eqref{eq:zw}$, $\eqref{eq:zw_approx}$, $\eqref{eq:relation_xy}$ and $\eqref{eq:Oomega_ij_triple_local}$ that
\begin{gather}
\norm{\Omega_i-\Omega_{ij}^\infty}=O(e^{-t_{ij}/2}),\quad\norm{\omega_{ij}-\omega_{ij}^\infty}=O(e^{-t_{ij}/2})
\quad\textrm{on }U'_{i,\alpha}\textrm{ for }\alpha\in\Lambda_{ij}^{(1)},\textrm{ and}
\label{eq:asymptotic_SU(2)_triple_1}\\
\Omega_{ij}^\infty =\Omega_i =\Omega_{ik}^\infty ,\quad\omega_{ij}^\infty =\omega_{ij}=\omega_{ik}=\omega_{ik}^\infty\quad\textrm{on }
U'_{i,\beta}\textrm{ for }\beta\in\Lambda_{ijk}\subset\Lambda_{ij}^{(2)},\label{eq:asymptotic_SU(2)_triple_2}
\end{gather}
where $\omega_{ij}$ is the $(1,1)$-part of $\omega_{ij}^\infty$, normalized so that $\Omega_i\wedge\overline{\Omega}_i =2\omega_{ij}\wedge\omega_{ij}$,
and $\norm{\cdot}$ is measured by the cylindrical metric $g_{ij}^\infty$ associated with $\omega_{ij}^\infty$.
Letting $c>1$ and shrinking $U_{i,\alpha}$ for each $\alpha\in\Lambda_{ij}^{(1)}\cup\Lambda_{ij}^{(2)}$ by the coordinate transformation $w_{ij,\alpha}\mapsto cw_{ij,\alpha}$
in conditions (B), (G) and (C), (H) if necessary, we may assume that $(\Omega_i,\omega_{ij})$ is an $\SU (2)$-structure on $W_{ij}\setminus D_i$.
Thus, we have the associated Hermitian metric $g_{ij}$ on $W_{ij}\setminus D_i$ such that
\begin{equation*}
g_{ij}=g_{ij}^\infty =g_{ik}^\infty =g_{ik}\quad\textrm{on }
U'_{i,\beta}\textrm{ for }\beta\in\Lambda_{ijk}\subset\Lambda_{ij}^{(2)}.
\end{equation*}
We also see from $\eqref{eq:pullback_t_ji}$ and $\eqref{eq:Oomega_ij_triple}$ that
\begin{equation}\label{eq:pullback_Oomega_triple}
h_{ij,\zeta}^*\Omega_{ji}^\infty =\Omega_{ij}^\infty,\quad h_{ij,\zeta}^*\omega_{ji}^\infty =\omega_{ij}^\infty
\quad\text{for all }\zeta\in\Delta^{\! *}.
\end{equation}

Now we construct a family of local smoothings of $X_1\cup X_2 \cup X_3$ around $D_{12}\cup D_{23}\cup D_{31}$.
The construction consists of the following steps.
\renewcommand{\labelenumi}{\emph{Step }$\arabic{enumi}$.}
\begin{enumerate}
\item Fix $\epsilon\leqslant e^{-1}$ and let $\Delta =\Delta_\epsilon =\set{\zeta\in\C|\norm{\zeta}<\epsilon}$ be a domain in $\C$.
Following Section $\ref{subsec:smoothing_double}$, we construct a family of local smoothings $\varpi'_{ij}:\mathcal{V}'_{ij}\to\Delta$
of $N'_{ij}\cup N'_{ji}$ in $N'_{ij}\oplus N'_{ji}$ around $D'_{ij}=N'_{ij}\cap N'_{ji}$
such that $\varpi_{ij}^{\prime -1}(0)=V'_{ij}\cup V'_{ji}$, where $V'_{ij}\subset N'_{ij}$ is given in $\eqref{eq:V_ij}$.
This gives a local model of a family of smoothings of $X_i\cup X_j$ around $D'_{ij}\subset X_i\cap X_j$.
We see that $\varpi_{ij}^{\prime -1}(\zeta )$ for $\zeta\in\Delta^{\! *}$ is obtained by gluing together $t_{ij}^{-1}(0,2T_\zeta )\subset V'_{ij}$ and
$t_{ji}^{-1}(0,2T_\zeta )\subset V'_{ji}$ using the diffeomorphism $h_{ij,\zeta}$ given by $\eqref{eq:h_ij}$,
where $T_\zeta$ is defined by $\eqref{eq:T_zeta}$.
Also, $\varpi_{ij}^{\prime -1}(\zeta )$ has an $\SU (2)$-structure which is induced from $N'_{ij}$ and $N'_{ji}$.\\

\item Let $\varpi'_{ij,\beta}:\mathcal{V}'_{ij,\beta}\to\Delta$ for $\beta\in\Lambda_{ij}^{(2)}$ be the restriction of $\varpi'_{ij}$
to $\mathcal{V}'_{ij,\beta}\subset V'_{ij,\beta}\oplus V'_{ji,\beta}$, where $V'_{ij,\beta}$ is given in $\eqref{eq:component_V_ij}$.
Then we can extend $\varpi'_{ij,\beta}:\mathcal{V}'_{ij,\beta}\to\Delta$ for each $\beta\in\Lambda_{ij}^{(2)}$
to $\varpi_{ij,\beta}:\mathcal{V}_{ij,\beta}\to\Delta$ so that $\mathcal{V}_{ij,\beta}$ is defined in $V_{ij,\beta}\oplus V_{ji,\beta}$,
where $V_{ij,\beta}$ is given in $\eqref{eq:component_V_ij}$.
By replacing each $\varpi'_{ij,\beta}$ with $\varpi_{ij,\beta}$ in $\varpi'_{ij}$,
we obtain a family of local smoothings $\varpi_{ij}:\mathcal{V}_{ij}\to\Delta$ of $N_{ij}\cup N_{ji}$ in $N_{ij}\oplus N_{ji}$ around $D_{ij}$.\\

\item For $\C^3=\{\bm{u}=(u^1,u^2,u^3)\}$, let $H_i=\set{\bm{u}\in\C^3|u^i=0}$ be hyperplanes in $\C^3$,
and let $L_{ij}=H_i\cap H_j$ be double curves.
Then for $\beta\in\Lambda_{123}$ we construct a family of local smoothings $\varpi_{123,\beta}:\mathcal{V}_{123,\beta}\to\Delta$
of an SNC complex surface $H_1\cup H_2\cup H_3$ in $\C^3$ around $L_{12}\cup L_{23}\cup L_{31}$,
such that $\mathcal{V}_{123,\beta}$ is an open neighborhood of the triple point $\bm{0}$ in $\C^3$ and
$\varpi_{123,\beta}^{-1}(0)=(H_1\cup H_2\cup H_3)\cap\mathcal{V}_{123}$.
Also, following Example $\ref{ex:tetra_in_CP3}$, we define an $\SU (2)$-structure on each fiber $\varpi_{123,\beta}^{-1}(\zeta )$.\\

\item We define a diffeomorphism $\Psi_{ij,\beta}:\mathcal{V}_{123,\beta}\to\mathcal{V}_{ij,\beta}\subset\mathcal{V}_{ij}$ for $\beta\in\Lambda_{123}$
which is compatible with the projections to $\Delta$ and preserves the $\SU (2)$-structures on the corresponding fibers.
Then the transition functions $\gamma_{i'j',ij,\beta}=\Psi_{i'j',\beta}\circ\Psi_{ij,\beta}^{-1}:\mathcal{V}_{ij,\beta}\to\mathcal{V}_{i'j',\beta}$ satisfy a cocycle condition.
Thus, we can glue together $\mathcal{V}_{12}$, $\mathcal{V}_{23}$, $\mathcal{V}_{31}$, and $\mathcal{V}_{123,\beta}$ for all $\beta\in\Lambda_{123}$
using the diffeomorphisms $\Psi_{12,\beta}$, $\Psi_{23,\beta}$, and $\Psi_{31,\beta}$
to obtain a family of local smoothing $\varpi_{123}:\mathcal{V}_{123}\to\Delta$ of $N_{12}\cup N_{23}\cup N_{31}$ around $D_{12}\cup D_{23}\cup D_{31}$.
This gives a local model of a family of smoothings of $X_1\cup X_2\cup X_3$ around $D_{12}\cup D_{23}\cup D_{31}$.\\

\item We define an injective diffeomorphism $\widetilde{\Phi}_{ij}:\mathcal{V}_{ij}\setminus (N_{ji}\setminus D_{ij})\to W_{ij}\times\Delta$,
where $W_{ij}=\Phi_{ij}(V_{ij})$ is a neighborhood of $D_{ij}$ in $X_i$.
Let $\mathcal{W}_{ij}=\smallcup_{\zeta\in\Delta}W_{ij}^{T_\zeta +1}\times\{\zeta\}$ and 
$\mathcal{X}_i=\smallcup_{\zeta\in\Delta}X_i^{T_\zeta +1}\times\{\zeta\}$ be as defined in $\eqref{eq:XX-WW}$,
where $W_{ij}^T$ and $X_i^T$ are defined in $\eqref{eq:W-X^T}$.
Then we can define injective diffeomorphisms $\widetilde{\Phi}_{ij}^{-1}\cup\widetilde{\Phi}_{ik}^{-1}:
\mathcal{W}_{ij}\cup\mathcal{W}_{ik}\to\mathcal{V}_{ij}\cup\mathcal{V}_{ik}\subset\mathcal{V}_{123}$.
Thus by gluing together $\mathcal{X}_1,\mathcal{X}_2,\mathcal{X}_3$ and $\mathcal{V}_{123}$
along $\mathcal{W}_{ij}\cup\mathcal{W}_{ik}\subset\mathcal{X}_{i}$ for all triples $(i,j,k)$ with $\epsilon_{ijk}=1$
using the injective diffeomorphisms $\widetilde{\Phi}_{ij}^{-1}\cup\widetilde{\Phi}_{ik}^{-1}$,
we obtain the desired family of local smoothings of $X_1\cup X_2\cup X_3$ around $D_{12}\cup D_{23}\cup D_{31}$.
\end{enumerate}

\begin{step}\label{step:3-1}
Fix $\epsilon\leqslant e^{-1}$ and let $\Delta =\Delta_\epsilon =\set{\zeta\in\C|\norm{\zeta}<\epsilon}$ be a domain in $\C$.
Let $V'_{ij}$ be the neighborhood of $D'_{ij}$ in $N'_{ij}$ defined in $\eqref{eq:V_ij}$.
Then a family of local smoothings $\varpi'_{ij}:\mathcal{V}'_{ij}\to\Delta_\epsilon$ of $N'_{ij}\cup N'_{ji}$ over $\Delta_\epsilon$ is given by
\begin{align}
&\mathcal{V}'_{ij}=\Set{(x_{ij,\alpha},y_{ij,\alpha},y_{ji,\alpha})\in V'_{ij}\oplus V'_{ji}|\substack{\begin{aligned}
&y_{ij,\alpha}y_{ji,\alpha}\in\Delta\textrm{ if }\alpha\in\Lambda_{ij}^{(1)},\textrm{ and}\\
&x_{ij,\alpha}y_{ij,\alpha}y_{ji,\alpha}\in\Delta\textrm{ if }\alpha\in\Lambda_{ij}^{(2)}\end{aligned}}},\label{eq:VV'_ij}\\
&\varpi'_{ij}(x_{ij,\alpha},y_{ij,\alpha},y_{ji,\alpha})=
\begin{dcases}
y_{ij,\alpha}y_{ji,\alpha}&\textrm{if }\alpha\in\Lambda_{ij}^{(1)},\\
x_{ij,\alpha}y_{ij,\alpha}y_{ji,\alpha}&\textrm{if }\alpha\in\Lambda_{ij}^{(2)}.
\end{dcases}\nonumber
\end{align}
Note that the projection $\varpi'_{ij}$ is well-defined according to condition (J) in Section $\ref{subsec:coordinates}$.
Let $p'_{ij}:\mathcal{V}'_{ij}\to N'_{ij}$ be the projection.
Then following the argument in Section $\ref{subsec:smoothing_double}$, we see that 
$p'_{ij,\zeta}$ is a diffeomorphism on $\varpi_{ij}^{\prime -1}(\zeta)$ for $\zeta\in\Delta^{\! *}$,
while on $\varpi_{ij}^{\prime -1}(0)=V'_{ij}\cup V'_{ji}$ we have
\begin{equation}\label{eq:p'_ij,0}
p'_{ij,0}\text{ is }\begin{dcases}
\text{an identity map}&\text{on }V'_{ij},\\
\text{the projection map to }D'_{ij}&\text{on }V'_{ji}.\end{dcases}
\end{equation}
Also, in the same way as we obtained $\eqref{eq:relation_hp}$ and $\eqref{eq:t_12+t_21}$ in Section $\ref{subsec:smoothing_double}$,
we have for $\zeta\in\Delta^{\! *}$ that
\begin{gather}
h_{ij,\zeta}=p'_{ji,\zeta}\circ p_{ij,\zeta}^{\prime -1}\quad\text{on }\varpi_{ij}^{\prime -1}(\zeta ),\quad\text{and}\label{eq:relation_hp'_triple}\\
t_{ij}\circ p'_{ij,\zeta}+t_{ji}\circ p'_{ji,\zeta}=-\log\norm{\zeta}^2=2T_\zeta .\label{eq:t_ij+t_ji}
\end{gather}
Thus, it follows from $\eqref{eq:t_ij+t_ji}$ that $\varpi_{ij}^{\prime -1}(\zeta )$ for $\zeta\in\Delta^{\! *}$ is obtained
by gluing together $t_{ij}^{-1}(0,2T_\zeta )\subset N'_{ij}$ and $t_{ji}^{-1}(0,2T_\zeta)\subset N'_{ji}$ using the diffeomorphism $h_{ij,\zeta}$.

Now if we restrict the domain of $p'_{ij}$ to $\mathcal{V}'_{ij}\setminus (V'_{ji}\setminus D'_{ij})$ and correspondingly that of $p'_{ij,\zeta}$,
then we see from $\eqref{eq:p'_ij,0}$ that $p'_{ij,\zeta}$ is a diffeomorphism for all $\zeta\in\Delta$.
Thus, the $\SU (2)$-structure $(\Omega_{ij}^\infty ,\omega_{ij}^\infty )$ on $N_{ij}$
induces an $\SU (2)$-structure $(\Omega'_{ij,\zeta},\omega'_{ij,\zeta})$ on $\varpi_{ij}^{-1}(\zeta )$ as
\begin{equation*}
(\Omega'_{ij,\zeta},\omega'_{ij,\zeta})=p_{ij,\zeta}^{\prime *}(\Omega_{ij}^\infty ,\omega_{ij}^\infty )=\begin{dcases}
p_{ij,\zeta}^{\prime *}(\Omega_{ij}^\infty ,\omega_{ij}^\infty )&\text{on }\omega_{ij}^{\prime -1}(\zeta )\text{ for }\zeta\in\Delta^{\! *},\\
(\Omega_{ij}^\infty ,\omega_{ij}^\infty )&\text{on }V'_{ij}\subset\varpi_{ij}^{\prime -1}(0)\text{ for }\zeta =0,\end{dcases}
\end{equation*}
which satisfies
\begin{equation*}
(\Omega'_{ij,\zeta},\omega'_{ij,\zeta})=(\Omega'_{ji,\zeta},\omega'_{ji,\zeta})\quad\text{on }\varpi_{ij}^{\prime -1}(\zeta )\text{ for }\zeta\in\Delta^{\! *}
\end{equation*}
due to $\eqref{eq:pullback_Oomega_triple}$ and $\eqref{eq:relation_hp'_triple}$.
In particular, the induced $\SU (2)$-structure $(\Omega'_{ij,0},\omega'_{ij,0})\cup (\Omega'_{ji,0},\omega'_{ji,0})$
on $\varpi_{ij}^{\prime -1}(0)=V'_{ij}\cup V'_{ji}$ coincides with the original one 
$(\Omega_{ij}^\infty ,\omega_{ij}^\infty )\cup (\Omega_{ji}^\infty ,\omega_{ji}^\infty )$ on $N'_{ij}\cup N'_{ji}$.
\end{step}

\begin{step}
Let us study $\varpi'_{ij}:\mathcal{V}'_{ij}\to\Delta$ constructed above in more detail.
According to the definition of $V'_{ij}$ in $\eqref{eq:V_ij}$, and $\eqref{eq:zw_approx}$, $\eqref{eq:norm_ij_triple}$, we can decompose $\mathcal{V}'_{ij}$ as
\begin{align*}
\mathcal{V}'_{ij}=&\mathcal{V}_{ij}^{(1)}\cup\smallcup_{\beta\in\Lambda_{ij}^{(2)}}\mathcal{V}'_{ij,\beta},\quad\text{where we set}\\
\mathcal{V}_{ij}^{(1)}=&\Set{\bm{v}_{ij,\alpha}\in V_{ij}^{(1)}\oplus V_{ji}^{(1)}|\alpha\in\Lambda_{ij}^{(1)},y_{ij,\alpha}y_{ji,\alpha}\in\Delta},
\quad\text{and}\\
\mathcal{V}'_{ij,\beta}=&\set{\bm{v}_{ij,\beta}\in V'_{ij,\beta}\oplus V'_{ji,\beta}|x_{ij,\beta}y_{ij,\beta}y_{ji,\beta}\in\Delta}\\
=&\Set{\bm{v}_{ij,\beta}\in\C^3|\substack{\begin{aligned}
&0<\norm{x_{ij,\beta}}<1,\norm{y_{ij,\beta}}<1,\norm{y_{ji,\beta}}<1\text{ and}\\
&x_{ij,\beta}y_{ij,\beta}y_{ji,\beta}\in\Delta\end{aligned}}}\quad\text{for }\beta\in\Lambda_{ij}^{(2)}.
\end{align*}
Thus, we can extend $\mathcal{V}'_{ij,\beta}$ smoothly to $\mathcal{V}_{ij,\beta}$ defined by
\begin{align*}
\mathcal{V}_{ij,\beta}&=\set{\bm{v}_{ij,\beta}\in V_{ij,\beta}\oplus V_{ji,\beta}|x_{ij,\beta}y_{ij,\beta}y_{ji,\beta}\in\Delta}\\
&=\Set{\bm{v}_{ij,\beta}\in\C^3|\substack{\begin{aligned}
&\norm{x_{ij,\beta}}<1,\norm{y_{ij,\beta}}<1,\norm{y_{ji,\beta}}<1\text{ and}\\
&x_{ij,\beta}y_{ij,\beta}y_{ji,\beta}\in\Delta\end{aligned}}},
\end{align*}
where $V_{ij,\beta}$ is defined in $\eqref{eq:component_V_ij}$.
Hence, replacing $\mathcal{V}'_{ij,\beta}$ with $\mathcal{V}_{ij,\beta}$ for all $\beta\in\Lambda_{ij}^{(2)}$ in $\mathcal{V}'_{ij}$
and extending the projection $\varpi'_{ij}$ correspondingly,
we obtain a family of local smoothings $\varpi_{ij}:\mathcal{V}_{ij}\to\Delta$ of $N_{ij}\cup N_{ji}$ in $N_{ij}\oplus N_{ji}$ around $D_{ij}$.
This turns out to be the same as replacing $V'_{ij}\oplus V'_{ji}$ in $\eqref{eq:VV'_ij}$ with $V_{ij}\oplus V_{ji}$, where $V_{ij}$ is defined in $\eqref{eq:V_ij}$.
Since $\mathcal{V}_{ij}$ and $\mathcal{V}_{ji}$ only differ by the order of $V_{ij}$ and $V_{ji}$ in the definition,
we will identify $\mathcal{V}_{ij}$ and $\mathcal{V}_{ji}$.

Now define $H_{ij,\beta}\subset\mathcal{V}_{ij,\beta}$ for $\beta\in\Lambda_{ij}^{(2)}$ and $H_{ij}\subset\mathcal{V}_{ij}$ by
\begin{equation*}
H_{ij,\beta}=\Set{\bm{v}_{ij,\beta}\in V_{ij,\beta}\oplus V_{ji,\beta}|x_{ij,\beta}=0}\quad\text{and}\quad
H_{ij}=\smallcup_{\beta\in\Lambda_{ij}^{(2)}}H_{ij,\beta},
\end{equation*}
so that we have $\varpi_{ij}^{-1}(0)=V_{ij}\cup V_{ji}\cup H_{ij}$.
Let $p_{ij}:\mathcal{V}_{ij}\to V_{ij}$ be the projection and $p_{ij,\zeta}=\restrict{p_{ij}}{\varpi_{ij}^{-1}(\zeta )}$.
If we restrict the domain of $p_{ij}$ to $\mathcal{V}_{ij}\setminus ((V_{ji}\setminus D_{ij})\cup (H_{ij}\setminus\pi_{ij}^{-1}(T_{ij})))$,
and correspondingly that of $p_{ij,\zeta}$, then $p_{ij,\zeta}$ is a diffeomorphism for all $\zeta\in\Delta$.
Then in the same way as in Step $\ref{step:3-1}$, we can define an $\SU (2)$-structure
$(\Omega_{ij,\zeta},\omega_{ij,\zeta})=p_{ij,\zeta}^*(\Omega_{ij}^\infty ,\omega_{ij}^\infty )$ on $\varpi_{ij}^{-1}(\zeta )$ for $\zeta\in\Delta$
satisfying $(\Omega_{ij,\zeta},\omega_{ij,\zeta})=(\Omega_{ji,\zeta},\omega_{ji,\zeta})$ on $\varpi_{ij}^{-1}(\zeta )$ for $\zeta\in\Delta^{\! *}$.
In particular, we have $(\Omega_{ij,0},\omega_{ij,0})=(\Omega_{ij}^\infty ,\omega_{ij}^\infty )$ on $V_{ij}\subset\varpi_{ij}^{-1}(0)$,
but at this point $H_{ij}\subset\varpi_{ij}^{-1}(0)$ does not have an $\SU (2)$-structure.
\end{step}

\begin{step}
Let $\C^3=\{\bm{u}=(u^1,u^2,u^3)\}$, $H_i=\set{\bm{u}\in\C^3|u^i=0}$, and $L_{ij}=H_i\cap H_j$ as above.
For $\beta\in\Lambda_{123}$, we construct a family of local smoothings $\varpi_{123,\beta}:\mathcal{V}_{123,\beta}\to\Delta$
of an SNC complex surface $H_1\cup H_2\cup H_3$ in $\C^3$ around $L_{12}\cup L_{23}\cup L_{31}$ by
\begin{align*}
&\mathcal{V}_{123,\beta}=\Set{\bm{u}\in\C^3|\norm{u^1}<1,\norm{u^2}<1,\norm{u^3}<1\text{ and }u^1u^2u^3\in\Delta},\quad\text{ and}\\
&\varpi_{123,\beta}(u^1,u^2,u^3)=u^1u^2u^3.
\end{align*}
Then $\mathcal{V}_{123,\beta}$ is an open neighborhood of the triple point $\bm{0}$ in $\C^3$
and $\varpi_{123,\beta}^{-1}(0)=(H_1\cup H_2\cup H_3)\cap\mathcal{V}_{123,\beta}$.
To define an $\SU (2)$-structure on each fiber $\varpi_{123}^{-1}(\zeta )$ over $\zeta\in\Delta$,
let $\eta^i=\der u^i/u^i$ be a meromorphic $1$-form on $\C^3$, and
consider a meromorphic volume form $\Omega_{H_i,\beta}$ and a singular Hermitian form $\omega_{H_i,\beta}$ on $H_i$ defined by
\begin{align*}
\Omega_{H_i,\beta}&=-\sigma_{ijk}\eta^j\wedge\eta^k,\\
\omega_{H_i,\beta}&=\frac{\I}{2\sqrt{3}}\left\{\eta^j\wedge\overline{\eta}^j
+\eta^k\wedge\overline{\eta}^k+(\eta^j+\eta^k)\wedge\left(\overline{\eta}^j+\overline{\eta}^k\right)\right\},
\end{align*}
where $j$ and $k$ are chosen so that $\epsilon_{ijk}\neq 0$.
Then it is easy to check that $(\Omega_{H_i,\beta},\omega_{H_i,\beta})$ defines an $\SU (2)$-structure on $H_i$,
i.e., $\Omega_{H_i,\beta}\wedge\overline{\Omega}_{H_i,\beta}=2\omega_{H_i,\beta}\wedge\omega_{H_i,\beta}$.
Note that $\Omega_{H_i,\beta}$ is induced from a meromorphic volume form $\Omega_{\C^3,\beta}$ by
\begin{align*}
\Omega_{H_i,\beta}&=\residue_{H_i}{\Omega_{\C^3,\beta}},\quad\text{where}\\
\Omega_{\C^3,\beta}&=\sigma_{123}\,\eta^1\wedge\eta^2\wedge\eta^3=\sigma_{ijk}\,\eta^i\wedge\eta^j\wedge\eta^k\quad\text{for }\epsilon_{ijk}\neq 0.
\end{align*}

To define an $\SU (2)$-structure $(\Omega_{123,\beta ,\zeta},\omega_{123,\beta ,\zeta})$
on each fiber $\varpi_{123,\beta}^{-1}(\zeta)$ over $\zeta\in\Delta$, we define projections $p_i:\C^3\to H_i$ by
\begin{equation*}
p_i(u^1,u^2,u^3)=(u^j,u^k),
\end{equation*}
where $j$ and $k$ are determined so that $j<k$ and $\epsilon_{ijk}\neq 0$.
Also, we define $p_{i,\zeta}=\restrict{p_i}{\varpi_{123,\beta}^{-1}(\zeta )}$.
Suppose $\zeta\in\Delta^{\! *}$.
Then we see that
\begin{equation*}
p_{i,\zeta}:\varpi_{123,\beta}^{-1}(\zeta )\to\Set{\bm{u}\in\mathcal{V}_{123,\beta}|\norm{u^j}<1,\norm{u^k}<1,\norm{\zeta}<\norm{u^ju^k}}
\end{equation*}
is a diffeomorphism for $\epsilon_{ijk}\neq 0$, which satisfies
\begin{equation}\label{eq:pullback_Oomega_H}
(\Omega_{H_j,\beta},\omega_{H_j,\beta})=(p_{i,\zeta}\circ p_{j,\zeta}^{-1})^*(\Omega_{H_i,\beta},\omega_{H_i,\beta}).
\end{equation}
Meanwhile, if $\zeta =0$, then on each irreducible component of $\varpi_{123,\beta}^{-1}(0)$ we have that
\begin{equation*}
p_{i,0}\text{ is }\begin{dcases}
\text{an identity map}&\text{on }H_i\cap\mathcal{V}_{123,\beta},\\
\text{the projection map to }L_{i\ell}\subset H_i&\text{on }H_\ell\cap\mathcal{V}_{123,\beta}\text{ for }\ell\neq i.\end{dcases}
\end{equation*}
Thus to make $p_{i,\zeta}$ diffeomorphic on $\varpi_{123}^{-1}(\zeta )$ for all $\zeta\in\Delta$, we restrict the domain of $p_i$
to $\mathcal{V}_{123}\setminus ((H_j\setminus L_{ij})\cup (H_k\setminus L_{ik}))$, where $\epsilon_{ijk}=1$.
Consequently, by $\eqref{eq:pullback_Oomega_H}$, we have a well-defined $\SU (2)$-structure
$(\Omega_{123,\beta ,\zeta},\omega_{123,\beta ,\zeta})=p_{i,\zeta}^*(\Omega_{H_i,\beta},\omega_{H_i,\beta})$, which is also written as
\begin{equation*}
(\Omega_{123,\beta ,\zeta},\omega_{123,\beta ,\zeta})=\begin{dcases}
p_{i,\zeta}^*(\Omega_{H_i,\beta},\omega_{H_i,\beta})&\text{on }\varpi_{123,\beta}^{-1}(\zeta )\text{ for }\zeta\in\Delta^{\! *},\\
(\Omega_{H_i,0},\omega_{H_i,0})&\text{on }H_i\cap\varpi_{123,0}^{-1}(0)\text{ for }\zeta =0.\end{dcases}
\end{equation*}
We have an alternative expression of $\omega_{123,\beta ,\zeta}$ as
\begin{equation}\label{eq:omega_123,beta,zeta_alt}
\omega_{123,\beta ,\zeta}=\restrict{\omega_{\C^3}}{\varpi_{123,\beta}^{-1}(\zeta )}\quad\text{for }\zeta\in\Delta^{\! *},
\end{equation}
where $\omega_{\C^3}$ is a singular Hermitian form on $\C^3$ defined by
\begin{equation*}
\omega_{\C^3}=\frac{\I}{2\sqrt{3}}\left(\eta^1\wedge\overline{\eta}^1+\eta^2\wedge\overline{\eta}^2+\eta^3\wedge\overline{\eta}^3\right) ,
\end{equation*}
and $\eqref{eq:omega_123,beta,zeta_alt}$ follows from $\eta^1+\eta^2+\eta^3=0$ on $\varpi_{123,\beta}^{-1}(\zeta )$ for $\zeta\in\Delta^{\! *}$.
\end{step}

\begin{step}
Let $k=\nu_{ij}$ and $\beta\in\Lambda_{123}$.
By making the local coordinate $x_{ij,\beta}$ of $D_{ij}$ correspond to $u^k$ in $\mathcal{V}_{123,\beta}$ and using $\eqref{eq:relation_xy}$,
we can define a diffeomorphism $\Psi_{ij,\beta}:\mathcal{V}_{123,\beta}\to\mathcal{V}_{ij,\beta}\subset\mathcal{V}_{ij}$ by
\begin{equation*}
\Psi_{ij,\beta}:(u^1,u^2,u^3)\mapsto(x_{ij,\beta},y_{ij,\beta},y_{ji,\beta})=\begin{dcases}
(u^k,u^j,u^i)\quad\textrm{if }\epsilon_{ijk}=1,\\
(u^k,u^i,u^j)\quad\textrm{if }\epsilon_{ijk}=-1,\end{dcases}
\end{equation*}
which is compatible with the projections to $\Delta$ and maps $\varpi_{ij,\beta}^{-1}(\zeta )$ diffeomorphically onto $\varpi_{123,\beta}^{-1}(\zeta )$.
In particular, under the identification of $\mathcal{V}_{ij}$ and $\mathcal{V}_{ji}$, $\Psi_{ij,\beta}$ and $\Psi_{ji,\beta}$ are also identical.
Also, since the transition functions $\gamma_{i'j',ij,\beta}:\mathcal{V}_{ij,\beta}\to\mathcal{V}_{i'j',\beta}$ for $\epsilon_{ij},\epsilon_{i',j'}\neq 0$ are given by
\begin{equation*}
\gamma_{i'j',ij,\beta}=\Psi_{i'j',\beta}\circ\Psi_{ij,\beta}^{-1},
\end{equation*}
these satisfy a cocycle condition
\begin{equation*}
\gamma_{ij,jk,\beta}\circ\gamma_{jk,ki,\beta}\circ\gamma_{ki,ij,\beta}=\id_{\mathcal{V}_{ij,\beta}}\quad\text{for all }i,j,k\text{ with }\epsilon_{ijk}\neq 0.
\end{equation*}

Let $\Psi_{ij,\beta ,\zeta}$ be the restriction of $\Psi_{ij,\beta}$ on $\varpi_{123}^{-1}(\zeta )$.
Then we can check that
\begin{equation*}
\Psi_{ij,\beta,\zeta}^*(\Omega_{ij,\zeta},\omega_{ij,\zeta})=(\Omega_{123,\beta ,\zeta},\omega_{123,\beta ,\zeta})
\quad\text{for all }\zeta\in\Delta ,
\end{equation*}
so that $\Psi_{ij,\beta ,\zeta}$ preserves the $\SU (2)$-structures
on $\varpi_{123,\beta}^{-1}(\zeta )$ and $\varpi_{ij,\beta}^{-1}(\zeta )\subset\varpi_{ij}^{-1}(\zeta )$ for all $\zeta\in\Delta$.
Hence noting that $\mathcal{V}_{ij}\cap\mathcal{V}_{jk}=\smallcup_{\beta\in\Lambda_{ijk}}\mathcal{V}_{123,\beta}$,
we can glue together $\mathcal{V}_{123,\beta}$ and $\mathcal{V}_{ij}(=\mathcal{V}_{ji}$) for all $\beta\in\Lambda_{123}$ and $(i,j)=(1,2),(2,3),(3,1)$
using the diffeomorphisms $\Psi_{ij,\beta}:\mathcal{V}_{123,\beta}\to\mathcal{V}_{ij,\beta}\subset\mathcal{V}_{ij}$ for all $\beta\in\Lambda_{123}$
to obtain a family of local smoothings $\varpi_{123}:\mathcal{V}_{123}\to\Delta$ of $N_{12}\cup N_{23}\cup N_{31}$ around $D_{12}\cup D_{23}\cup D_{31}$.
This gluing procedure for constructing $\mathcal{V}_{123}$ is diagrammed as follows:
\begin{equation}\label{eq:gluing_V_123}
\vcenter{\xymatrix@R=1pc{
&\mathcal{V}_{ij}\ar@{}[r]|-*{\supset}&\mathcal{V}_{ij,\beta}\ar[dddl]_-{\gamma_{jk,ij}}^\circlearrowleft&&\\
&&&&\\
&&\mathcal{V}_{123,\beta}\ar[uu]|-{\Psi_{ij,\beta}}\ar[dl]^{\Psi_{jk,\beta}}\ar[dr]_{\Psi_{ki,\beta}}&&\\
\mathcal{V}_{jk}\ar@{}[r]|-*{\supset}&\mathcal{V}_{jk,\beta}\ar[rr]_-{\gamma_{ki,jk,\beta}}^\circlearrowleft&&\mathcal{V}_{ki,\beta}\ar[uuul]_-{\gamma_{ij,ki,\beta}}^\circlearrowleft&\mathcal{V}_{ki}\ar@{}[l]|-*{\subset},
}}
\end{equation}
where all arrows are diffeomorphisms.
By this diagram, we can regard each $\mathcal{V}_{ij}$ as an open submanifold of $\mathcal{V}_{123}$.
\end{step}

\begin{step}
As in Section $\ref{subsec:smoothing_double}$, we have an injective diffeomorphism
$\widetilde{\Phi}_{ij}:\mathcal{V}_{ij}\setminus (N_{ji}\setminus D_{ji})\to W_{ij}\times\Delta$ by 
$\widetilde{\Phi}_{ij}=(\widetilde{\Phi}_{ij}^1,\widetilde{\Phi}_{ij}^2)=(\Phi_{ij}\circ p_{ij},\varpi_{ij})$.
Then we see from $\eqref{eq:t_ij+t_ji}$ that
\begin{equation*}
\widetilde{\Phi}_{ij}^1(\varpi_{ij}^{-1}(\zeta ))=t_{ij}^{-1}(0,2T_\zeta )\quad\text{for }\zeta\in\Delta^{\! *}.
\end{equation*}
Meanwhile $V_{ij}$, which is an irreducible component of $\varpi_{ij}^{-1}(0)=V_{ij}\cup V_{ji}$, is mapped onto $W_{ij}$ by $\Phi_{ij}\circ p_{ij}$.
Now let $W_{ij}^T, X_i^T$ be as defined in $\eqref{eq:W-X^T}$, and $\mathcal{W}_{ij}=\smallcup_{\zeta\in\Delta}W_{ij}^{T_\zeta +1}\times\{\zeta\}$, 
$\mathcal{X}_i=\smallcup_{\zeta\in\Delta}X_i^{T_\zeta +1}\times\{\zeta\}$ be as defined in $\eqref{eq:XX-WW}$.
Then we see that
\begin{equation}\label{eq:image_Phi_ij_triple}
\Image\widetilde{\Phi}_{ij}\supset\mathcal{W}_{ij},
\end{equation}
and under the identification of $\mathcal{V}_{ij}$ and $\mathcal{V}_{ji}$, the map
$\widetilde{\Phi}_{ij}^{-1}\cup\widetilde{\Phi}_{ji}^{-1}:\mathcal{W}_{ij}\cup\mathcal{W}_{ji}\to\mathcal{V}_{ij}$ is surjective.

For $\beta\in\Lambda_{ijk}$ and $\zeta\in\Delta$, we have a diffeomorphism
\begin{equation*}
\xymatrix@R=0pt{
\varpi_{ij,\beta}^{-1}(\zeta )\ar[r]^{\gamma_{ik,ij,\beta}=\Psi_{ik,\beta}\circ\Psi_{ij,\beta}^{-1}}&\varpi_{ik,\beta}^{-1}(\zeta )\\
\din&\din\\
(x_{ij,\beta},y_{ij,\beta},y_{ji,\beta})\ar@{|->}[r]&(x_{ik,\beta},y_{ik,\beta},y_{ki,\beta})\mathrlap{{}=(y_{ij,\beta},x_{ij,\beta},y_{ji,\beta}),}}
\phantom{{}=(y_{ij,\beta},x_{ij,\beta},y_{ji,\beta}),}
\end{equation*}
under which $\widetilde{\Phi}_{ij}^1(x_{ij,\beta},y_{ij,\beta},y_{ji,\beta})$ and $\widetilde{\Phi}_{ik}^1(x_{ik,\beta},y_{ik,\beta},y_{ki,\beta})$
determine the same point in $U_{i,\beta}$ because of $\eqref{eq:relation_xy}$.
Since we have $W_{ij}\cap W_{ik}=\smallcup_{\beta\in\Lambda_{ijk}}U_{i,\beta}$, 
$\widetilde{\Phi}_{ij}$ and $\widetilde{\Phi}_{ik}$ are consistent as injective diffeomorphisms from $\mathcal{V}_{ij}$ and $\mathcal{V}_{ik}$
to $(W_{ij}\cup W_{ik})\times\Delta\subset X_i\times\Delta$.
Thus, $\widetilde{\Phi}_{ij}\cup\widetilde{\Phi}_{ik}:
(\mathcal{V}_{ij}\setminus (N_{ji}\setminus D_{ij}))\cup (\mathcal{V}_{ik}\setminus (N_{ki}\setminus D_{ik}))\to (W_{ij}\cup W_{ik})\times\Delta$
is well-defined as an injective diffeomorphism.
Meanwhile, we see from $\eqref{eq:image_Phi_ij_triple}$ that
$\Image (\widetilde{\Phi}_{ij}\cup\widetilde{\Phi}_{ik})\supset\mathcal{W}_{ij}\cup\mathcal{W}_{ik}$,
so that $(\widetilde{\Phi}_{ij}\cup\widetilde{\Phi}_{ik})^{-1}=\widetilde{\Phi}_{ij}^{-1}\cup\widetilde{\Phi}_{ik}^{-1}:
\mathcal{W}_{ij}\cup\mathcal{W}_{ik}\to
(\mathcal{V}_{ij}\setminus (N_{ji}\setminus D_{ij}))\cup (\mathcal{V}_{ik}\setminus (N_{ki}\setminus D_{ik}))\subset\mathcal{V}_{123}$
is also an injective diffeomorphism which is compatible with the projections to $\Delta$.

Hence, regarding $\mathcal{V}_{ij}(=\mathcal{V}_{ji})$ as an open submanifold of $\mathcal{V}_{123}$ for all pairs $(i,j)$ with $\epsilon_{ij\nu_{ij}}=1$,
we can glue together $\mathcal{X}_1,\mathcal{X}_2,\mathcal{X}_3$, and $\mathcal{V}_{123}$
along $\mathcal{W}_{ij}\cup\mathcal{W}_{ik}\subset\mathcal{X}_{i}$ for all $i$ and $j,k$ with $\epsilon_{ijk}=1$
using the injective diffeomorphisms $(\widetilde{\Phi}_{ij}\cup\widetilde{\Phi}_{ik})^{-1}:\mathcal{W}_{ij}\cup\mathcal{W}_{ik}\to
(\mathcal{V}_{ij}\setminus (N_{ji}\setminus D_{ij}))\cup (\mathcal{V}_{ik}\setminus (N_{ki}\setminus D_{ik}))\subset\mathcal{V}_{123}$,
to obtain the desired family of local smoothings of $X_1\cup X_2\cup X_3$ around $D_{12}\cup D_{23}\cup D_{31}$.
This gluing procedure for $i$ is diagrammed as follows:
\begin{equation}\label{eq:gluing_triple}
\vcenter{\xymatrix@R=0pt@C=0pt{
\mathcal{X}_{i}&&&&&\mathcal{V}_{ij}\ar@{}[r]|-*{\subset}&\mathcal{V}_{123}\ar@{}[r]|-*{\supset}&\mathcal{V}_{ik}\\
\dsubset&&&&&\dsubset&&\dsubset\\
\mathcal{W}_{ij}\cup\mathcal{W}_{ik}\ar@{^{(}->}[rrrrr]^-{\widetilde{\Phi}_{ij}^{-1}\cup\widetilde{\Phi}_{ik}^{-1}}
&&&&&(\mathcal{V}_{ij}\setminus (N_{ji}\setminus D_{ij}))\ar@{}[rr]|-*{\cup}&&(\mathcal{V}_{ik}\setminus (N_{ki}\setminus D_{ik}))\\
\dsubset&&&&&\dsubset&&\dsubset\\
(W_{ij}^{T_\zeta +1}\cup W_{ik}^{T_\zeta +1})\times\{\zeta\}\ar@{^{(}->}[rrrrr]^-{\widetilde{\Phi}_{ij}^{-1}\cup\widetilde{\Phi}_{ik}^{-1}}
&&&&&\varpi_{ij}^{-1}(\zeta )\setminus (V_{ji}\setminus D_{ij})\ar@{}[rr]|-*{\cup}&&\varpi_{ik}^{-1}(\zeta )\setminus (V_{ki}\setminus D_{ik})
}}
\end{equation}
where $\mathcal{V}_{ij}$ is constructed by diagram $\eqref{eq:gluing_double}$ with subscripts $1,2$ replaced with $i,j$ using Steps $1,2$,
and $\mathcal{V}_{123}$ is constructed by diagram $\eqref{eq:gluing_V_123}$.
Also, the last line of $\eqref{eq:gluing_triple}$ for all $i$ yields the fiber $\varpi_{123}^{-1}(\zeta )$ of the local smoothings $\varpi_{123}:\mathcal{V}_{123}\to\Delta$ over $\zeta\in\Delta$.
\end{step}
Note that at this point we have only constructed differential geometric smoothings, and thus each fiber over $\zeta\in\Delta$ 
is only given as a smooth manifold without a complex structure.
In Section $\ref{subsec:existence}$, we shall construct on each fiber over $\zeta\in\Delta$ a complex structure which depends continuously on $\zeta$.

\subsection{Existence of holomorphic volume forms on global smoothings}\label{subsec:existence}
Here we shall prove Theorem $\ref{thm:smoothing}$.
\begin{proof}[Proof of Theorem $\ref{thm:smoothing}$]
Let $\Delta =\Delta_\epsilon =\set{\zeta\in\C|\norm{\zeta}<\epsilon}$ for $\epsilon\leqslant e^{-3}$
and let $T_\zeta$ for $\zeta\in\Delta$ be as in $\eqref{eq:T_zeta}$, so that we have $T_\zeta\in (T_\epsilon ,\infty ]$ with $T_\epsilon\geqslant 3$.
Let $\mathcal{X}_i$ and $\mathcal{V}_{ij},\mathcal{V}_{ijk}$ be as defined in $\eqref{eq:XX-WW}$
and Sections $\ref{subsec:smoothing_double}$ and $\ref{subsec:smoothing_triple}$ respectively.
Then for all pairs $(n_1,n_2)$ with $n_1<n_2$, $\Lambda_{n_1n_2}\neq\emptyset$ and $\Lambda_{n_1n_2}^{(2)}=\emptyset$,
we glue together $\mathcal{X}_{n_1}$, $\mathcal{X}_{n_2}$ and $\mathcal{V}_{n_1n_2}$ according to diagram $\eqref{eq:gluing_double}$,
in which subscripts $1,2$ and $i,j$ are replaced with $n_1,n_2$ and $n_i,n_j$, respectively.
At the same time, for all triples $(n_1,n_2,n_3)$ with $n_1<n_2<n_3$ and $\Lambda_{n_1n_2n_3}\neq\emptyset$,
we glue together $\mathcal{X}_{n_1}$, $\mathcal{X}_{n_2}$, $\mathcal{X}_{n_3}$ and $\mathcal{V}_{n_1n_2n_3}$
according to $\eqref{eq:gluing_triple}$, in which subscripts $i,j,k$ are replaced with $n_i,n_i,n_k$.
As a result, we obtain a family of global smoothings $\varpi :\mathcal{X}\to\Delta$ of $X$,
which satisfies parts (a) and (b) of Theorem \ref{thm:smoothing}.

For each double curve $D_{ij}\subset X_i$, we obtained the Hermitian form $\omega_{ij}$ on $W_{ij}\setminus D_i$
which defines an $\SU (2)$-structure together with $\Omega_i$ and satisfies
$\omega_{ij}=\omega_{ik}$ on $W_{ij}\cap W_{ik}\setminus D_i=\smallcup_{\beta\in\Lambda_{ijk}}U'_{i,\beta}$.
We also obtained the Hermitian metrics $g_{ij}$ on $W_{ij}\setminus D_i$ associated with the $\SU (2)$-structure $(\Omega_i,\omega_{ij})$
such that $g_{ij}=g_{ik}$ on $W_{ij}\cap W_{ik}\setminus D_i$.
Then we have the following two results.
\begin{lemma}
There exists a Hermitian form $\omega_i$ on $X_i\setminus D_i$ such that $(\Omega_i,\omega_i)$ defines an $\SU (2)$-structure on $X_i\setminus D_i$
and we have
\begin{equation*}
\omega_i=\omega_{ij}\quad\text{on }(W_{ij}\setminus D_i)\setminus W_{ij}^1
=\widetilde{t}_{ij}^{-1}[1,\infty )\setminus\smallcup_{k\in I_{ij}}\widetilde{t}_{ik}^{-1}(0,1)\text{ for all }j\in I_i.
\end{equation*}
\end{lemma}
\begin{proof}
We shall follow the argument in \cite{Doi09}, Section $3.3$.
Let $\omega_i^1$ be a Hermitian form on the compact submanifold $\overline{X_i^1}$ of $X_i\setminus D_i$
normalized so that $2\omega_i^1\wedge\omega_i^1=\Omega_i\wedge\overline{\Omega}_i$, and $g_i^1$ be the associated Hermitian metric.
Then gluing together $g_i^1$ and $g_{ij}$ for all $j\in I_i$ using a cut-off function 
which takes $1$ on $\overline{X_i^0}=X_i\setminus\smallcup_{j\in I_i}W_{ij}$ and $0$ outside $X_i^1$,
we have a Hermitian metric $\widehat{g}_i$ on $X_i\setminus D_i$ such that
\begin{equation*}
\widehat{g}_i=\begin{dcases}
g_i^1&\text{on }\overline{X_i^0},\\
g_{ij}&\text{on }\widetilde{t}_{ij}^{-1}[1,\infty )\setminus\smallcup_{k\in I_{ij}}\widetilde{t}_{ik}^{-1}(0,1)\text{ for all }j\in I_i.
\end{dcases}
\end{equation*}
Letting $\widehat{\omega}_i$ be the associated Hermitian form,
we have $2\lambda_i\widehat{\omega}_i\wedge\widehat{\omega}_i=\Omega_i\wedge\overline{\Omega}_i$
for some positive function $\lambda_i$ on $X_i\setminus D_i$ such that $\lambda_i\equiv 1$ on $\overline{X_i^0}$ and outside $X_i^1.$
Then $\omega_i=\lambda_i^{1/2}\widehat{\omega}_i$ gives the desired Hermitian form.
\end{proof}
\begin{lemma}
There exists a smooth complex $1$-form $\xi_{ij}$ on $W_{ij}$ such that
\begin{equation*}
\Omega_i-\Omega_{ij}^\infty =\der\xi_{ij},\quad\text{and }\norm{\nabla^m\xi_{ij}}=O(e^{-t_{ij}/2})
\quad\text{for all }m\geqslant 0.
\end{equation*}
In particular, we have $\xi_{ij}=0$ on $U_{i,\beta}$ for $\beta\in\Lambda_{ij}^{(2)}$.
\end{lemma}
\begin{proof}
The first assertion follows from $\eqref{eq:asymptotic_SU(2)_double}$, $\eqref{eq:asymptotic_SU(2)_triple_1}$, and \cite{Doi09}, Proposition $3.4$.
Then the second assertion follows from $\eqref{eq:asymptotic_SU(2)_triple_2}$.
\end{proof}
Hence for $\zeta\neq 0$, we can define a pair $(\Omega_{i,\zeta},\omega_{i,\zeta})$ of a smooth complex and a real $2$-form on $X_i\setminus D_i$ by
\begin{align*}
\Omega_{i,\zeta}&=\Omega_i-\der\smallsum_{j\in I_i}(1-\rho_{T_\zeta -1}(t_{ij}))\xi_{ij},\\
\omega_{i,\zeta}&=\omega_i+\smallsum_{j\in I_i}(1-\rho_{T_\zeta -1}(t_{ij}))(\omega_{ij}^\infty -\omega_i),
\end{align*}
where $\rho_T(x)=\rho (x-T+1)$ is a translation of the cut-off function $\rho :\R\to [0,1]$ with 
\begin{equation*}
\rho (x)=\begin{dcases}
1&\text{if }x\leqslant 0,\\
0&\text{if }x\geqslant 1,\end{dcases}
\quad\text{so that}\quad\rho_T (x)=\begin{dcases}
1&\text{if }x\leqslant T-1,\\
0&\text{if }x\geqslant T.\end{dcases}
\end{equation*}
Then $\Omega_{i,\zeta}$ is $\der$-closed, and under the decomposition $\eqref{eq:W_ij}$ of $W_{ij}$, we have
\begin{equation}\label{eq:Oomega_i,zeta}
(\Omega_{i,\zeta },\omega_{i,\zeta})=\begin{dcases}
(\Omega_i,\omega_i)&\text{on }\overline{X_i^{T_\zeta -2}},\\
(\Omega_{ij}^\infty ,\omega_{ij}^\infty)&\text{on }\widetilde{t}_{ij}^{-1}[T_\zeta -1,\infty )\cap W_{ij}^{(1)}.\\
(\Omega_{ij}^\infty ,\omega_{ij}^\infty)=(\Omega_{ik}^\infty ,\omega_{ik}^\infty)
&\text{on }(\widetilde{t}_{ij}^{-1}[1,\infty )\cup\widetilde{t}_{ik}^{-1}[1,\infty ))
\cap U_{i,\beta}\text{ for }\beta\in\Lambda_{ijk},
\end{dcases}
\end{equation}
Thus, $(\Omega_{i,\zeta},\omega_{i,\zeta})$ is an $\SU (2)$-structure on $X_i\setminus D_i$ except on $\widetilde{t}_{ij}^{-1}(T_\zeta -2,T_\zeta -1)\cap W_{ij}^{(1)}$.
By $\eqref{eq:asymptotic_SU(2)_double}$ and $\eqref{eq:asymptotic_SU(2)_triple_1}$, we have an estimate
\begin{equation}\label{ineq:est_SU(2)_failure}
\norm{(\Omega_{i,\zeta},\omega_{i,\zeta})-(\Omega_{ij}^\infty ,\omega_{ij}^\infty )}\leqslant C_ie^{-T_\zeta /2}
\quad\text{on }\widetilde{t}_{ij}^{-1}(T_\zeta -2,T_\zeta -1)\cap W_{ij}^{(1)},
\end{equation}
where the norm is measured by the associated metric $g_{ij}^\infty$, and $C_i$ is a constant which is independent of $T_\zeta$.

Now recall that $X_\zeta$ is constructed as a differentiable manifold by the gluing procedures
according to the last lines of diagrams $\eqref{eq:gluing_double}$ and $\eqref{eq:gluing_triple}$ around all double lines and triple points.
Then we see from $\eqref{eq:Oomega_i,zeta}$ that the pairs $(\Omega_{i,\zeta},\omega_{i,\zeta})$ of $2$-forms on $X_i^{T_\zeta +1}$ for all $i$ extend
to a pair $(\widetilde{\Omega}_\zeta ,\widetilde{\omega}_\zeta )$ on all of $X_\zeta$
so that $\widetilde{\Omega}_\zeta$ is $\der$-closed, and $\widetilde{\Omega}_\zeta ,\widetilde{\omega}_\zeta )$ coincides with
the $\SU (2)$-structure $(\Omega_{ij,\zeta},\omega_{ij,\zeta})$ on the image of
\begin{equation*}
\widetilde{t}_{ij}^{-1}(T_\zeta -1,T_\zeta +1)\cup\smallcup_{k\in I_{ij}}\widetilde{t}_{ij}^{-1}[1,T_\zeta +1)\cap\widetilde{t}_{ik}^{-1}[1,T_\zeta +1)
\end{equation*}
under $\widetilde{\Phi}_{ij}^{-1}$ in $\omega_{ij}^{-1}(\zeta )$.
Now set $T_{\rho_*}=2\log (\max_i\{ C_i\}/\rho_*)$ and assume $T_\zeta >T_{\rho_*}$ hereafter.
Then we have $C_ie^{-T_\zeta /2}<\rho_*$ for all $i$,
and thus by $\eqref{ineq:est_SU(2)_failure}$, Lemma $\ref{lem:rho*}$, and Definition $\ref{def:Theta}$,
we can define an $\SU (2)$-structure $(\psi_\zeta ,\kappa_\zeta )=\Theta (\widetilde{\Omega}_\zeta ,\widetilde{\omega}_\zeta )$ on $X_\zeta$.
Let $\phi_\zeta =\widetilde{\Omega}_\zeta -\psi_\zeta$, so that we have $\der\psi_\zeta +\der\phi_\zeta =0$.
\begin{lemma}\label{lem:est_phi-kappa}
We have estimates
\begin{equation*}
\Norm{\phi_\zeta}_{L^p}\leqslant Ce^{-T_\zeta /2},\quad\Norm{\der\phi_\zeta}_{L^p}\leqslant Ce^{-T_\zeta /2},
\quad\text{and }\Norm{\der\kappa_\zeta}_{C^0}\leqslant C
\end{equation*}
for some positive constants $C$ which are independent of $T_\zeta$, where the norms are measured by the metric $g_\zeta$ on $X_\zeta$
associated with the $\SU (2)$-structure $(\psi_\zeta ,\kappa_\zeta )$.
\end{lemma}
\begin{proof}
See \cite{Doi09}, Proposition $3.6$.
\end{proof}
Under a diffeomorphism $\mathcal{X}\setminus X\simeq M\times\Delta^{\! *}$,
the family $\set{(\psi_\zeta, \kappa_\zeta )|\zeta\in\Delta^{\! *}}$ of $\SU(2)$-structures on $M$ is smooth with respect to
both $p\in M$ and $\zeta\in\Delta^{\! *}$.
Since $(M,g_\zeta )$ has a linear volume growth in $T_\zeta$, we see from Lemma $\ref{lem:est_phi-kappa}$ that
\begin{equation*}
\Norm{\der\kappa_\zeta}_{L^p}\leqslant C T_\zeta^{1/p},
\end{equation*}
where $g_\zeta$ is the metric associated with $(\psi_\zeta, \kappa_\zeta )$.
To distinguish $\epsilon$ in Theorem $\ref{thm:existence}$ from that used for the radius of $\Delta =\Delta_\epsilon$,
we denote the former by $\epsilon'$, while the latter remains the same.
Let $\epsilon'=e^{-\gamma T_\zeta}$ for $0<\gamma\leqslant 1/6$ and $T_*(\rho )=\max\{-(\log\epsilon_*)/\gamma ,T_{\rho_*},3\}$,
where $\epsilon_*=\epsilon_*(\rho )$ is used in Theorem $\ref{thm:existence}$.
If $T_\zeta >T_*(\rho )$, so that $\epsilon'<\epsilon_*(\rho )$,
then by Theorem $\ref{thm:existence}$ there exists a unique $\eta_\zeta$ with $\Norm{\eta_\zeta}_{C^0}\leqslant\rho$ for each $\zeta\in\Delta^{\! *}$
such that $\Omega_\zeta =\Theta_1(\psi_\zeta +\eta_\zeta ,\kappa_\zeta)$ is a $\der$-closed $\SL (2,\C )$-structure on $X_\zeta$.
Thus letting $\epsilon =\log T_*(\rho )$ for some $\rho <\rho_*$, we have $T_\zeta >T_*(\rho )$ for all $\zeta\in\Delta$, and part (c) of Theorem $\ref{thm:smoothing}$ holds.

Since one can see that as $\zeta'\to\zeta\in\Delta^{\! *}$, $\eta_{\zeta'}$ converges to $\eta_\zeta$ in $L^8_1(\wedge^2_-T^*M,g_\zeta )\hookrightarrow C^{0,1/2}(\wedge^2_-T^*M,g_\zeta )$,
the resulting family $\set{\Omega_\zeta|\zeta\in\Delta^{\! *}}$ of $\der$-closed $\SL (2,\C )$-structures on $M$ is \emph{continuous} with respect to $\zeta$.
Also, for $T_\zeta >T_*(\rho )$ we have an estimate on $X_i^{T_\zeta +1}\subset X_\zeta =\varpi^{-1}(\zeta )$ as
\begin{equation}\label{ineq:Omega_zeta-Omega_i}
\begin{aligned}
\Norm{\Omega_\zeta -\Omega_i}_{C^0_i}&\leqslant
\Norm{\Omega_\zeta -\psi_\zeta}_{C^0_i}+\Norm{\psi_\zeta -\widetilde{\Omega}_i}_{C^0_i}+\Norm{\widetilde{\Omega}_\zeta -\Omega_i}_{C^0_i}\\
&<(1+C')\left(\Norm{\eta_\zeta}_{C^0_\zeta}+\Norm{\psi_\zeta -\widetilde{\Omega}_i}_{C^0_\zeta}+\Norm{\widetilde{\Omega}_\zeta -\Omega_i}_{C^0_\zeta}\right)\\
&<(1+C')(\rho +\rho+C''e^{-T_\zeta /2})\quad\text{for all }i,
\end{aligned}
\end{equation}
for some positive constants $C'$ and $C''$ independent of $\rho$, where the $C^0$-norms $\Norm{\cdot}_{C^0_i}$ and $\Norm{\cdot}_{C^0_\zeta}$ are measured on $X_i^{T_\zeta +1}$
by the metrics $g_i$ and $g_\zeta$ associated with the $\SU (2)$-structures $(\Omega_i,\omega_i)$ and $(\psi_\zeta ,\kappa_\zeta )$, respectively.
Then redefining $T_*(\rho )$ so that $C''e^{-T_*(\rho )/2}\leqslant\rho$ in $\eqref{ineq:Omega_zeta-Omega_i}$, we have
\begin{equation*}
\Norm{\Omega_\zeta -\Omega_i}_{C^0_i}<3(1+C')\rho\quad\text{on }X_i^{T_\zeta +1}\text{ for all }i\text{ and }\zeta\in\Delta^{\! *}\text{ with }T_\zeta >T_*(\rho ),
\end{equation*}
which implies that the complex structure $I_\zeta$ on $X_\zeta$
induced by the $\SL (2,\C)$-structure $\Omega_\zeta$
converges uniformly as $\zeta\to 0$ to the original complex structure on the central fiber $X=\varpi^{-1}(0)$
outside the singular locus $D=\smallcup_iD_i$.
Hence, the continuity in part (d) is proved.
This completes the proof of Theorem $\ref{thm:smoothing}$.
\end{proof}

\subsection{Examples of $d$-semistable SNC complex surfaces with trivial canonical bundle without triple points}\label{subsec:ex_SNC_double}
Here we shall give some examples of $d$-semistable SNC complex surfaces with trivial canonical bundle without triple points,
which are smoothable to complex tori, primary Kodaira surfaces, or $K3$ surfaces due to Theorem $\ref{thm:smoothing}$.
Our examples are based on those given in \cite{Doi09}, Examples $5.1$ and $5.3$.
It is worth mentioning that although the classical smoothability result of Friedman cannot be applied to the SNC complex surfaces $X$ given in Example $\ref{ex:tori_Kodaira}$ 
because we have $H^1(X_i,\mathcal{O}_{X_i})\neq 0$ for some irreducible component $X_i$ of $X$, 
the modern techniques for smoothings (\cite{FFR19}, Theorem $1.1$ and \cite{CLM19}, Corollary $5.15$) are applicable.
A typical example to see this generalization is given as follows.
Let $X=X_1\cup X_2$ be a $3$-dimensional SNC complex manifold such that $X_1$ and $X_2$ are two copies of $\C P^3$, and $D=X_1\cap X_2$ is a quartic surface.
Then $X$ is not $d$-semistable, but $\mathcal{T}_X^1\cong N_{D/X_1}\otimes N_{D/X_2}$ is generated by global sections (see \cite{FFR19}, Example $1.3$).
Consequently, $X$ is smoothable to Calabi-Yau threefolds due to \cite{FFR19}, Theorem $1.1$.
A similar argument works for SNC complex surfaces, and so with Examples $\ref{ex:K3_double}$--$\ref{ex:K3_typeII}$.
Thus, it seems that some of the examples given here are already known to the experts.
However, it is still valuable to list these examples in the rest of this article 
because they are helpful to illustrate properties and technical features of conditions (i)--(iii) in Theorem $\ref{thm:smoothing}$.
Particularly, Example $\ref{ex:Fujita}$ provides a nice example to see that condition (ii) in Theorem $\ref{thm:smoothing}$
is only a necessary condition for the canonical bundle of an SNC complex surface to be trivial.

\begin{example}\label{ex:K3_double}
For $d\in \{ 0,\pm 1,\pm 2,\pm 3\}$, let $C$ and $C_d$ be smooth curves of degree $3$ and $3-d$ in $\C P^2$, respectively,
such that $C_d$ intersects $C$ transversely, which we regard as trivially satisfied for $C_3=\emptyset$.
Then $C$ is an anticanonical divisor on $\C P^2$.
Let $\pi_d:X_C(d)\to\C P^2$ be the blow-up of $\C P^2$ at the points $C\cap C_d$, and $D$ be the proper transform of $C$ in $X_C(d)$,
so $\pi_d$ maps $D$ isomorphically to $C$.
Also, let $E_d$ be an exceptional divisor $\pi_d^{-1}(C\cap C_d)$.
Then since we have
\begin{equation}\label{eq:bl-up_div}
K_{X_C(d)}=\pi_d^*K_{\C P^2}\otimes [E_d]\quad\text{and}\quad [D]=\pi_d^*[C]\otimes [E_d]^{-1}
\end{equation}
(see, e.g., \cite{GH94}, p. $608$), we calculate the canonical bundle $K_{X_C(d)}$ of $X_C(d)$ as
\begin{equation}\label{eq:antican_K3}
K_{X_C(d)}=\pi_d^*K_{\C P^2}\otimes (\pi_d^*[C]\otimes [D]^{-1})=\pi_d^*(K_{\C P^2}\otimes [C])\otimes [D]^{-1}\cong [D]^{-1},
\end{equation}
so that $D$ is an anticanonical divisor on $X_C(d)$.

Meanwhile, using the adjunction formula and $\eqref{eq:bl-up_div}$, we have
\begin{equation}\label{eq:bl-up_normal}
N_{D/X_C(d)}\cong\restrict{[D]}{D}=\restrict{\pi_d^*[C]}{D}\otimes\restrict{[E_d]^{-1}}{D}.
\end{equation}
Since $\restrict{\pi_d}{D}:D\to C$ is an isomorphism and $C_d$ intersects $C$ transversely, we calculate $\restrict{[E_d]}{D}$ as
\begin{equation}\label{eq:div_E}
\restrict{[E_d]}{D}=[E_d\cap D]_D=\pi_d^*([C\cap C_d]_C)=\pi_d^*(\restrict{[C_d]}{C}),
\end{equation}
where $\left[ E_d\cap D\right]_D$ and $\left[ C\cap C_d\right]_C$ mean that we consider these as divisors on $D$ and $C$, respectively.
Similarly, letting $C'$ be a cubic curve which intersects $C$ transversely, so that $C'$ is linearly equivalent to $C$, we find that
\begin{equation}\label{eq:pullback_div_C}
\restrict{(\pi_d^*[C])}{D}\cong\restrict{(\pi_d^*[C'])}{D}=\pi_d^*(\restrict{[C']}{C}).
\end{equation}
Thus, putting $\eqref{eq:div_E}$ and $\eqref{eq:pullback_div_C}$ into $\eqref{eq:bl-up_normal}$ gives that
\begin{equation}\label{eq:normal_K3}
\begin{aligned}
N_{D/X_C(d)}\cong&\restrict{[D]}{D}\cong\restrict{(\pi_d^*[C']\otimes [E_d]^{-1})}{D}=\pi_d^*\left(\restrict{([C']\otimes [C_d]^{-1})}{C}\right)\\
\cong&\restrict{([C']\otimes [C_d]^{-1})}{C}\cong\mathcal{O}_C(3)\otimes\mathcal{O}_C(d-3)=\mathcal{O}_C(d).
\end{aligned}
\end{equation}

Now for the above cubic curve $C$ and $d=0,1,2,3$, let $X_1=X_C(-d),X_2=X_C(d)$
and $D_1,D_2$ be the corresponding anticanonical divisors isomorphic to $C$.
Consider an SNC complex surface $X=X_1\cup X_2$ obtained by the gluing isomorphism $D_1\xrightarrow{\pi_{-d}}C\xrightarrow{\pi_d^{-1}}D_2$
and local embeddings into $\C^3$.
Then $\eqref{eq:normal_K3}$ and $\eqref{eq:antican_K3}$ lead to conditions (i) and (ii) of Theorem $\ref{thm:smoothing}$, respectively, while (iii) is obvious.
Thus by Theorem \ref{thm:smoothing}, we obtain a family $\varpi :\mathcal{X}\to\Delta$ of global smoothings of $X$ with $\varpi^{-1}(0)=X$.
Calculating the Euler characteristics of $X_1,X_2$, and $D$ as
\begin{gather*}
\chi (X_1)=\chi (\C P^2)+\# (C\cap C_{-d})\cdot (\chi (\C P^1)-\chi (\text{point}))=12+3d,\\
\chi (X_2)=12-3d\quad\text{and}\quad\chi (D)=0,
\end{gather*}
we see that the Euler characteristic of the general fiber $X_\zeta =\varpi^{-1}(\zeta )$ of $\varpi$ over $\zeta\in\Delta^{\! *}$ is given by
\begin{equation*}
\chi (X_\zeta )=\chi (X)=\chi (X_1)+\chi (X_2)-\chi (D)=24,
\end{equation*}
where we used the invariance of homology under continuous deformations.
Hence, we see from the Enriques-Kodaira classification of compact complex surfaces with trivial canonical bundle \cite{BHPV}
that $X_\zeta$ is a $K3$ surface.
\end{example}

\begin{example}\label{ex:tori_Kodaira}
For $d\in\Z$, let $Y_{\C P^2}(d)=\mathbb{P}(\mathcal{O}_{\C P^2}\oplus\mathcal{O}_{\C P^2}(d))$ be a $\C P^1$-bundle over $\C P^2$,
and $D_{\C P^2,0}=\{ ([\bm{z}],[1,0])\}$ and $D_{\C P^2,\infty}=\{ ([\bm{z}],[0,1])\}$ be the zero and the infinity section of $Y_{\C P^2}(d)$, respectively,
where $\bm{z}=(z^0,z^1,z^2)\in\C^3\setminus (0,0,0)$.
Then we have
\begin{align*}
Y_{\C P^2}(d)\setminus D_{\C P^2,0}&=\set{([\bm{z}],[\xi_\infty ,1])\in Y_{\C P^2}(d)}\cong\mathcal{O}_{\C P^2}(-d),\\
Y_{\C P^2}(d)\setminus D_{\C P^2,\infty}&=\set{([\bm{z}],[1,\xi_0])\in Y_{\C P^2}(d)}\cong\mathcal{O}_{\C P^2}(d),\quad\text{where}\\
\xi_0\xi_\infty&=1\quad\text{on }Y_{\C P^2}(d)\setminus (D_{\C P^2,0}\cup D_{\C P^2,\infty}),
\end{align*}
under which isomorphisms we have
\begin{align*}
D_{\C P^2,0}&\cong\set{([\bm{z}],\xi_0)\in\mathcal{O}_{\C P^2}(d)|\xi_0=0}\cong\C P^2,\\
D_{\C P^2,\infty}&\cong\set{([\bm{z}],\xi_\infty )\in\mathcal{O}_{\C P^2}(-d)|\xi_\infty =0}\cong\C P^2.
\end{align*}
Thus by the adjunction formula, we can calculate $N_{D_{\C P^2,0}/Y_{\C P^2}(d)}$ and $N_{D_{\C P^2,\infty}/Y_{\C P^2}(d)}$ as
\begin{align*}
N_{D_{\C P^2,0}/Y_{\C P^2}(d)}&\cong\restrict{[D_{\C P^2,0}]}{D_{\C P^2,0}}\cong\mathcal{O}_{\C P^2}(d)\quad\text{and}\\
N_{D_{\C P^2,\infty}/Y_{\C P^2}(d)}&\cong\restrict{[D_{\C P^2,\infty}]}{D_{\C P^2,\infty}}\cong\mathcal{O}_{\C P^2}(-d)
\end{align*}
by considering the transition functions.

Also, for a cubic curve $C$ in $\C P^2$, let $Y_C(d)=\restrict{Y_{\C P^2}(d)}{C}=\mathbb{P}(\mathcal{O}_C\oplus\mathcal{O}_C(d))$,
and let $D_0=D_{C,0}$ and $D_\infty =D_{C,\infty}$ be the zero and the infinity section of $Y_C(d)$, respectively.
Then in the same way as above, we calculate $N_{D_0/Y_C(d)}$ and $N_{D_\infty /Y_C(d)}$ as
\begin{equation}\label{eq:d-ss_Y_C}
N_{D_0/Y_C(d)}\cong\mathcal{O}_C(d)\quad\text{and}\quad N_{D_\infty/Y_C(d)}\cong\mathcal{O}_C(-d).
\end{equation}

Now consider a local coordinate system $\{ U_i,\bm{\zeta}_i\}$ on $\C P^2$ given by
\begin{gather*}
U_i=\set{[\bm{z}]\in\C P^2|z^i\neq 0}\quad\text{and}\quad\bm{\zeta}_i=(\zeta_i^0,\zeta_i^1,\zeta_i^2)\in\C^2\times\{ 1\} ,\quad\text{where we set}\\
\zeta_i^j=z^j/z^i,\quad\text{so that}\quad\zeta_i^i=1.
\end{gather*}
Also, let $(\bm{\zeta}_i,\xi_{0,i})$ and $(\bm{\zeta}_i,\xi_{\infty ,i})$ be coordinates which trivialize
$Y_{\C P^2}(d)\setminus D_{\C P^2,\infty}\cong\mathcal{O}_{\C P^2}(d)$ and $Y_{\C P^2}(d)\setminus D_{\C P^2,0}\cong\mathcal{O}_{\C P^2}(-d)$
over $U_i$, respectively, so that we have
\begin{equation*}
\xi_{0,i}=(\zeta_j^i)^d\xi_{0,j},\quad\xi_{\infty ,i}=(\zeta_j^i)^{-d}\xi_{\infty ,j},\quad\text{and}\quad\xi_{0,i}\xi_{\infty ,i}=1.
\end{equation*}
Letting $\psi_C$ be a holomorphic volume form on $C$, we can consistently define a meromorphic volume form $\Omega$ on $Y_C(d)$ 
with a single pole along $D_0\cup D_\infty$ by
\begin{equation}\label{eq:Omega_Y_C}
\Omega =\begin{dcases}
-\psi_C\wedge\restrict{\frac{\der\xi_{0,i}}{\xi_{0,i}}}{C}
&\text{on }\restrict{(Y_C(d)\setminus D_\infty )}{C\cap U_i}\cong\restrict{\mathcal{O}_C(d)}{C\cap U_i}\\
\psi_C\wedge\restrict{\frac{\der\xi_{\infty ,i}}{\xi_{\infty ,i}}}{C}
&\text{on }\restrict{(Y_C(d)\setminus D_0)}{C\cap U_i}\cong\restrict{\mathcal{O}_C(-d)}{C\cap U_i}.
\end{dcases}
\end{equation}
Thus, we see that $\Omega$ gives a trivialization
\begin{equation}\label{eq:antican_Y_C}
K_{Y_C(d)}\otimes [D_0]\otimes [D_\infty ]\cong\mathcal{O}_{Y_C(d)},
\end{equation}
so that $Y_C(d)$ has an anticanonical divisor $D=D_0+D_\infty$.
Also, $\Omega$ satisfies the relation
\begin{equation}\label{eq:res_Y_C}
\residue_{D_0}{\Omega}=\psi_C=-\residue_{D_\infty}{\Omega}.
\end{equation}

For $N\in\N$ and $i=1,\dots ,N$, let $X_i$ be a copy of $Y_C(d)$ with an anticanonical divisor $D_i=D_{i,0}+D_{i,\infty}$.
We construct an SNC complex surface $X=\smallcup_{i=1}^NX_i$ by gluing together $D_{i,\infty}$ and $D_{i+1,0}$ for all $i=1,\dots ,N$ using the isomorphisms
\begin{equation*}
D_{i,\infty}\to D_{i+1,0},\quad ([\bm{z}_i],[1,0])\mapsto ([\bm{z}_{i+1}],[0,1]),\quad\text{where we set }D_{N+1,0}=D_{1,0},
\end{equation*}
together with local embeddings into $\C^3$.
Then we see that $\eqref{eq:d-ss_Y_C}$, $\eqref{eq:antican_Y_C}$ and $\eqref{eq:res_Y_C}$ give
conditions (i), (ii) and (iii) of Theorem $\ref{thm:smoothing}$, respectively.
Thus by Theorem \ref{thm:smoothing}, we obtain a family $\varpi :\mathcal{X}\to\Delta$ of global smoothings of $X$ with $\varpi^{-1}(0)=X$.
One can show that the general fiber $X_\zeta =\varpi^{-1}(\zeta )$ for $\zeta\in\Delta^{\! *}$ is topologically $S_d\times S^1$, 
where $S_d$ is the $\U(1)$-bundle associated with the complex line bundle $\mathcal{O}_C(d)$, which has Betti numbers
\begin{equation*}
b_1(S_d)=b_2(S_d)=\begin{dcases}
3&\text{for }d=0,\\
2&\text{for }d\neq 0.
\end{dcases}
\end{equation*}
Consequently, the general fiber $X_\zeta$ has Betti numbers
\begin{equation*}
(b_1(X_\zeta ),b_2(X_\zeta ))=\begin{dcases}
(4,6)&\text{for }d=0,\\
(3,4)&\text{for }d\neq 0.
\end{dcases}
\end{equation*}
Hence by the classification of compact complex surfaces with trivial canonical bundle \cite{BHPV},
we see that the general fiber $X_\zeta$ is a complex torus for $d=0$ and a primary Kodaira surface for $d\neq 0$.
In particular, the central fiber $X$ for $d\neq 0$ cannot be K\"{a}hlerian because $b^1(X)=b^1(X_\zeta )=3$ is odd.
We remark that we can construct $Y_C(d)$ from a $\C P^1$-bundle $Y_{\C P^n}$ over $\C P^n$ of any complex dimension $n\geqslant 2$
and an elliptic curve $C$ embedded in $\C P^n$.
See also Example $\ref{ex:Fujita}$, in which we take $N=2$ and use a gluing map $D_{2,\infty}\to D_{1,0}$ not being an identity isomorphism.
\end{example}

\begin{example}\label{ex:K3_typeII}
Let $C$ be a cubic curve in $\C P^2$, and let $X_C(d)$ and $Y_C(d)$ be as in Examples $\ref{ex:K3_double}$ and $\ref{ex:tori_Kodaira}$, respectively.
For $N\geqslant 2$ and $d=0,1,2,3$, let $X_1=X_C(-d),X_N=X_C(d)$, and $X_2,\dots ,X_{N-1}$ be copies of $Y_C(d)$.
Let us denote by $D_i$ the anticanonical divisor on $X_i$ constructed in the above examples, where $D_i=D_{i,0}+D_{i,\infty}$ for $i=2,\dots ,N-1$.
Then we obtain an SNC complex surface $X=\smallcup_{i=1}^NX_i$
using gluing isomorphisms $D_1\to D_{2,0}$, $D_{i,\infty}\to D_{i+1,0}$ for $i=2,\dots ,N-1$, and $D_{N-1,\infty}\to D_N$.
In almost the same way as in Examples $\ref{ex:K3_double}$ and $\ref{ex:tori_Kodaira}$, we can apply Theorem $\ref{thm:smoothing}$
to obtain a family $\varpi :\mathcal{X}\to\Delta$ of global smoothings of $X$ with $\varpi^{-1}(0)=X$.
Since we have $\chi (X_\zeta )=\chi (X)=\chi (X_1)+\chi (X_N)=24$ as in Example $\ref{ex:K3_double}$ using $\chi (Y_C(d))=0$ and $\chi (C)=0$,
$X$ is smoothable to $K3$ surfaces.
Note that this example gives an explicit construction of $d$-semistable $K3$ surfaces of Type \rom{2} with any number $N$ of irreducible components.
\end{example}

Here we shall give some examples of SNC complex surfaces for which condition (ii) holds but the canonical bundle is \emph{not} trivial.
The following example for the case of $d=0$ and $k=1$ is due to K. Fujita \cite{Fuj21}, which we extend to all integers $d$ and $k=1,2,3$.
The proof given here is more explicit and elementary than the original one.
\begin{example}\label{ex:Fujita}
Let $C$ be an elliptic curve embedded in $\C P^n$ for some $n$, so $K_C$ is trivial.
We take $C$ so that $C$ is isomorphic to $\C /\Lambda$ with the standard coordinate $z$, 
where $\Lambda$ is the lattice in $\C$ generated by $1$ and $\I$.
We will identify $C$ with $\C /\Lambda$ below and use $z\in\C /\Lambda$ as a coordinate on $C$.
For $d\in\Z$, let $Y_C(d)$ be a $\C P^1$-bundle over $C$ obtained as in Example $\ref{ex:tori_Kodaira}$.
Then $Y_C(d)$ has an anticanonical divisor $D=D_0+D_\infty$ with $D_0,D_\infty\cong C$,
and a meromorphic volume form $\Omega$ given by $\eqref{eq:Omega_Y_C}$ with a single pole along $D_0$ and $D_\infty$
which trivializes $K_{Y_C(d)}\otimes [D_0]\otimes [D_\infty ]$, where we set $\psi_C=\der z$.

Now let $X_1$ and $X_2$ be two copies of $Y_C(d)$, and define an SNC complex surface $X=X_1\cup X_2$
using gluing isomorphisms $\tau_1:D_{1,\infty}\to D_{2,0}$ and $\tau_2:D_{2,\infty}\to D_{1,0}$,
together with local embeddings into $\C^3$, where $\tau_1$ and $\tau_2$ are given by
\begin{equation}\label{eq:tau}
\tau_1:z\mapsto z\quad\text{and}\quad\tau_2:z\mapsto (\I )^k z\quad\text{for }k=0,1,2,3.
\end{equation}
Note that if $k=0$ in $\eqref{eq:tau}$, the resulting SNC complex surface $X$ is the same as that obtained in Example $\ref{ex:tori_Kodaira}$.

As we saw in Example $\ref{ex:tori_Kodaira}$, $X$ satisfies condition (ii) of Theorem $\ref{thm:smoothing}$.
However, the following result shows that the canonical bundle $K_X$ of $X$ is \emph{not} trivial for $k=1,2,3$, while $K_X$ is trivial for $k=0$ as desired.
\begin{claim}
We have $H^0(X,K_X)\cong\C$ for $k=0$ and $H^0(X,K_X)=0$ for $k=1,2,3$.
\end{claim}
\begin{proof}
Applying Proposition $\ref{prop:gl.sect_K_X}$ to our example, we see that $H^0(X,K_X)$ is given by the kernel of the linear map
\begin{equation*}
\rho =\rho_1\oplus\rho_2:H^0(X_1,L_1)\oplus H^0(X_2,L_2)\to H^0(D_{1,\infty},L_{1,\infty})\oplus H^0(D_{2,\infty},L_{2,\infty}),
\end{equation*}
where $\rho_i$ is given by
\begin{equation*}
\vcenter{\xymatrix@R=0pt{
\quad\mathllap{\rho_i:}H^0(X_1,L_1)\oplus H^0(X_2,L_2)\ar[r]&H^0(D_{i,\infty},L_{i,\infty})\\
\quad\din&\din\\
\quad (s_1,s_2)\ar@{|->}[r]&\displaystyle\residue_{D_{i,\infty}}s_i+\tau_i^*\left(\residue_{D_{i',0}}s_{i'}\right)}}
\quad\text{for }i=1,2\text{ and }i'=3-i.
\end{equation*}
Noting $H^0(X_i,L_i)=\C\,\Omega ,H^0(D_{i,p},L_{i,p})=\C\,\der z$, and $\tau_i^*\der z=\der (\tau_i(z))$ for $i=1,2$ and $p=0,\infty$,
and using $\eqref{eq:res_Y_C},\eqref{eq:tau}$, we compute as
\begin{equation*}
\rho (c_1\Omega ,c_2\Omega )=\left( -c_1+c_2,(\I)^kc_1-c_2\right)\der z.
\end{equation*}
Hence, we find that
\begin{equation*}
H^0(X,K_X)\cong\Ker\rho =\begin{dcases}
\set{(c_1\Omega ,c_2\Omega )|c_1=c_2}\cong\C&\text{if }k=0,\\
\set{(c_1\Omega ,c_2\Omega )|c_1=c_2=0}=0&\text{if }k=1,2,3
\end{dcases}
\end{equation*}
as desired.
This completes the proof.
\end{proof}
We can modify the above example as follows.
For a modification, we can take any number $N\in\N$ of components $X_i=Y_C(d)$ and take $\tau_i:D_{i,\infty}\to D_{i+1,0}$
as $z\mapsto (\I )^{k_i}z$ with $k_i=0,1,2,3$ for $i=1,\dots ,N$ as in Example $\ref{ex:tori_Kodaira}$.
Then one can see that $H^0(X,K_X)\cong\C$ if $\sum_ik_i\equiv 0\mod 4$ and $H^0(X,K_X)=0$ otherwise.
For a further modification, we can take the lattice $\Lambda$ as general, which still satisfies $\Lambda =-\Lambda$.
In this case, we can take $\tau_i$ as $z\mapsto (-1)^{k_i}z$ with $k_i=0,1$.
Then similarly, one finds that $H^0(X,K_X)\cong\C$ if $\sum_ik_i$ is even, and $H^0(X,K_X)=0$ if $\sum_ik_i$ is odd.
\end{example}

\section{Examples of $d$-semistable $K3$ surfaces of Type \rom{3}}\label{sec:TypeIIIK3}
In this section, we provide several examples of $d$-semistable SNC complex surfaces $X$ with triple points
to which we can apply Theorem $\ref{thm:smoothing}$.
Furthermore, we show that all of our examples are $d$-semistable $K3$ surfaces of Type \rom{3}
by computing the Euler characteristic of the general fiber $X_\zeta$ of the resulting smoothings,
and using the result of the Enriques-Kodaira classification as well.

Let us recall Example $\ref{ex:tetra_in_CP3}$, where we considered an SNC complex surface $X$ in $\C P^3$
with four hyperplanes $X_0,\dots ,X_3$ as irreducible components.
We saw that in the special case where $X_i$ are given by $X_i=\{[\bm{z}]\in\C P^3|z^i=0\}$,
$X$ satisfies conditions (ii) and (iii) of Theorem $\ref{thm:smoothing}$, but does not satisfy (i), i.e., $X$ is \emph{not} $d$-semistable.
If we change the configuration of $X_0, \dots, X_3$ so that they do not have fourfold intersections,
then the number of triple points may change, but we see that $X$ still satisfies conditions (ii) and (iii).
To obtain a $d$-semistable SNC complex surface from $X$, we will use N.-H. Lee's criterion given in \cite{Lee19}
(see also Lemma $\ref{lem:CNC}$) and blow up $X$ at appropriate points in the double lines.

We will rename $X$, $X_i$, $D_{ij}$ in Example $\ref{ex:tetra_in_CP3}$ as $X'$, $H_i$, $L_{ij}$, respectively,
because we want to let $X=\smallcup_{i=0}^3X_i$ with $X_i\cap X_j=D_{ij}$ be the desired $d$-semistable SNC complex surface.
We will otherwise use the same notation.
Now for an arbitrary configuration of four hyperplanes $H_i$ in $\C P^3$,
we newly define triple points $p_i$ by $p_i=H_0\cap\dots\cap \widehat{H_i}\cap\dots\cap H_3$, where we omitted $H_i$.

\subsection{Blow-up of an SNC complex surface at finite points in double curves excluding triple points}\label{subsec:bl-up_SNC}
Here we shall prove Proposition $\ref{prop:bl-up_SNC}$,
which will be used in constructing examples of $d$-semistable $K3$ surfaces of Type \rom{3} in later sections.

Consider an SNC complex surface $X'=\smallcup_iX'_i$ with double curves $D'_i=\sum_{j\in I_{i}}D_{ij}'$
and triple points $T'_{ij}=\sum_{k\in I_{ij}}T'_{ijk}$ on each $X'_i$.
Suppose that $X'$ satisfies conditions (ii) and (iii) of Theorem $\ref{thm:smoothing}$.
For a point $p\in D'_{ij}\setminus T'_{ij}$, we blow up $X'_i\cup X'_j$ at $p$ as follows.

We may assume $i=1, j=2$ and choose coordinates $(z_{12,\alpha},w_{12,\alpha})$ on $U_{1,\alpha}$, and $(z_{21,\alpha},w_{21,\alpha})$ on $U_{2,\alpha}$
for $\alpha\in\Lambda_{12}^{(1)}$ such that $(z_{12,\alpha}(p),w_{12,\alpha}(p))=(z_{21,\alpha}(p),w_{21,\alpha}(p))=(0,0)$,
$D'_{12}=\{ w_{12,\alpha}=0\}$ on $U_{1,\alpha}$ and $D'_{21}=\{ w_{21,\alpha}=0\}$ on $U_{2,\alpha}$.
We may also assume that there exists an embedding $\phi_p=\phi_{1,p}\cup\phi_{2,p}$ of $X'_1\cup X'_2$ around $p$ into $\C^3$ given by
\begin{equation*}
\phi_{i,p}:U_{i,\alpha}\subset\C^2_i\xhookrightarrow{~~\iota_i~~}\C^3,\quad\text{where}\quad
\iota_i:(w^0,w^i)\mapsto\begin{dcases}
(w^0,w^1,0)&\text{for }i=1,\\
(w^0,0,w^2)&\text{for }i=2.
\end{dcases}
\end{equation*}
Then the blow-up $\pi_p=\pi_{p,1}\cup\pi_{p,2}:\widetilde{U}_{1,\alpha}\cup\widetilde{U}_{2,\alpha}\to U_{1,\alpha}\cup U_{2,\alpha}$ of
$U_{1,\alpha}\cup U_{2,\alpha}\subset X'_1\cup X'_2$ at $p$ is determined by the commutative diagram
\begin{equation*}
\vcenter{\xymatrix{
\widetilde{U}_{i,\alpha}\ar@{}[r]|-*{\subset}\ar@{->>}[d]_{\pi_{p,i}}\ar@{}[dr]|-*{\commute}
&\widetilde{\C}^2_i\ar@{^{(}->}[r]^-{\widetilde{\iota}_i}\ar@{->>}[d]^{\pi_i}\ar@{}[dr]|-*{\commute}&\widetilde{\C}^3\ar@{->>}[d]^{\pi}\\
U_{i,\alpha}\ar@{}[r]|-*{\subset}&\C^2_i\ar@{^{(}->}[r]^-{\iota_i}&\C^3,
}}
\end{equation*}
where $\pi :\widetilde{\C}^3\to\C^3$ and $\pi_i:\widetilde{\C}^2_i\to\C^2$ for $i=1,2$
are the blow-ups of $\C^3$ and $\C^2_i=\C^2$ at $(0,0,0)$ and $(0,0)$, respectively,
and $\widetilde{\iota}_i:\widetilde{\C}^2_i\to\widetilde{\C}^3$ and $\iota_i:\C^2_i\to\C^3$ are embeddings.
More explicitly, $\widetilde{\C}^3$ and $\widetilde{\C}^2_i$ for $i=1,2$ are given by
\begin{align*}
\widetilde{\C}^3=&\mathrm{Bl}_{(0,0,0)}\C^3\\
=&\set{((w^0,w^1,w^2),[\xi^0,\xi^1,\xi^2])\in\C^3\times\C P^2|w^i\xi^j=w^j\xi^i\text{ for }0\leqslant i<j\leqslant 2}\quad\text{and}\\
\widetilde{\C}^2_i=&\mathrm{Bl}_{(0,0)}\C^2_i
=\set{((w^0,w^i),[\xi^0,\xi^i])\in\C^2_i\times\C P^1_i|w^0\xi^i=w^i\xi^0},
\end{align*}
whereas $\pi :\widetilde{\C}^3\to\C^3$ and $\pi_i:\widetilde{\C}^2_i\to\C^2_i$ are the natural projections
with exceptional divisors $\C P^2$ and $\C P^1_i$, respectively,
and $\widetilde{\iota}_i:\widetilde{\C}^2_i\to\widetilde{\C}^3$ for $i=1,2$ are given by
\begin{equation*}
\widetilde{\iota}_i:((w^0,w^i),[\xi^0,\xi^i])\mapsto (\iota_i(w^0,w^i),[\iota_i(\xi^0,\xi^i)]).
\end{equation*}
Note that the natural projections $\widetilde{\C}^3\to\C P^2$ and $\widetilde{\C}^2\to\C P^1$ yield the universal line bundles
$\mathcal{O}_{\C P^2}(-1)$ and $\mathcal{O}_{\C P^1}(-1)$, respectively.

Now if we define $\widetilde{U}'_i,\widetilde{U}''_i\subset\widetilde{\C}^2_i$ for $i=1,2$ by
\begin{equation}\label{eq:cov_Bl_C2}
\begin{aligned}
\widetilde{U}'_i&=\set{((w^0,w^0\xi^i),[1,\xi^i ])|w^0,\xi^i\in\C}\cong\C^2,\\
\widetilde{U}''_i&=\set{((w^i\xi^0,w^i),[\xi^0,1])|w^i,\xi^0\in\C}\cong\C^2,
\end{aligned}
\end{equation}
then we have an open covering $\widetilde{\C}^2_i=\widetilde{U}'_i\cup\widetilde{U}''_i$.
Thus, defining $\widetilde{U}'_{i,\alpha}$ and $\widetilde{U}''_{i,\alpha}$ by
\begin{equation*}
\widetilde{U}'_{i,\alpha}=\widetilde{U}_{i,\alpha}\cap\widetilde{U}'_i\quad\text{and}\quad
\widetilde{U}''_{i,\alpha}=\widetilde{U}_{i,\alpha}\cap\widetilde{U}''_i,
\end{equation*}
we have $\widetilde{U}_{i,\alpha}=\widetilde{U}'_{i,\alpha}\cup\widetilde{U}''_{i,\alpha}$.
Let $i'=3-i$ for $i=1,2$ as before.
Under the identification of $\widetilde{U}'_i$ and $\widetilde{U}''_i$ with $\C^2$ in $\eqref{eq:cov_Bl_C2}$,
we can use coordinates $(z_{ij,\alpha},\xi_{ij,\alpha})$ on $\widetilde{U}'_{i,\alpha}$ and 
$(w_{ij,\alpha},\zeta_{ij,\alpha})$ on $\widetilde{U}''_{i,\alpha}$, respectively, where $j=i'$.
Letting the proper transform of $D'_{12}$ be $D_{12}$, we see that
\begin{equation*}
D_{12}\cap\widetilde{U}'_{i,\alpha}=\{\xi_{ij,\alpha}=0\}\quad\text{and}\quad
D_{12}\cap\widetilde{U}''_{i,\alpha}=\emptyset\quad\text{for }i=1,2\text{ and }j=i'.
\end{equation*}

Meanwhile, the meromorphic volume form $\Omega_i$ on $U_{i,\alpha}$ lifts to
$\widetilde{\Omega}_i$ on $\widetilde{U}'_{i,\alpha}$ and $\widetilde{U}''_{i,\alpha}$, which is locally represented as
\begin{equation*}
\widetilde{\Omega}_i=\begin{dcases}
-\epsilon_{ij}\der z_{ij,\alpha}\wedge\frac{\der\xi_{ij,\alpha}}{\xi_{ij,\alpha}}&\text{on }\widetilde{U}'_{i,\alpha},\\
-\epsilon_{ij}\der\zeta_{ij,\alpha}\wedge\der w_{ij,\alpha}&\text{on }\widetilde{U}''_{i,\alpha}
\end{dcases}\quad\text{for }i=1,2\text{ and }j=i'.
\end{equation*}
Thus, we see that $\widetilde{\Omega}_i$ has a single pole along $D_{12}$ and yields a trivialization
\begin{equation*}
K_{\widetilde{U}_{i,\alpha}}\otimes [D_{12}\cap\widetilde{U}_{i,\alpha}]\cong\mathcal{O}_{\widetilde{U}_{i,\alpha}},
\end{equation*}
which leads to condition (ii) of Theorem $\ref{thm:smoothing}$.
Also, we have
\begin{equation*}
\residue_{D_{12}\cap\widetilde{U}_{1,\alpha}}\widetilde{\Omega}_1=-\residue_{D_{12}\cap\widetilde{U}_{2,\alpha}}\widetilde{\Omega}_2=\der x_{12,\alpha},
\end{equation*}
which leads to condition (iii) of Theorem $\ref{thm:smoothing}$.

Defining $\widetilde{U}^3_0\subset\widetilde{\C}^3$ by
\begin{equation*}
\widetilde{U}^3_0=\set{((w^0,w^0\xi^1,w^0\xi^2),[1,\xi^1,\xi^2])|w^0,\xi^1,\xi^2\in\C}\cong\C^3,
\end{equation*}
we see that $\iota_i(\widetilde{U}'_i)$ is given by $\{\xi^{i'}=0\}$ in $\widetilde{U}^3_0$,
and thus $\widetilde{U}'_{1,\alpha}\cup\widetilde{U}'_{2,\alpha}$ embeds into $\widetilde{U}^3_0\cong\C^3$ as $\{\xi_{12,\alpha}\xi_{21,\alpha}=0\}$.
Hence, the blow-up of the SNC complex surface $X'_1\cup X'_2$ at $p$ is again an SNC complex surface
around $\pi_p^{-1}(p)\cap D_{12}$ corresponding to $((0,0,0),[1,0,0])\in\widetilde{\C}^3$.

Extending the above local argument to the whole of the blow-up of $X'$, we finally have the following result.
\begin{proposition}\label{prop:bl-up_SNC}
Let $X'=\smallcup_iX'_i$ be an SNC complex surface satisfying conditions $(\mathrm{ii})$ and $(\mathrm{iii})$ of Theorem $\ref{thm:smoothing}$.
Let $X=\smallcup_iX_i$ be the blow-up of $X'$ at finite points in the double curves excluding the triple points.
Then $X$ is also an SNC complex surface satisfying conditions $(\mathrm{ii})$ and $(\mathrm{iii})$ of Theorem $\ref{thm:smoothing}$.
\end{proposition}

\subsection{A $d$-semistable $K3$ surface of Type \rom{3}}\label{subsec:ex_4triple}
Here and hereafter, we will denote tensor products of line bundles by their sums.
Also, we will denote the divisor class $[D]$ simply by $D$.
Throughout this section, indices $i,j$, and $k$ will take $0,1,2$, or $3$, and
if $i$ and $j$ are placed together as subscripts, then we will understand as $i<j$ unless otherwise stated.
\begin{example}\label{ex:4triple}
As we mentioned above, we rewrite the SNC complex surface with triple points in Example $\ref{ex:tetra_in_CP3}$
as $X'=\smallcup_{i=0}^3H_i$ in $\C P^3$ with $L_{ij}=H_i\cap H_j$.
Then the triple points in $X'$ are given by
\begin{align*}
p_0=\{ [1,0,0,0]\} ,\quad p_1&=\{ [0,1,0,0]\} ,\quad p_2=\{ [0,0,1,0]\} ,\quad\text{and}\quad p_3=\{ [0,0,0,1]\} ,
\end{align*}
which are shown in the following figure.
\begin{figure}[H]
\includegraphics[width=0.55\hsize]{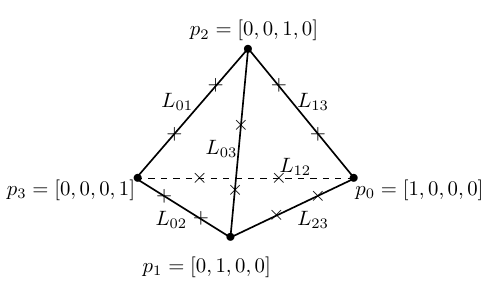}
\caption{Double lines and triple points in $X'=\smallcup_{i=0}^3H_i$}
\label{fig:Tetra}
\end{figure}

To make $X'$ $d$-semistable, we consider the blow-up $X_i$ of $H_i$ at two points except for the triple points in each $L_{ij}$.
According to Remark $2.11$ of \cite{Fr83}, we define $N_{X'}(L_{ij})$ in $\mathrm{Pic}(L_{ij})$ by
\begin{equation}\label{def:CollNormClass}
N_{X'}(L_{ij})=\restrict{L_{ij}}{L_{ij}}+\restrict{L_{ji}}{L_{ij}}+T_{ij}
\end{equation}
for the SNC variety $X'$.
See also equation $(4.1)$ of \cite{Lee19}.
Then let us define the \emph{collective normal class} of $X'$ by
\[
(N_{X'}(L_{ij}))_{0\leqslant i<j\leqslant 3}=
(N_{X'}(L_{01}),N_{X'}(L_{02}),\dots ,N_{X'}(L_{23}))\in\smalloplus_{0\leqslant i<j\leqslant3}\mathrm{Pic}(L_{ij}).
\]
The following description of $d$-semistability is a consequence of $\eqref{def:d-s.s2}$ in Definition $\ref{definition:d-s.s}$.
\begin{lemma} \label{lem:CNC}
An SNC complex surface $X$ is $d$-semistable if and only if the collective normal class of $X$ is trivial.
\end{lemma}
In fact, N.-H. Lee showed that $d$-semistability of $X$ defined in $\eqref{def:d-s.s}$
is equivalent to the triviality of the collective normal class of $X$.
See \cite{Lee19}, Proposition $4.1$ for further details.

Now we return to our example.
In Example $\ref{ex:tetra_in_CP3}$, $(4)$, we saw that $\restrict{L_{ij}}{L_{ij}}\cong\restrict{L_{ji}}{L_{ij}}\cong\mathcal{O}_{\C P^1}(1)$
and $T_{ij}\cong  \mathcal{O}_{\C P^1}(2)$.
Hence, we have $N_{X'}(L_{ij})\cong\mathcal{O}_{L_{ij}}(4)$, and we see that the collective normal class of $X'$ is a divisor class
\[
(\mathcal{O}_{L_{01}}(4),\mathcal{O}_{L_{02}}(4),\dots ,\mathcal{O}_{L_{23}}(4))\in\smalloplus_{0\leqslant i<j\leqslant 3}\mathrm{Pic}(L_{ij}).
\]
We choose two points $p_{ij}$ and $p'_{ij}$ in each $L_{ij}\setminus T_{ij}$ as shown in Figure $\ref{fig:Tetra}$ with symbol $\times$.
Setting $P_{ij}=\{ p_{ij}, p'_{ij}\}$ and $P=\smallcup_{0\leqslant i<j\leqslant 3} P_{ij}$, we take the simultaneous blow-up $\pi$ of $X'$ at $P$:
\[
\pi : X=\mathrm{Bl}_{P}(X')\dasharrow X'\subset\C P^3.
\]
Let $X_i$, $D_{ij}$, and $q_i$ be the proper transforms of irreducible components $H_i$, double lines $L_{ij}$, and triple points $p_i$
under the blow-up $\pi$, respectively.
Let $E$ be the exceptional divisor, which is a disjoint union of $2\# I_i$ copies of $\C P^1$.
Then $X$ can be obtained as another SNC complex surface $X=\smallcup_{i=0}^3X_i$.
Hence, we have the following.

\begin{claim}\label{claim:dss4Triple}
The SNC complex surface $X=\smallcup_{i=0}^3X_i$ is $d$-semistable.
\end{claim}
\begin{proof}
We use the same notation as above.
Then we see that $N_X(D_{ij})$ is linearly equivalent to the divisor class of $P$.
Hence $\eqref{def:CollNormClass}$ yields that
\[
N_{X'}(L_{ij})=\restrict{L_{ij}}{L_{ij}}+\restrict{L_{ji}}{L_{ij}}+p_k+p_\ell \sim P ,
\]
where $k$ and $\ell$ are chosen so that $\{ i,j,k,\ell\} =\{ 0,1,2,3\}$.
Hence, a straightforward computation shows that
\begin{align*}
N_X( D_{ij})&=\restrict{D_{ij}}{D_{ij}}+\restrict{D_{ji}}{D_{ij}}+q_k+q_\ell\\
&\sim\bigl(\restrict{\pi}{D_{ij}} \bigr)^*\bigl(\restrict{L_{ij}}{L_{ij}}-P \bigr)+
\bigl(\restrict{\pi}{D_{ij}}\bigr)^*\bigl(\restrict{L_{ji}}{L_{ij}}\bigr)
+\bigl(\restrict{\pi}{D_{ij}}\bigr)^*p_k+ \bigl(\restrict{\pi}{D_{ij}}\bigr)^*p_\ell\\
&=\bigl(\restrict{\pi}{D_{ij}}\bigr)^*\bigl(\restrict{L_{ij}}{L_{ij}}+\restrict{L_{ji}}{L_{ij}}+p_k+p_\ell -P\bigr)\\
&=\bigl(\restrict{\pi}{D_{ij}}\bigr)^*\bigl( N_{X'}(L_{ij})-P\bigr)\sim\bigl(\restrict{\pi}{D_{ij}}\bigr)^*(0)=0
\end{align*}
for all $0\leqslant i<j\leqslant 3$, which implies $d$-semistability of $X$ by Lemma $\ref{lem:CNC}$.
\end{proof}

By Proposition $\ref{prop:bl-up_SNC}$ and Claim $\ref{claim:dss4Triple}$, we can apply Theorem $\ref{thm:smoothing}$
to obtain a family of smoothings $\varpi :\mathcal{X}\to\Delta$.
Since topology is invariant under continuous deformations, we can find the Betti numbers of
the general fiber $X_\zeta =\varpi^{-1}(\zeta )$ for $\zeta\in\Delta^{\! *}$ from those of the central fiber $X$.
In particular, the Euler characteristic of the general fiber $X_\zeta$ is calculated as
\begin{align}
\begin{split}\label{eq:EulerK3}
\chi(X_\zeta)&=\chi(X)=\chi (X')-\chi(\{ 12\text{ points}\})+\chi(E)\\
&=\smallsum_{i=0}^3\chi (H_i)+12=24,
\end{split}
\end{align}
where $E$ denotes the exceptional divisor of the blow-up $\pi$.
Moreover, one can compute the integral cohomology group of the general fiber from those of the components of the central SNC fiber.
See \cite{LeeThesis}, Chapter \rom{4} and \cite{Lee19}, Proposition $3.2$ for further details.
According to the Enriques-Kodaira classification of compact complex surfaces with trivial canonical bundle, 
the resulting compact complex surface $X_\zeta$ with trivial canonical bundle is a $K3$ surface.
\end{example}

\section{Two hyperplanes and a quadric surface in $\C P^3$}\label{sec:MatchingProb}
We apply the argument in the previous section to more general SNC complex surfaces and will provide a more technical example.
In this case, we encounter the issue that one cannot glue all components together along their intersections
because the intersection parts are not isomorphic in general (see  Section $\ref{subsec:MP}$).
This kind of problem does not happen in the case of the doubling construction \cite{Doi09, DY14}.
However, we will see that there is a good way to handle this sort of mismatch problem by choosing carefully where and in what order we take the blow-ups.

\subsection{Notation}\label{subsec:Notation}
Let $H_1$, $H_2$ be two hyperplanes and $H_3$ be a quadric surface in $\C P^3$.
Assume the union $Y=H_1\cup H_2 \cup H_3$ is an SNC surface.
For an SNC complex surface $Y=H_i\cup H_j \cup H_k$, let us denote $L_{ij}=H_i\cap H_j$ and $T_{ijk}=H_i\cap H_j \cap H_k$, respectively.
For later use, we denote the set of double curves $L_{ij}$ by $C_k$ with $k=\nu_{ij}$,
where $\nu_{ij}\in\{ 1,2,3\}$ is the unique number satisfying $\epsilon_{ij\nu_{ij}}\neq 0$ as in Section \ref{subsec:smoothing_triple}.
Then we find that the collective normal class of $Y$ is a divisor class
\[
(\mathcal{O}_{C_1}(4), \mathcal{O}_{C_2}(4),\mathcal{O}_{C_3}(4) )\in \mathrm{Pic}(C_1)\oplus \mathrm{Pic}(C_2)\oplus \mathrm{Pic}(C_3).
\]

For $i=1,2$, we choose nonsingular points $P_i$ in $\norm{\mathcal{O}_{C_i}(4)}$ consisting of eight distinct points:
$P_i=\{ p_{i,1},p_{i,2},\dots ,p_{i,8}\}$.
Also we choose $P_3$ in the linear system $\norm{\mathcal{O}_{C_3}(4)}$ which consists of four nonsingular points.
Furthermore, we may assume that all $P_i$'s are distinct.
Let $\tau$ be the set of triple points $T_{123}=H_1 \cap H_2 \cap H_3$ where each $H_i$ intersects the rest transversely.
In order to avoid the following \emph{mismatch problem}, we may choose $P_i\in\norm{\mathcal{O}_{C_i}(4)}$ satisfying
\begin{equation}\label{condi:Distinct}
P_i \cap \tau = \emptyset \quad \text{ for all }i\in \{ 1,2,3\} ,
\end{equation} 
that is, $P_i$ and $\tau$ are distinct sets of points for all $i$.

\subsection{The mismatch problem}\label{subsec:MP}
We denote the blow-up of $X$ at $P$ by $\mathrm{Bl}_P(X)$ as in Section \ref{sec:TypeIIIK3}.
Suppose that we symmetrically set
\[
X_1=\mathrm{Bl}_{P_3}(H_1), \quad X_2=\mathrm{Bl}_{P_1}(H_2), \quad \text{ and } \quad X_3=\mathrm{Bl}_{P_2}(H_3)
\]
and take the proper transforms $D_{ij}$ of $L_{ij}=C_k$ in $X_j$.
Then we have to glue together $X_i$ and $X_j$ along their intersections in a suitable way.
Otherwise, if we mistakenly choose $D_{21}$ in $X_1$ and $D_{12}$ in $X_2$, the mismatch problem may occur,
that is, the intersections $D_{12}$, $D_{21}$ which we want to glue together are not isomorphic.
For instance, let $D_{21}$ be the proper transform of $L_{21}$ under the blow-up $X_1=\mathrm{Bl}_{P_3}(H_1) \dasharrow H_1$.
Now we assume that $P_i\cap \tau \neq \emptyset$.
Since the blow-up locus $P_3$ lies on $L_{21}$, we see that $D_{21}\cong L_{21}=C_3$.
However, $P_1$ intersects $C_3$ at $P_1 \cap \tau$ ($1$  or  $2$ points).
This implies that $D_{12}$ is obtained as the blow-up of $L_{12}$ along $P_1\cap\tau$,
namely $D_{12}=\mathrm{Bl}_{P_1\cap\tau}(L_{12})\dasharrow L_{12}$.
Thus, $D_{12}$ cannot be identified with $D_{21}$.
Consequently, we cannot glue together $D_{21}$ in $X_1$ and $D_{12}$ in $X_2$.
In the following subsection, we shall see how to deal with this sort of technical issue.

\subsection{A $d$-semistable SNC complex surface}\label{subsec:dssSNC}
As we saw in the previous section, if we choose the blow-up locus in a symmetric way, the mismatch problem may occur.
Hence, we shall take the blow-up of each component $H_i$ not to be symmetric,
whereas the proper transforms $D_{ij}$ and $D_{ji}$ to be isomorphic.
More precisely, we will construct a $d$-semistable SNC complex surface $X$ in three steps.

\renewcommand{\labelenumi}{\emph{Step }\arabic{enumi}.}
\begin{enumerate}
\item For $\{i,j\}=\{1,2\}$, we take the blow-up $\pi_i$ of $H_i$ at $P_j$ with $P_j\in|\mathcal O_{C_j}(4)|$,
and consider the proper transform $L'_{ji}$ of $L_{ji}$ under the blow-up $\pi_i$.
Then we show that $L_{ji}' \cong L_{ij}'$ in Claim $\ref{claim:Y21}$.
\item We take the blow-up of $H_1'$ at $P_3'$ where $P_3'$ is the proper transform of $P_3$ under $\pi_1$.
Then we obtain $H_1''=\mathrm{Bl}_{P_3'}(H_1')$ which will be a component of an SNC complex surface in the next step.
\item Finally, we construct the desired $d$-semistable SNC complex surface $X=X_1\cup X_2\cup X_3$
with a normalization $\psi:H_1''\cup H_2'\cup H_3\to X$
such that $\psi(H_1'')=X_1$, $\psi(H_2')=X_2$ and $\psi(H_3)=X_3$ in Proposition $\ref{prop:d-s.s}$.
\end{enumerate}
We introduce our setting in this subsection.
In accordance with the previous argument in Section $\ref{subsec:MP}$,
we choose $P_i\in\norm{\mathcal{O}_{C_i}(4)}$ for $i=1,2,3$ satisfying $\eqref{condi:Distinct}$.
Note that the triple locus $\tau$ consists of two points.
Furthermore, we regard $L_{ij}$ as a divisor of $H_j$, whereas we treat $L_{ji}$ as a divisor of $H_i$,
although $L_{ij}$ and $L_{ji}$ are isomorphic to each other.

\setcounter{step}{0}
\begin{step}\label{step:4-1}
For $\{ i,j\} =\{ 1,2\}$, we consider the blow-up $\pi_i: H_i'=\mathrm{Bl}_{P_j}(H_i) \dasharrow H_i$
and take the proper transforms $L_{3i}'$ (resp. $L_{ji}'$) of $L_{3i}$ (resp. $L_{ji}$) under $\pi_i$.
Let $E_j=\pi_i^{-1}(P_j)$ be the exceptional divisor in $H_i'$.
Then we have isomorphisms
\begin{equation}\label{isom:Y3i}
H_3 \supset L_{i3}\cong L_{3i}\cong L_{3i}' \subset H_i'.
\end{equation}
Moreover, we claim the following isomorphism between $L_{21}' \subset H_1'$ and $L_{12}' \subset H_2'$.
\begin{claim}\label{claim:Y21}
$L_{21}'\cong L_{12}'$.
\end{claim}
\begin{proof}
For $\{ i,j\} =\{ 1,2\}$, $L_{ji}'$ is obtained as the blow-up of $L_{ji}$ at $P_i$.
On the other hand, $P_i=\{ 8\text{ points}\}\subset \C P^1$ for each $i=1,2$ and $L_{21}\cong L_{12} =C_3$ imply that
\[
H_1'\supset L'_{21}=\mathrm{Bl}_{P_1}(L_{21})=\mathrm{Bl}_{P_2}(L_{12})=L'_{12}\subset H_2'
\]
as desired.
\end{proof}
Let $P_3'$ be the proper transform of $P_3$ on $C_3 \subset H_1$ under the blow-up $\pi_1: H_1' \dasharrow H_1$.
Setting $\tau'=L_{21}'\cap L_{31}'$, we see that $P_3' \cap \tau'= \emptyset$ by assumption $\eqref{condi:Distinct}$.
Accordingly, $P_3' \cap L_{31}'=\emptyset$, i.e., $P_3'$ does not meet with $L_{31}'$.
\end{step}

\begin{figure}
\includegraphics[width=0.35\hsize]{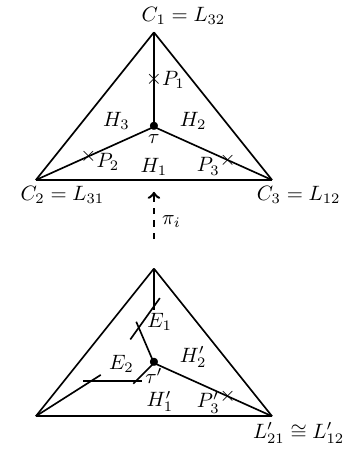}
\caption{The Blow-up \,$\pi_i$}
\end{figure}

\begin{step}\label{step:4-2}
Next we take the blow-up of $H_1'$ at $P_3'$
\[
\pi_1': H_1''=\mathrm{Bl}_{P_3'}(H_1') \dasharrow H_1'
\]
and consider the proper transform $E_2'$ of $E_2$ under $\pi_1'$.
Let $E_3'=\pi_1'^{-1}(P_3')$ be the exceptional divisor.
Since $P_3'\notin L_{31}'=C_2'$, the blow-up $\pi'_1$ does not change $L_{31}'$.
Hence, the proper transform $L_{31}''$ of $L_{31}'$ in $H_1'$ is isomorphic to $L_{31}'$, namely
\begin{equation}\label{isom:Y31}
H_1''\supset L_{31}''\cong L_{31}'.
\end{equation}
Meanwhile, $P_3'=\{ 4\text{ points}\}\subset L_{21}'$ and Claim $\ref{claim:Y21}$ guarantees that there are isomorphisms
\begin{equation}\label{isom:Y21}
L_{21}''\cong L_{21}'\cong L_{12}'. 
\end{equation}
\end{step}

\begin{figure}
\includegraphics[width=0.35\hsize]{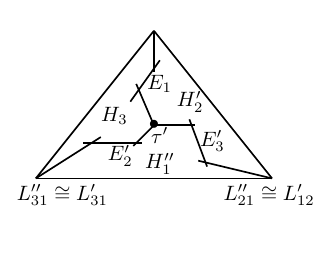}
\vspace{-\baselineskip}
\caption{The SNC complex surface \,$X$}
\end{figure}

\begin{step}
Now we construct an SNC complex surface by gluing $H_1''$, $H_2'$ and $H_3$ together along their intersections.
As a consequence of $\eqref{isom:Y3i}$, $\eqref{isom:Y31}$ and $\eqref{isom:Y21}$ we see that
\begin{equation}\label{isom:DoubleCurve}
L_{21}''\cong L_{12}',\quad L_{32}'\cong L_{23}, \quad \text{and} \quad L_{31}''\cong L_{31}'\cong L_{13}.
\end{equation} 
Eventually, we obtain the following induced isomorphisms
\begin{equation}\label{isom:TripleLoci}
L_{21}''\cap L_{31}''\cong L_{12}'\cap L_{13}\cong L_{12}'\cap L_{23}\cong L_{12}'\cap L_{32}'\cong L_{13}\cap L_{23}\, ~(\cong\tau' ).
\end{equation}
We use $\eqref{isom:DoubleCurve}$ and $\eqref{isom:TripleLoci}$ for gluing all components $H_1''$, $H_2'$ and $H_3$ together.
For example, we can glue $H_1''$ and $H_2'$ together by using $L_{21}''\cong L_{12}'$,
and further we need to consider the isomorphism $L_{21}''\cap L_{31}''\cong \tau'$ in $H_1''$ because there are three components.
Then one can construct an SNC complex surface $X=X_1\cup X_2\cup X_3$ with a normalization $\psi :H_1''\cup H_2'\cup H_3\rightarrow X$ such that
$\psi (H_1'')=X_1$, $\psi (H_2')=X_2$ and $\psi (H_3)=X_3$.

Setting $D_{(i)}=D_{jk}=X_j\cap X_k$ for $\{ i,j,k\} =\{ 1,2,3\}$, we show the following result:
\begin{proposition}\label{prop:d-s.s}
$X$ is $d$-semistable.
\end{proposition}
The rest of this subsection is devoted to prove Proposition $\ref{prop:d-s.s}$.
In the light of Lemma $\ref{lem:CNC}$, Proposition $\ref{prop:d-s.s}$ is an immediate consequence of the following.
\begin{claim}\label{claim:CNC}
Let $P_i\in \norm{\mathcal{O}_{C_i}(4)}$ be nonsingular points as in Section $\ref{subsec:Notation}$.
Then $\{ P_1,P_2,P_3\}$ determines a collective normal class.
Moreover, $X$ has trivial collective normal class.
\end{claim}
\begin{proof}[Proof of Claim $\ref{claim:CNC}$]
For the first part of the statement, we saw already in Section $\ref{subsec:Notation}$.
For the proof of the second part, it suffices to show that
\begin{center}
\begin{tabular}{lclcl}
(i)~ $N_X(D_{(1)})=0$, & \qquad \quad (ii)~ $N_X(D_{(2)})=0$, & \qquad \quad (iii)~ $N_X(D_{(3)})=0$ 
\end{tabular}
\end{center}
for our purpose.
Recalling $\pi_i: H_i'=\mathrm{Bl}_{P_j}(H_i)\dasharrow H_i$ for $\{ i,j\}=\{ 1,2\}$ in Step $1$,
and the definition of the normal bundle for the SNC complex surface $Y$, we see that $N_Y(L_{i3})$ is linearly equivalent to $P_j$:
\begin{equation}\label{eq:linequiv}
N_Y(L_{i3})=\restrict{L_{i3}}{L_{i3}}+\restrict{L_{3i}}{L_{i3}}+T_{ji3}\sim P_j.
\end{equation}

\noindent (i)~\, Let $X_{123}$ be $X_1\cap X_2 \cap X_3$.
For $i=2$, $j=1$, $\eqref{eq:linequiv}$ implies that
\begin{align*}
N_X(D_{(1)})&=\restrict{D_{23}}{D_{23}}+\restrict{D_{32}}{D_{23}}+X_{123}\\
&\sim\bigl(\restrict{\pi_2}{D_{23}}\bigr)^*\bigl(\restrict{L_{23}}{L_{23}}\bigr)
+\bigl(\restrict{\pi_2}{D_{23}}\bigr)^*\bigl(\restrict{L_{32}}{L_{23}}-P_1\bigr)
+\bigl(\restrict{\pi_2}{D_{23}}\bigr)^*\bigl(T_{123}\bigr)\\
&=\bigl(\restrict{\pi_2}{D_{23}}\bigr)^*\bigl(\restrict{L_{23}}{L_{23}}+\restrict{L_{32}}{L_{23}}+T_{123}-P_1\bigr)\\
&=\bigl(\restrict{\pi_2}{D_{23}}\bigr)^*\bigl( N_Y(L_{23})-P_1 \bigr)\sim\bigl(\restrict{\pi_2}{D_{23}}\bigr)^*(0)=0.
\end{align*}
Thus, we have $N_X(D_{(1)})=0$.
Analogously, one can show that (ii)\, $N_X(D_{(2)})=0$.\\

\noindent (iii)~\, Finally, we show that $N_X(D_{(3)})=0$.
Note that in this case we shall consider
\[
\pi'_1:H_1''=\mathrm{Bl}_{P_3'}(H_1')\dasharrow H_1'\quad \text{with}\quad P_3'\subset L_{12}'
\]
as we mentioned in Step $2$.
Therefore, we have $N_{Y'} (L_{12}')\sim P_3'$ for $Y'=H_1'\cup H_2'\cup H_3$.
Let $Y_{312}'$ be $H_3\cap H_1'\cap H_2'$.
Then we have
\begin{align*}
N_X(D_{(3)})&=\restrict{D_{12}}{D_{12}}+\restrict{D_{21}}{D_{12}}+X_{312}\\
&\sim \bigl(\restrict{\pi_1'}{D_{12}}\bigr)^*\bigl(\restrict{L_{12}'}{L_{12}'}\bigr)
+\bigl(\restrict{\pi_1'}{D_{12}}\bigr)^*\bigl(\restrict{L'_{21}}{L'_{12}}-P_3'\bigr)
+\bigl(\restrict{\pi_1'}{D_{12}}\bigr)^*\bigl(Y_{312}'\bigr)\\
&=\bigl(\restrict{\pi_1'}{D_{12}}\bigr)^*\bigl(\restrict{L_{12}'}{L_{12}'}+\restrict{L'_{21}}{L'_{12}}+Y_{312}'-P_3'\bigr)\\
&=\bigl(\restrict{\pi_1'}{D_{12}} \bigr)^*\bigl( N_{Y'}(L'_{12})-P_3'\bigr)\sim\bigl(\restrict{\pi'_1}{D_{12}}\bigr)^*(0)=0.
\end{align*}
This completes the proof of the claim.
\end{proof}
\end{step}

\subsection{Computation of the Euler characteristic}\label{subsec:Euler}
Applying Theorem $\ref{thm:smoothing}$ to $X$, we obtain a family of smoothings $\varpi :\mathcal{X}\rightarrow\Delta$ of $X$
whose general fibers $X_\zeta=\varpi^{-1}(\zeta)$ are compact complex surfaces with trivial canonical bundle.
In fact, we will show the following.
\begin{proposition}
$X$ is a $d$-semistable $K3$ surface of Type \rom{3}, that is, the Euler characteristic of $X_\zeta$ is $24$.
\end{proposition}  
\begin{proof} 
We remark that the Euler characteristic of $X_\zeta$ is given by
\begin{equation}\label{eq:EulerNumber}
\chi (X_\zeta)=\smallsum_{i=1}^3\chi (X_i)-2\smallsum_{i<j}\chi (D_{ij})+3\chi(X_{123}).
\end{equation}
(See Proposition $3.2$ in \cite{Lee19}.)
By taking a sequence of blow-ups in Step $\ref{step:4-1}$ and Step $\ref{step:4-2}$, we see that
\begin{align*}
\chi (X_1)&=\chi (H''_1)=\chi (H_1')-\chi (P_3')+\chi (E_3')=\chi (H_1')+\chi (P_3')\\
&=\chi (H_1)-\chi (P_2)+\chi (E_2)+\chi (P_3')\\
&=\chi (H_1)+\chi (P_2)+\chi (P_3')=\chi (H_1)+\chi (P_2)+\chi (P_3),\\
\chi (X_2)&=\chi (H'_2)=\chi (H_2)-\chi (P_1)+\chi (E_1)=\chi (H_2)+\chi (P_1),\quad\chi (X_3)=\chi (H_3),\\
\chi (D_{12})&=\chi (L'_{12})=\chi (L_{12}),\quad\chi (D_{23})=\chi (L_{23}),\quad\chi (D_{13})=\chi (L_{13}),
\quad\text{and}\\
\chi (X_{123})&=\chi (T_{123})=\chi (\tau )=2.
\end{align*}
Plugging these results into $\eqref{eq:EulerNumber}$, we find that
\begin{equation}\label{eq:EulerNumber2}
\chi (X_\zeta)=\smallsum_{i=1}^3\chi (H_i)-2\smallsum_{i<j}\chi (L_{ij})+3\chi (\tau )+\smallsum_{i=1}^3\chi (P_i).
\end{equation} 
Hence, one can reduce the problem to the computation of the Euler characteristic for each piece.
For $i=1,2$, we see that $\chi(H_i)=3$.
In the same manner as \cite{DY19}, p. $375$, we find the Hodge numbers of $H_3$ using Theorem $4.3$ in \cite{DY19}.
Then we see that
\[
h^{p,q}(H_3)=\begin{array}{ccccc}
&& 1 && \\
& 0 && 0 & \\
0 && 2 && 0 \\
& 0 && 0 & \\
&& 1 && \\
\end{array},\quad \text{so that}\quad\chi (H_3)=4.
\]
Similarly, we have
\begin{gather*}
\chi (L_{13})=\chi (L_{23})=2,\quad\chi (L_{12})=\chi (\C P^1)=2,\\
\chi (\tau )=2,\quad\chi (P_1)=\chi (P_2)=8,\quad\chi (P_3)=4.
\end{gather*}
Consequently, $\eqref{eq:EulerNumber2}$ yields that
\[
\chi (X_\zeta)=(3+3+4)-2(2+2+2)+3\cdot 2+(8+8+4)=24.
\]
Hence, the assertion is verified.
\end{proof}



\begin{thebibliography}{99}

\bibitem[BHPV]{BHPV}
W. Barth, K. Hulek, C. Peters and A. Van de Ven,
{\em Compact complex surfaces. Second edition}, A Series of Modern Surveys in Mathematics, 4. Springer-Verlag, Berlin, $2004$.

\bibitem[C11]{Cl11}
R. Clancy, {\em New examples of compact manifolds with holonomy $\Spin$},
Ann. Glob. Anal. Geom. {\bf{40}} $(2011)$, $203$--$222$.

\bibitem[CLM19]{CLM19}
K. Chan, N. C. Leung, Z. N. Ma,
Geometry of the Maurer-Cartan equation near degenerate Calabi-Yau varieties, arXiv:$1902$.$11174$v5, to appear in J. Differential Geom.

\bibitem[D09]{Doi09}
M. Doi,   
\newblock{\em Gluing construction of compact complex surfaces with trivial canonical bundle},
\newblock J. Math. Soc. Japan
\newblock {\bf{61}} $(2009)$, $853$--$884$.

\bibitem[DY14]{DY14}
M. Doi and N. Yotsutani,
\newblock{\em Doubling construction of Calabi-Yau threefolds},
\newblock New York J. Math.
\newblock {\bf{20}}, $(2014)$, $1203$--$1235$.

\bibitem[DY15]{DY15}
M. Doi and N. Yotsutani,
\newblock{\em Doubling construction of Calabi-Yau fourfolds from toric Fano fourfolds},
\newblock Commun. Math. Stat.
\newblock {\bf{3}}, $(2015)$, $423$--$447$.

\bibitem[DY19]{DY19}
M. Doi and N. Yotsutani,
\newblock{\em Gluing construction of compact Spin(7)-manifolds},
\newblock J. Math. Soc. Japan
\newblock {\bf{71}} $(2019)$, $349$--$382$.

\bibitem[FFR19]{FFR19}
S. Felten, M. Filip and H. Ruddat,
\newblock{\em Smoothing toroidal crossing spaces},
\newblock arXiv:$1908$.$11235$, to appear in Forum of Mathematics Pi.

\bibitem[FM]{FM83}R. Friedmann, and D. Morrison (Eds.),
{\em The Birational Geometry of Degenerations},
Progress in Mathematics (PM, volume $29$) Birkh\"{a}user Boston $1983$.

\bibitem[Fr83]{Fr83}
R. Friedman,
\newblock {\em Global smoothings of varieties with normal crossings},
\newblock Ann. Math. {\bf{118}} $(1983)$, $75$--$114$.

\bibitem[Fu21]{Fuj21}K. Fujita, private communication, 2021.

\bibitem[Go04]{Goto04}
R. Goto,   
\newblock{\em Moduli spaces of topological calibrations, Calabi-Yau, hyperK\"{a}hler, $G_2$ and $\mathrm{Spin} (7)$ structures},
\newblock Internat. J. Math.,
\newblock {\bf{15}},  $(2004)$, $211$--$257$.

\bibitem[GH]{GH94}
P. Griffiths and J. Harris,
\newblock {\em Principles of algebraic geometry},
\newblock Wiley Classics Library. John Wiley and Sons, Inc., New York, $1994$. xiv+$813$ pp.

\bibitem[GP]{GP}
V. Guillemin and A. Pollack, 
\newblock {\em Differential Topology},
\newblock Prentice-Hall, New Jersey, $1974$, xvi+$222$ pp.

\bibitem[HS19]{HS19}
K. Hashimoto and T. Sano,
\newblock{\em Examples of non-K\"{a}hler Calabi-Yau $3$-folds with arbitrarily large $b_2$},
\newblock arXiv:$1902$.$01027$v2, to appear in Geom. Topol.

\bibitem[H]{Huybrechts}
D. Huybrechts,
\newblock {\em Lectures on $K3$ surfaces},
\newblock Cambridge Studies in Advanced Mathematics, {\bf{158}}. Cambridge University Press, Cambridge, $2016$. xi+$485$ pp.

\bibitem[J]{Joyce}
D. D. Joyce,
\newblock {\em Compact manifolds with special holonomy},
\newblock Oxford Mathematical Monographs. Oxford University Press, Oxford, $2000$. xii+$436$ pp.


\bibitem[KN94]{KN94}
Y. Kawamata and Y. Namikawa,
\newblock {\em Logarithmic deformations of normal crossing varieties and smoothing of degenerate Calabi-Yau varieties},
\newblock Invent. Math.
\newblock {\bf{118}} ($1994$), $395$--$409$.

\bibitem[Ko03]{Kov03}
A. Kovalev,
\newblock{\em Twisted connected sums and special Riemannian holonomy},
\newblock J. Reine Angew. Math.
\newblock {\bf{565}} $(2003)$, $125$--$160$.

\bibitem[Ku77]{Ku77}
V. Kulikov,
\newblock {\em Degenerations of $K3$ surfaces and Enriques surfaces},
\newblock Izv. Akad. Nauk SSSR Ser. Mat. {\bf{41}} ($1977$), no. $5$, $1008$--$1042$, $1199$.

\bibitem[L06]{LeeThesis}
N.-H. Lee,
\newblock{\em Constructive Calabi-Yau manifolds},
\newblock Ph. D thesis, University of Michigan, $2006$.

\bibitem[L19]{Lee19}
N.-H. Lee,
\newblock{\em d-semistable Calabi-Yau threefolds of Type \rom{3}},
\newblock Manuscripta Math.
\newblock {\bf{161}} $(2020)$, $257$--$281$.

\bibitem[N88]{Ni88}
K. Nishiguchi,
\newblock {\em Degeneration of $K3$ surfaces},
\newblock J. Math. Kyoto Univ.
\newblock {\bf{28}} $(1988)$, $267$--$300$.

\bibitem[Pe77]{P77}
U. Persson, 
\newblock {\em On degenerations of Algebraic Surfaces}, 
\newblock Memoirs A.M.S. {\bf{189}} $(1977)$, xv+$144$ pp.

\bibitem[RS20]{RS20}
H. Ruddat and B. Siebert,
\newblock {\em Period integrals from wall structures via tropical cycles, canonical coordinates in mirror symmetry and analyticity of toric degenerations},
\newblock Publ. Math. IHES {\bf{132}} $(2020)$, $1$--$82$.

\end{thebibliography}
\end{document}